\documentclass[aop,preprint]{imsart}

\usepackage{amsmath,amsthm,amssymb,mathtools}
\usepackage{multicol,hhline,comment,graphicx,url}%,color,multirow
\usepackage{mathrsfs}	% for the \mathscr font, currently for alpha-diversity
\usepackage{accents} % setting up a left-arrow accent
\usepackage{hyperref}
\usepackage{cellspace}		% make vertical spacing in tabular look nice
\usepackage{enumitem} % roman refs to list items
\setlist[enumerate,1]{leftmargin=1cm}

\usepackage{accents,stackengine,scalerel}
\theoremstyle{plain}

\newtheorem{theorem}{Theorem}[section]

\newtheorem{proposition}[theorem]{Proposition}
\newtheorem{corollary}[theorem]{Corollary}
\newtheorem{lemma}[theorem]{Lemma}

\theoremstyle{definition}
\newtheorem{definition}[theorem]{Definition}

\theoremstyle{remark}

\newcommand{\ff}{\mathbf{f}}

\makeatletter
 \DeclareRobustCommand{\checkarg}{\@ifnextchar[{\@witharg}{}}
 \DeclareRobustCommand{\@witharg}[1][]{\ensuremath{\left(#1\right)}}
 \DeclareRobustCommand{\scaleGen}[1]{\@ifnextchar[{\@scalewithargs{#1}}{\odot^{}_{#1}}}
 \def\@scalewithargs#1[#2][#3]{#2 \odot^{}_{#1} #3}
\makeatother

\def\IPspace{\mathcal{I}}
\def\dI{d_{\IPspace}}
\def\HIPspace{\IPspace_H}

\def\dH{d_H}

\def\Exc{\mathcal{E}}

\def\SExc{\Sigma(\Exc)}

\def\cExc{\SExc}
\def\mBxc{\nu_{\textnormal{BES}}}	% Ito measure for BESQ(-1) exc

\def\cC{\mathcal{C}}	% continuous functions
\def\ScS{\Sigma(\cC)}
\def\cCRI{\cC([0,\infty),\IPspace)}

\def\cD{\mathcal{D}}

\def\DS{\cD_{\textnormal{stb}}}

\def\SDS{\Sigma(\DS)}

\def\DSxc{\cD_{\textnormal{exc}}}
\def\SDSxc{\Sigma(\DSxc)}

\def\ScDSxc{\SDSxc}
\def\mSxc{\nu_{\textnormal{stb}}}	% Ito measure for Stable(1+alpha) exc
\def\H{\mathcal{N}^{\textnormal{\,sp}}}
\def\SH{\Sigma(\H)}
\def\cH{\Sigma(\H)}
\def\Hfin{\H_{\textnormal{fin}}}
\def\SHfin{\Sigma(\Hfin)}
\def\cHfin{\Sigma(\Hfin)}
\def\Hxc#1{\H_{#1\textnormal{cld}}}
\def\SHxc#1{\Sigma\left(\Hxc{#1}\right)}
\def\cHxc#1{\SHxc{#1}}

\def\mClade{\nu_{\textnormal{cld}}}	% Ito measure for bi-clades

\def\Hs{\mathcal{N}^{\textnormal{\,sp},*}}

\def\Hfins{\Hs_{\textnormal{fin}}}
\def\SHfins{\Sigma(\Hfins)}

\def\cN{\mathcal{N}}

\def\cNRHf{\mathcal{N}\big([0,\infty)\times\Hfin\big)}

\def\cNRE{\mathcal{N}\big([0,\infty)\times\Exc\big)}
\def\ScNRE{\Sigma\big(\cNRE\big)}
\def\cNS{\mathcal{N}(\cS)}

\def\cS{\mathcal{S}}
\def\ScS{\Sigma(\mathcal{S})}
\def\cT{\mathcal{T}}
\def\ScT{\Sigma(\mathcal{T})}
\def\len{\textnormal{len}}			% duration of an excursion of the JCCP / corresponding counting measure
\def\life{\zeta}					% lifetime of a spindle / total height of JCCP corresponding to a clade
\def\dis{\textnormal{dis}}			% distortion
\def\IPmag#1{\left\|\vphantom{I}#1\right\|}		% total mass of an IP

\def\skewer{\textsc{skewer}}		% skewer map
\def\skewerP{\widebar{\skewer}}		% skewer map
\def\cutoffL#1#2{\textsc{cutoff}^{\leq}\ensuremath{\left(#1,#2\right)}}
\def\cutoffG#1#2{\textsc{cutoff}^{\geq}\ensuremath{\left(#1,#2\right)}}

\def\Dirac#1{\delta\left( #1 \right)}
\def\DiracBig#1{\delta\big( #1 \big)}	% the \hat accent leads to \left(\right) being huge

\def\reverse{\mathcal{R}}			% time-reversal operators
\def\reverseexc{\reverse_{\textnormal{spdl}}}		% time-reversal for BESQ
\def\reverseincr{\reverse_{\textnormal{stb}}}		% time-reversal for stable
\def\reverseH{\reverse_{\textnormal{cld}}}			% time-reversal for clades
\def\scaleB{\scaleGen{\textnormal{spdl}}}		% scaling for BESQ
\def\scaleS{\scaleGen{\textnormal{stb}}}		% scaling for Stable(3/2)
\def\scaleH{\scaleGen{\textnormal{cld}}}		% scaling for clades

\def\scaleI{\scaleGen{\textnormal{IP}}}	% scaling for IPs

\def\ShiftRestrict#1#2{#1\big|^{\leftarrow}_{#2}} % shifted to be supported on [0, b-a]
\def\shiftrestrict#1#2{#1|^{\leftarrow}_{#2}}
\def\Restrict#1#2{#1\big|_{#2}}
\def\restrict#1#2{#1|_{#2}}
\def\Concat{ \mathop{ \raisebox{-2pt}{\Huge$\star$} } }
\def\ConcatIL{ \mbox{\huge $\star$} }
\def\concat{\star}
\def\bN{\mathbf{N}}			% (P)RM of spindles
\def\bF{\mathbf{F}}			% (P)RM of clades
\def\bG{\mathbf{G}}			% (P)RM of Stable(3/2) excursions
\def\bX{\mathbf{X}}			% JCCP associated with \bN

\newcommand{\IPLT}{\mathscr{D}}

\def\bM{\mathbf{M}}			% (P)RM of jumps of \bX
\def\bff{\mathbf{f}}		% a special spindle

\newcommand{\td}[1]{\widetilde{#1}}
\def\tdN{\widetilde{\mathbf{N}}}
\def\tdF{\widetilde{\mathbf{F}}}
\def\tdX{\widetilde{\bX}}

\def\tdl{\widetilde{\ell}}

\newcommand{\wh}[1]{\widehat{#1}}
\def\whN{\widehat{\mathbf{N}}}
\def\whF{\widehat{\mathbf{F}}}

\newcommand{\ol}[1]{\widebar{#1}}
\def\olN{\widebar{\mathbf{N}}}
\def\olF{\widebar{\mathbf{F}}}
\def\olX{\widebar{\bX}}

\def\oll{\widebar{\ell}}

\def\BR{\mathbb{R}}				% reals
\def\BN{\mathbb{N}}				% natural numbers
\def\Leb{\textnormal{Leb}}		% Lebesgue measure

\def\to{\rightarrow}
\def\downto{\downarrow}

\def\cf{\mathbf{1}}				% characteristic fctn
\def\Pr{\mathbf{P}}				% probability
\def\BPr{\mathbb{P}}			% probability
\def\EV{\mathbf{E}}				% expected value
\def\cF{\mathcal{F}}			% sigma-alg
\def\cFI{\cF_{\IPspace}}		% sigma-alg on space of cts IPPs
\def\cG{\mathcal{G}}			% alternative sigma-alg
\def\cH{\mathcal{H}}			% alternative sigma-alg
\def\indep{ \mathop{\perp\hspace{-5.5pt}\perp} }
\def\cA{\mathcal{A}}	% sometimes an index set
\def\distribfont#1{\texttt{\upshape #1}}%\textnormal{#1}}}
\def\ExpDist{\distribfont{Exponential}\checkarg}

\def\InvGammaDist{\distribfont{InverseGamma}\checkarg}

\def\PRM{\distribfont{PRM}\checkarg}

\def\Stable{\distribfont{Stable}\checkarg}
\def\StableA{\ensuremath{\distribfont{Stable}(1+\alpha)}}
\def\BESQ{\distribfont{BESQ}\checkarg}

\def\cadlag{c\`adl\`ag}

\stackMath
\makeatletter
\let\save@mathaccent\mathaccent
\newcommand*\if@single[3]{%
  \setbox0\hbox{${\mathaccent"0362{#1}}^H$}%
  \setbox2\hbox{${\mathaccent"0362{\kern0pt#1}}^H$}%
  \ifdim\ht0=\ht2 #3\else #2\fi
  }
\newcommand*\rel@kern[1]{\kern#1\dimexpr\macc@kerna}
\newcommand{\widebar}{}% initialize
\DeclareRobustCommand*\widebar[1]{\@ifnextchar^{\wide@bar{#1}{0}}{\wide@bar{#1}{1}}}
\newcommand*\wide@bar[2]{\if@single{#1}{\wide@bar@{#1}{#2}{1}}{\wide@bar@{#1}{#2}{2}}}
\newcommand*\wide@bar@[3]{%
  \begingroup
  \def\mathaccent##1##2{%
    \let\mathaccent\save@mathaccent
    \if#32 \let\macc@nucleus\first@char \fi
    \setbox\z@\hbox{$\macc@style{\macc@nucleus}_{}$}%
    \setbox\tw@\hbox{$\macc@style{\macc@nucleus}{}_{}$}%
    \dimen@\wd\tw@
    \advance\dimen@-\wd\z@
    \divide\dimen@ 3
    \@tempdima\wd\tw@
    \advance\@tempdima-\scriptspace
    \divide\@tempdima 10
    \advance\dimen@-\@tempdima
    \ifdim\dimen@>\z@ \dimen@0pt\fi
    \rel@kern{0.6}\kern-\dimen@
    \if#31
      \overline{\rel@kern{-0.6}\kern\dimen@\macc@nucleus\rel@kern{0.4}\kern\dimen@}%
      \advance\dimen@0.4\dimexpr\macc@kerna
      \let\final@kern#2%
      \ifdim\dimen@<\z@ \let\final@kern1\fi
      \if\final@kern1 \kern-\dimen@\fi
    \else
      \overline{\rel@kern{-0.6}\kern\dimen@#1}%
    \fi
  }%
  \macc@depth\@ne
  \let\math@bgroup\@empty \let\math@egroup\macc@set@skewchar
  \mathsurround\z@ \frozen@everymath{\mathgroup\macc@group\relax}%
  \macc@set@skewchar\relax
  \let\mathaccentV\macc@nested@a
  \if#31
    \macc@nested@a\relax111{#1}%
  \else
    \def\gobble@till@marker##1\endmarker{}%
    \futurelet\first@char\gobble@till@marker#1\endmarker
    \ifcat\noexpand\first@char A\else
      \def\first@char{}%
    \fi
    \macc@nested@a\relax111{\first@char}%
  \fi
  \endgroup
}
\makeatother

\renewcommand{\mBxc}{\nu}
\DeclareMathSymbol{\minus}{\mathbin}{AMSa}{"39}

\numberwithin{equation}{section}
\numberwithin{figure}{section}
\numberwithin{table}{section}

\setattribute{keyword}{KWD}{Keywords:}

\allowdisplaybreaks[3]

\begin{document}

\begin{frontmatter}
 
 \title{Diffusions on a space of interval partitions:\\ construction from marked L\'evy processes\thanksref{T0}}
 
 \runtitle{Construction of interval partition diffusions}
 \runauthor{Forman, Pal, Rizzolo and Winkel}
 \thankstext{T0}{This research is partially supported by NSF grants {DMS-1204840, DMS-1308340, DMS-1612483, DMS-1855568}, UW-RRF grant A112251, EPSRC grant EP/K029797/1}
 
 \begin{aug}
  \author{\fnms{Noah} \snm{Forman}\thanksref{m4}\ead[label=e1]{noah.forman@gmail.com}},
  \author{\fnms{Soumik} \snm{Pal}\thanksref{m2}\ead[label=e2]{soumikpal@gmail.com}},
  \author{\fnms{Douglas} \snm{Rizzolo}\thanksref{m3}\ead[label=e3]{drizzolo@udel.edu}},\\ and 
  \author{\fnms{Matthias} \snm{Winkel}\thanksref{m1}\ead[label=e4]{winkel@stats.ox.ac.uk}}
  
  \affiliation{McMaster University\thanksmark{m4}, University of Washington\thanksmark{m2},\\ University of Delaware\thanksmark{m3}, University of Oxford\,\thanksmark{m1}}
  
  \address{Department of Mathematics \& Statistics\\ McMaster University\\ 1280 Main Street West\\ Hamilton, Ontario L8S 4K1\\ Canada\\
   \printead{e1}
  }
  
  \address{Department of Mathematics \\ University of Washington\\ Seattle WA 98195\\ USA\\
   \printead{e2}
  }
  
  \address{Department of Mathematical Sciences\\ University of Delaware\\ Newark DE 19716\\ USA\\
   \printead{e3}
  }
  
  \address{Department of Statistics\\ University of Oxford\\ 24--29 St Giles'\\ Oxford OX1 3LB\\ UK\\
   \printead{e4}
  }
 \end{aug} 
  
\begin{abstract} Consider a spectrally positive \StableA\ process whose jumps we interpret as lifetimes of individuals. We mark the jumps 
  by continuous excursions assigning ``sizes'' varying during the lifetime. As for Crump--Mode--Jagers processes (with ``characteristics''), we 
  consider for each level the collection of individuals alive. We arrange their ``sizes'' at the crossing height from left to right 
  to form an interval partition. We study the continuity and Markov properties of the interval-partition-valued process indexed by level. From the 
  perspective of the \StableA\ process, this yields new theorems of Ray--Knight-type. From the perspective of branching processes, this  
  yields new, self-similar models with dense sets of birth and death times of (mostly short-lived) individuals.
  This paper feeds into projects resolving conjectures by Feng and Sun (2010) on the existence of certain measure-valued diffusions with 
  Poisson--Dirichlet stationary laws, and by Aldous (1999) on the existence of a continuum-tree-valued diffusion.
\end{abstract}  
  
 

\end{frontmatter}

\section{Introduction}
\label{sec:intro}

%Consider a population in which each individual has a lifetime during which it can give birth to children. Suppose further that each individual has a ``size'' that evolves during its lifetime. If the individuals give birth to single children at the times of a homogeneous Poisson point process and satisfy natural independence assumptions, this is known as a binary homogeneous Crump--Mode--Jagers model. Quantities of interest include the number of individuals evolving with time, and the sum of ``sizes'' of individuals alive. 

We define interval partitions, following Aldous \cite[Section 17]{AldousExch} and Pitman \cite[Chapter 4]{CSP}.

\begin{definition}\label{def:IP_1}
 An \emph{interval partition} is a set $\beta$ of disjoint, open subintervals of some interval $[0,M]$, that cover $[0,M]$ up to a Lebesgue-null set. We write 
 $\IPmag{\beta}$ to denote $M$. We refer to the elements of an interval partition as its \emph{blocks}. The Lebesgue measure of a block is called its \emph{width} or \emph{mass}.
\end{definition}

An interval partition represents a totally ordered, summable collection of positive real numbers, for example, the interval partition generated naturally by the range of a subordinator (see Pitman and Yor \cite{PitmYor92}), or the partition of $[0,1]$ given by the complement of the zero-set of a Brownian bridge (Gnedin and Pitman \cite[Example 3]{GnedPitm05}). They also arise from the so-called stick-breaking schemes; see \cite[Example 2]{GnedPitm05}. More generally, interval partitions occur as limits of \textit{compositions} of natural numbers $n$, i.e.\ sequences of positive integers with sum $n$. Interval partitions serve as extremal points in paintbox representations of composition structures on $\mathbb{N}$; see Gnedin \cite{Gnedin97}. 

%% THIS PARAGRAPH

We will introduce a construction of diffusion processes on spaces of interval partitions. We focus on self-similar processes with the branching property that blocks evolve independently and each give birth at a constant rate. Our framework enables the construction of continuum analogues to natural up-down Markov chains on discrete partitions based upon the Chinese Restaurant Processes \cite{CSP,PitmWink09,RogeWink19}. These relate to members of the canonical two-parameter family of Poisson--Dirichlet distributions. On partitions with blocks ordered by decreasing mass, related diffusions have been introduced by Ethier and Kurtz \cite{EthiKurt81} and, more recently, by Petrov \cite{Petrov09}. Other known processes of interval partitions such as Bertoin's \cite{Bertoin02} are not path-continuous.

%%

% LEAD WITH Figure, paragraph explaining
\begin{figure}[t]
 \centering
 \input{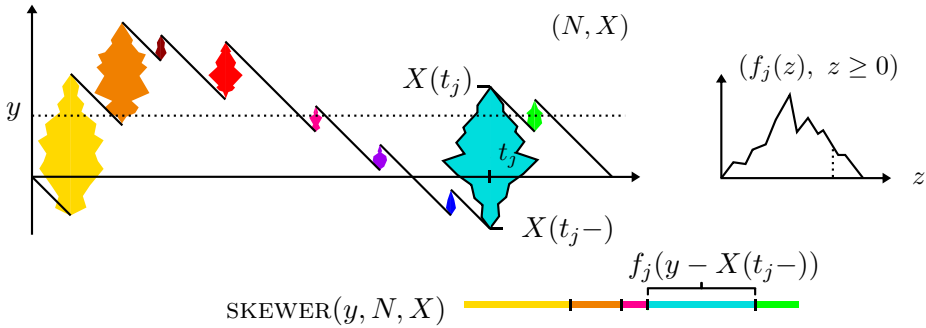}
 \caption{Left: The slanted black lines comprise the graph of the scaffolding $X$. Shaded blobs decorating jumps describe spindles: points $(t_j,f_j)$ of $N$. Right: Graph of one spindle. Bottom: A skewer, with blocks shaded to correspond to spindles; not drawn to scale.\label{fig:skewer_1}}
\end{figure}

Our construction is in the setting of \cite{Paper0} and requires two ingredients: (i) spectrally positive L\'evy processes that we call \textit{scaffolding}, and (ii) a family 
of independent %squared Bessel (\BESQ) 
excursions, called \textit{spindles}, one for each of the countably many jumps of the L\'evy processes. For each jump of the scaffolding, the corresponding excursion has a length given by the height of that jump. This allows us to imagine the spindles decorating the jumps. See Figure \ref{fig:skewer_1}, where we consider a scaffolding of finite variation and the spindles are represented by the laterally symmetric spindle-like shapes attached to the jumps. 

We define here an associated interval partition process, which at time $y$ is the output of a \emph{skewer} map at level $y$, as in Figure \ref{fig:skewer_1}. Let us 
first describe this map informally. 
As we move from left to right along the horizontal dotted line in Figure \ref{fig:skewer_1}, we encounter a sequence of spindles. Consider the widths of these spindles when intersected by this line, arrange them sequentially on the positive half-line, and slide them (as if on a skewer) towards the origin to remove gaps between them. The collection of the intervals of these widths now produces an interval partition. As $y$ varies we get a continuous process of interval partitions, which is our primary interest. See \ref{SuppSim} for a simulation of this. Any apparently overlapping spindles are still totally ordered left-to-right like the jumps in the scaffolding.

Let us formulate the above ideas more rigorously in the language of point processes that will be used throughout the rest of the paper. 
Recall that a continuous (positive) \textit{excursion} is a continuous function 
%Recall that a positive \textit{excursion} is a function 
$f\colon \BR \to [0,\infty)$ with the property that, for some $z>0$, we have $f(x) > 0$ if and only if $x\in (0,z)$. That is, the function escapes up from zero at time zero and is killed upon its first return. We write $\zeta(f) = z$; this is the \emph{lifetime} of the excursion. Let $\Exc$ denote a suitable space containing continuous excursions.

For $n \in \BN$, take $0 \leq t_1 < \cdots < t_n \leq T$ and let $f_1,\ldots,f_n$ denote continuous excursions. We represent this collection of pairs $(t_j,f_j)$ in a counting measure $N = \sum_{j=1}^n \Dirac{t_j,f_j}$. Here, $\Dirac{t,f}$ denotes a Dirac point mass at $(t,f)\in [0,\infty)\times\Exc$. For some constant $c_0>0$ and $t\in [0,T]$ we define\vspace{-.1cm}
\begin{equation}
 X(t) := -c_0t+\int_{[0,t]\times \Exc}\zeta(f)dN(u,f)= -c_0 t + \sum_{j=1}^n \zeta(f_j)\cf\left\{ 0\le t_j\le t \right\}.\vspace{-.1cm}\label{eq:discrete_JCCP_eg}
\end{equation}
When $t_j$s arrive at rate 1 and $f_j$s are i.i.d.\ from any distribution $\overline{\nu}$ on $\Exc$, then $X$ is a spectrally positive L\'evy process and $N$ is a Poisson random measure with intensity $\Leb\otimes\overline{\nu}$, both stopped at $T$. %Here, $\Leb$ denotes Lebesgue measure. 

\begin{definition}\label{def:skewer}
 For $y\in\BR$, $t\in [0,T]$, the \emph{aggregate mass} in $(N,X)$ at level $y$, up to time $t$ is\vspace{-0.2cm}
 \begin{equation}
  M_{N,X}^y(t) := \int_{[0,t]\times \Exc}f(y - X(s-))dN(s,f).\label{eq:agg_mass_from_spindles}\vspace{-0.1cm}
 \end{equation}
 The \emph{skewer} of $(N,X)$ at level $y$, denoted by $\skewer(y,N,X)$ is defined as \vspace{-0.1cm}
 \begin{equation}
  \left\{\left(M^y_{N,X}(t-),M^y_{N,X}(t)\right)\!\colon t\!\in\! [0,T],M^y_{N,X}(t-)\!<\!M^y_{N,X}(t)\right\}\!.\label{eq:skewer_def}\vspace{-0.1cm}
 \end{equation}
 The \emph{skewer process} $\skewerP(N,X)$ is defined as $\big( \skewer(y,N,X),y\!\geq\! 0\big)$.
\end{definition}

%% THIS PARAGRAPH

Lambert \cite{Lambert10} showed that certain Crump--Mode--Jagers branching processes could be represented in terms of L\'evy processes with bounded variation. For generalizations to unbounded variation, see \cite{LambertBravo18}. In our richer setting with spindles $f_i$,  $M_{N,X}(T)$ is structurally the same as the sum of characteristics studied by Jagers \cite{Jagers69,JagersBranching}; see \cite{Paper0}. In this analogy, the skewer process separates out the characteristics of the individuals in the population.

%Heuristically, it is useful to view $X$ as the contour process of an evolutionary tree of species, with the pre/post-jump height, $(X(t_j\minus),X(t_j))$, describing the birth/extinction time for one species, and the corresponding spindle $f_j$ describing its fluctuating population. Lambert \cite{Lambert10} observed that for the genealogies of binary, homogeneous Crump--Mode--Jagers processes, such contour processes are L\'evy processes. While Lambert did not mark jumps by excursions, similar ``characteristics'' evolving with time appear in Jagers' original work \cite{Jagers69,JagersBranching} on ``general branching processes.'' Jagers focused on sums of characteristics, while our skewer function represents the characteristics as blocks of an interval partition in the left-to-right order given by the contour function. 

%%

Figure \ref{fig:skewer_1} does not capture the level of complexity we require for self-similar processes. Fix $\alpha\in(0,1)$ and $q>\alpha$. Let $\kappa_{q}$ %\colon (0,\infty)\times\Sigma(\Exc)\to [0,1]$ 
denote a stochastic kernel from $(0,\infty)$ to $\Exc$ so that $\kappa_{q}(z,\,\cdot\,)$ is a probability distribution on continuous, positive excursions of lifetime $z$, with the scaling property:
\begin{equation}\label{eq:spindle_scaling}
 \text{if}\quad \ff\sim \kappa_{q}(1,\,\cdot\,) \quad \text{then} \quad z^q\ff(\,\cdot/z) \sim \kappa_{q}(z,\,\cdot\,),\ z\ge0.
\end{equation}
For some constant $c_\nu\in(0,\infty)$, let $\nu$ denote the $\sigma$-finite measure
\begin{equation}
 \nu = c_\nu\int_{(0,\infty)} \kappa_q(z,\cdot\,)z^{-\alpha-2}dz.\label{eqn:nu}
\end{equation}
Let $\bN$ be a Poisson random measure with intensity $\Leb\otimes\nu$ on $[0,\infty)\times \Exc$, \linebreak which we will abbreviate as 
\PRM[\Leb\otimes\nu]; this is our \emph{point process of spindles}. Each point in this process is a pair $(t,f)$ of time and spindle. The spindles are as in Figure \ref{fig:skewer_1}; time $t$ refers to the time axis of the scaffolding process, which is horizontal in that figure. The associated scaffolding is
\begin{equation}\label{eq scaffold}
 \bX(t) := \lim_{z\to 0}\left(\int_{[0,t]\times\{g\in\Exc\colon \zeta(g)>z\}}\zeta(f)d\bN(u,x,f) - t\frac{1}{\alpha}c_\nu z^{-\alpha}\right),
\end{equation}
for $t\ge 0$. This scaffolding is a spectrally positive \StableA\ L\'evy process with L\'evy measure $\Pi(dz) = \nu(\zeta\in dz) = c_\nu z^{-2-\alpha}$ and Laplace exponent
$\psi(\lambda) = c_\nu\Gamma(1-\alpha)\alpha^{-1}(1+\alpha)^{-1}\lambda^{1+\alpha}$.
This is the setting in which \cite[Corollary 7]{Paper0} established the continuity of what in our notation is the \em total mass process \em $(\|\skewer(y,\bN,\bX)\|,y\ge 0)$.
We strengthen this one-dimensional result to our interval-partition setting, as follows. 

\begin{theorem}[Continuity]\label{thm:diffusion_0} 
 Let $\alpha\in(0,1)$, $q>\alpha$ and $T>0$. Suppose that $\ff\sim\kappa_q(1,\,\cdot\,)$ is $\theta$-H\"older for some 
 $\theta\in(0,q-\alpha)$, and that the H\"older constant
 $$D_\theta:=\sup_{0\le r<s\le 1}\frac{|\ff(s)-\ff(r)|}{|s-r|^\theta}$$
 has moments of all orders. Consider $\nu$ as in \eqref{eqn:nu} and a Poisson random measure $\bN$ on $[0,T]\times\Exc$ with intensity $\Leb\otimes\mBxc$. Then  
  $\skewerP(\bN,\bX)$ is path-continuous in a suitable metric space $(\IPspace_{\alpha/q},d_{\alpha/q})$ of interval partitions.
\end{theorem}

We refer to such an interval partition evolution as a \em $(\nu,T)$-IP-evolution\em. 
The metric space $(\IPspace_{\alpha/q},d_{\alpha/q})$ was introduced in \cite{IPspace} (see also Section \ref{sec:IPspace} here) restricting to interval partitions $\beta$ for 
which the $\alpha/q$-diversity
$$\IPLT^{\alpha/q}_\beta(t):=\Gamma(1-\alpha/q)\lim_{h\downarrow 0}h^{\alpha/q}\#\{(a,b)\in\beta\colon\ |b-a|>h,b\le t\}$$
%
%% THIS PARAGRAPH
%
exists for all $t\in[0,\|\beta\|]$. We note that  
$\IPLT^{\alpha/q}_\beta(U):=\IPLT^{\alpha/q}_\beta(t)$ does not depend on $t\in U$ and write $\IPLT^{\alpha/q}_\beta(\infty):=\IPLT^{\alpha/q}_\beta(\|\beta\|)$. 

One can interpret our processes in the language of population genetics, with blocks representing species and their masses representing population sizes; then $\IPLT^{\alpha/q}_\beta(\infty)$ is a measure of genetic diversity in this regime of infinitely many species \cite{EthiKurt81,GPY2006alpha,GnedinRegenerationSurvey,RuggWalkFava13}. %The $(\nu,T)$-IP-evolution then represents evolving species frequencies. 
%The wealth of results and tools to study \StableA\ processes gives an excellent understanding of births and deaths of species. 
%The complexity of the jumps and level crossings of \StableA\ is such that birth and death times form dense sets and at a typical level, any two species are separated in the interval partition by infinitely many others. 

We prove Theorem \ref{thm:diffusion_0} at the end of Section \ref{sec:cont}. To obtain the Markov property of a $(\nu,T)$-IP-evolution, we need stronger assumptions. Specifically, we will assume that the spindles are associated with a $[0,\infty)$-valued diffusion process. It is well-known that self-similar diffusions form a three-parameter family, which we will call $(\alpha,c,q)$-block diffusions, and we find an appropriate excursion measure $\nu\!=\!\nu_{q,c}^{(-2\alpha)}$ when $\alpha\!\in\!(0,1)$, $q\!>\!\alpha$, $c\!>\!0$. See Section \ref{sec:BESQ}. Then, for suitable random times $T$, the $(\nu,T)$-IP-evolutions are Markovian. In Definition \ref{constr:type-1} we modify this construction to allow any fixed initial state. We refer to the processes thus constructed as $\nu$-IP-evolutions.

\begin{theorem}[Markovianity]\label{thm:diffusion}
 $\mBxc_{q,c}^{(-2\alpha)}$-IP-evolutions are self-similar path-continuous Hunt processes that can start at any point in $\big(\IPspace^{(1/q)}_{\alpha/q},d_{\alpha/q}\big)$, where \vspace{-0.2cm}
 $$\IPspace_{\alpha/q}^{(1/q)}:=\big\{\beta\in\IPspace_{\alpha/q}\colon\sum\nolimits_{U\in\beta}({\rm Leb}(U))^{1/q}<\infty\big\}.$$
\end{theorem}

This result is reminiscent of the Ray--Knight theorems that say that the local time process of (suitably stopped) Brownian motion is a diffusion as a process in the spatial
variable $y\ge 0$. We will show in a sequel paper 
\cite{PartB} that in the special case where the $(\alpha,c,q)$-block diffusion is a $(-2\alpha)$-dimensional squared Bessel processes, which we abbreviate as $\BESQ[-2\alpha]$, the total mass process is a $\BESQ[0]$, further strengthening the connection to the second Brownian Ray--Knight theorem \cite[Theorem XI.(2.3)]{RevuzYor}.

As a consequence of \cite[Theorem 2.4]{IPspace}, which establishes the continuity of various functions on $\IPspace_{\alpha/q}$, we obtain the following.

\begin{corollary}\label{cor:contfns} Let $(\beta^y,y\ge 0)$ be a $(\nu,T)$- or $\mBxc_{q,c}^{(-2\alpha)}$-IP-evolution as in Theorem \ref{thm:diffusion_0} or \ref{thm:diffusion}. Then the following processes are path-continuous.
  \begin{itemize}\item the measure-valued process $\textsc{measure}(\beta^y)\!=\!\sum_{U\in\beta^y}\Leb(U)\delta(\IPLT_{\beta^y}^{\alpha/q}(U))$, $y\ge 0$, in the space of compactly supported finite Borel
    measures on $[0,\infty)$ equipped with the topology of weak convergence;
    \item the (real-valued) total diversity process $\IPLT_{\beta^y}^{\alpha/q}(\infty)$, $y\ge 0$;
    \item the process $\textsc{ranked}(\beta^y)$, $y\ge 0$, of ranked block masses, in the space of summable decreasing sequences equipped with the $\ell_1$-metric;
    \item the total mass process $\|\beta^y\|$, $y\ge 0$. 
  \end{itemize}
\end{corollary}

\subsection{Bertoin's work on Bessel processes of dimension $d=1-\alpha\in(0,1)$}

In 1990, Bertoin \cite{Bertoin1990,Bertoin1990c} studied the excursions of Bessel processes of dimensions between 0 and 1. He 
decomposed the Bessel process $\mathbf{R}=\mathbf{B}-\frac{1}{2}(1-d)\mathbf{H}$ into a Brownian motion $\mathbf{B}$ and a path-continuous process $\mathbf{H}$ 
with zero quadratic variation. %He constructed excursions of the Markov process $(\mathbf{R},\mathbf{H})$ away from $(0,0)$, each consisting of infinitely 
%many excursions of $\mathbf{R}$ away from $0$.
%  We can associate with each excursion of $R$ away from $0$ a BESQ excursion and 
%Specifically, 
By extracting suitable statistics, namely $\mathbf{R}(t)$, where $\mathbf{H}(t)=y$, he showed \cite[Theorem II.2-II.3]{Bertoin1990c} that the 
measure-valued process
  $$y\mapsto\mu^y_{[0,T]}:=\sum_{0\le t\le T\colon\mathbf{H}(t)=y,\mathbf{R}(t)\neq 0}\delta_{\mathbf{R}(t)}$$ 
is path-continuous (with respect to the vague topology for $\sigma$-finite measures on $(0,\infty)$) and Markovian for suitable stopping times $T$ such as any inverse local time of $(\mathbf{R},\mathbf{H})$ at $(0,0)$. He further
showed in \cite[Corollary II.4]{Bertoin1990c} that 
  $$y\mapsto\lambda^y(T):=2\int_{(0,\infty)}x\mu_{[0,T]}^y(dx)=\sum_{0\le t\le T\colon\mathbf{H}(t)=y}\mathbf{R}(t)$$
is \BESQ[0]. We further explore the connection to Bertoin's results in \cite{PartB}. 
%Since $\mathbf{H}$ is increasing during excursions of $\mathbf{R}$ with slope $1/\mathbf{R}$, we may view $\mathbf{R}(t)$ as the 
%(occupation density) local time at level $\mathbf{H}(t)$ that $\mathbf{H}$ gains during the excursion of $\mathbf{R}$ straddling $t$. 
Specifically, we will establish the following more precise connection to our work on IP-evolutions.

\begin{theorem}\label{thm:Bertoin} In Bertoin's setting, 
  $$\left\{\left(\lambda^y(t-),\lambda^y(t)\right)\colon t\in[0,T],\,\mathbf{R}(t)\neq 0,\,\mathbf{H}(t)=y\right\}$$ 
  is a $\big(\nu_{1,1}^{(-2(1-d))},T\big)$-IP-evolution.
\end{theorem}
%NOTE: Could push that display in-line

\subsection{Further motivation: conjectures by Aldous, and by Feng and Sun}

In 1999, David Aldous \cite{AldousDiffusionProblem} asked to find a continuum analogue of a Markov chain on a space of $n$-leaf binary trees with the uniform 
tree as its invariant distribution. Specifically, he related this discrete Markov chain to certain Wright--Fisher diffusions with negative 
mutation rates and asked if there is a continuum-tree-valued process that has the Brownian Continuum Random Tree \cite{AldousCRT1} as its invariant distribution and incorporates 
the Wright--Fisher diffusions. The latter have since been studied by Pal \cite{Pal13}. In \cite{PartB,Paper3,Paper4}, we further study IP-evolutions and construct the 
continuum-tree-valued process, superseding the unpublished preprint \cite{PalPreprint}. Indeed, $\textsc{measure}(\beta^y)$ as in Corollary \ref{cor:contfns} can be viewed as a branch of length
$\IPLT_{\beta^y}^{\alpha/q}(\infty)$ equipped with subtree masses $\Leb(U)$ at locations $\IPLT_{\beta^y}^{\alpha/q}(U)$, which for $\nu_{1,1}^{(-1)}$-IP-evolutions is 
related to the evolution of a suitable branch in the continuum-tree process. 

A related process was proposed by L\"ohr, Mytnik and Winter as a scaling limit of Aldous's Markov chain on a new class of trees called algebraic trees, which ``can be seen as metric trees where one has `forgotten' the metric'' \cite{LohrMytnWint18}. This allows an approach via classical martingale problem methods. Our approach allows us to capture the metric, finding continuously fluctuating distances corresponding to local times of stable processes, at the expense of requiring the new constructions given here. 

%Our approach here gives a pathwise construction of the Aldous diffusion. In particular, we show the evolution of subtree masses to be Wright-Fisher diffusions. This connection is left implicit in \cite{LohrMytnWint18}. We also prove that the evolution of the tree-metric is related to the evolution of local times at various levels of a stable Levy process. This evolution, which was omitted in \cite{LohrMytnWint18}, is also the evolution of diversity of our IP-evolutions.

%Their connection to Wright--Fisher diffusions remains implicit. Our approach gives access to the full continuum tree and its subtrees spanned by branch points, including branch lengths (diversities!), which were discarded in \cite{LohrMytnWint18}.

In 2010, Feng and Sun \cite{FengSun10} conjectured the existence of a measure-valued process whose atom sizes (in ranked order) follow Petrov's Poisson--Dirichlet diffusion
\cite{Petrov09}, in a two-parameter setting $\alpha\in(0,1)$ and $\theta>-\alpha$. Petrov's diffusions, extending Ethier and Kurtz's $\alpha=0$ case \cite{EthiKurt81}, are 
stationary. In \cite{PartB}, we exhibit stationary Poisson--Dirichlet IP-evolutions with parameters $(\alpha,0)$ or $(\alpha,\alpha)$, laws which appear as zero sets of Brownian 
motion, Brownian bridge and more general Bessel processes and bridges. We show in \cite{Part2} that the associated ranked process is Petrov's diffusion. Our construction here 
gives new insights into Petrov's diffusions and provides an approach to measure-valued processes, which we intend to explore in future work.

\subsection{Structure of this paper}

Once the scaffolding-and-spindles construction of IP-evolutions
has been formally set up in Section \ref{sec:sample_space}, we can prove Theorem \ref{thm:diffusion_0} quickly in Section 
\ref{sec:cont}, building on the groundwork of \cite{Paper0,IPspace}. Proving Theorem \ref{thm:diffusion} 
requires more technical machinery. To avoid getting bogged down in technicality and notation before delivering 
the main punches of this paper, in the following section we: (1) highlight the main intermediate results on the path to this proof, 
(2) informally support these results with soft arguments, and (3) explain how they are used to prove Theorem 
\ref{thm:diffusion}.

\subsection{Overview of the proof of Theorem \ref{thm:diffusion}}\label{sec:overview}

The overall strategy is to start in the setting of \PRM s of \BESQ\ spindles, then generalize the initial states, then extend to more general spindles.

One of the key tools in our approach is It\^o's excursion theory \cite{Ito72, GreePitm80, BertoinLevy}.  Let $(\bN,\bX)$ be as before: $\bN$ is a \PRM\ on $[0,\infty)\times\Exc$ and $\bX$ the \StableA\ L\'evy process as in \eqref{eq scaffold}. Then $\bX$ is recurrent at every level. If we fix a level $y$, we can decompose $\bX$ into its excursions away from level $y$. Each excursion comprises three parts: an escape downwards from $y$, a single jump up across $y$, and a subsequent descent back to $y$. Each excursion thus contributes a unique block to the interval partition obtained from the skewer construction at level $y$, corresponding to the one jump across $y$; see Figure \ref{fig:bi-clade_decomp}. 

From It\^o's excursion theory, the excursions of $\bX$ away from $y$ %(excluding the partial excursion before the first hitting time of $y$) 
naturally form a \PRM\ whose atoms $\Dirac{s,g}$ are excursions $g$ of $\bX$ (shifted to start at $0$ and be excursions away from $0$) together with the local time $s$ at which they occur. It is helpful to think of our approach as separately marking jumps of each excursion of $\bX$ away from $y$ with spindles, which results in a \PRM\ of marked excursions. Technically, rather than working with \PRM s of marked excursions it is easier to identify each excursion of $\bX$ with the points in $\bN$ that determine it via \eqref{eq scaffold}.  This allows us to easily look simultaneously at the excursions of $\bX$ away from different levels $y$.  This results in a \PRM\ $\bF^y$ such that each atom $\Dirac{s,N}$ of $\bF^y$ is itself a point measure $N$ comprising those atoms of $\bN$ that correspond to an excursion $g$ of $\bX$. The details of this are contained in \emph{Proposition \ref{prop:bi-clade_PRM}}. 

The point measure $N$ corresponding to $g$ is called a ``\em bi-clade.\em''  This terminology arises from a genetics interpretation.  In this interpretation, the atoms of $N$ corresponding to the portion of $g$ that lies above level $y$ represent the descendants of the individual corresponding to the jump across level $y$ in the excursion, and thus represent a clade within the larger population represented by $\bN$.  ``Bi-clade'' comes from the fact that the atoms of $N$ also include the clade-like collection of atoms corresponding to the portion of $g$ that lies below level $y$, which we call an anti-clade.  The intensity measure of $\bF^y$ is of the form $\Leb \otimes \mClade$, where $\mClade$ is a $\sigma$-finite measure on bi-clades.

%\begin{figure}
%\centering
%% \includegraphics[scale=.6]{SRWExc40.png}\quad 
%\includegraphics[scale=.3]{StabExc5000v2.png} \qquad \includegraphics[scale=.3]{StabExc1v2.png}
%\caption{%Left: Decomposition of simple random walk on $\BZ$ into i.i.d.\ excursions about level 0. Middle: 
%Left: Decomposition of a spectrally positive \StableA\ L\'evy process into excursions. %Because 0 is a null-recurrent state for this process, there are infinitely many excursions in bounded time (though the graphic has finite resolution). 
%Right: A single excursion from the stable process. % in the middle. 
%Note the escape downwards from 0, the single jump up across level zero, and then the descent back to zero.\label{fig:stable_exc}}
%\end{figure}

%\begin{figure}
% \centering
% \includegraphics[scale=.375]{StabExcPPP5000.png}
% \caption{A Poisson point process, with each `\textbf{X}' representing an excursion of the Stable process in Figure \ref{fig:stable_exc}, correspondingly color-coded, with the height of the `\textbf{X}' corresponding to the length of the excursion.\label{fig:stable_exc_PPP}}
%\end{figure}

%Note how, in the middle and right panels of Figure \ref{fig:stable_exc}, each excursion of the \StableA\ process comprises three parts: an escape downwards from zero, a single jump up across level zero, and a subsequent descent back to level zero (though it may be that the positive part of the excursion, after the jump, is much longer than the negative part prior, or vice versa). Thus, the jumps across level zero in $\bX$ correspond bijectively to the bi-clades about level 0. 

\emph{Lemma \ref{lem:clade:invariance}} notes scaling and time-reversal invariance properties of $\mClade$ following from invariance properties of $\mBxc_{q,c}^{(-2\alpha)}$ and the \StableA\ process.

\emph{Lemma \ref{lem:mid_spindle_Markov}} notes a ``mid-spindle Markov property.'' Imagine a ``spindle-reader process'' -- an ant walking along the graph of $\bX$ and, at each jump, walking up from the bottom to the top of the jump, ``reading'' the corresponding spindle. 
Then Lemma \ref{lem:mid_spindle_Markov} can be thought of as a Markov property for this spindle-reader process at certain special stopping times. In particular, we prove this at the $k^{\text{th}}$ time that a spindle crosses level $y$ with mass greater than $\epsilon$. 
Specifically, we show that conditionally given the mass $f_T(y-\bX(t-))$ of the spindle $f_T$ as it crosses level $y$, the piece of that spindle below that level, $\check f_T^y$, and the point process prior to that time, $\restrict{\bN}{[0,T)}$, are independent of the piece of the spindle above level $y$, $\hat f_T^y$, and the subsequent point process, $\shiftrestrict{\bN}{(T,\infty)}$. 
Intuitively, this is a consequence of the Poisson property of $\bN$ and the Markov property of the spindles. %\texttt{Maybe say more about the proof}.

\begin{figure}
 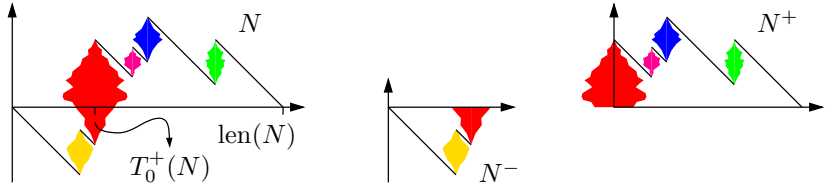
 \caption{Decomposition of a bi-clade around the middle mass, $m_0$.\label{fig:bi-clade_decomp}}
\end{figure}

%Note how %, in the middle and right panels of 
%in Figure \ref{fig:stable_exc} 
%Each excursion of the \StableA\ process comprises three parts: an escape downwards from zero, a single jump up across level zero, and a subsequent descent back to level zero. 
In Figure \ref{fig:bi-clade_decomp}, we see a bi-clade $N$ decomposed into two parts $(N^-,N^+)$, with the spindle marking the middle jump split into two \emph{broken spindles}. We write $m_0$ to denote the width (or \emph{mass}) of this middle spindle as it crosses level zero.

\emph{Proposition \ref{prop:clade_splitting}} states, firstly, that the measure $\mClade$ admits a natural disintegration by conditioning on $m_0$, so that, given $\bar N \sim \mClade(\,\cdot \mid m_0 = 1)$, we obtain a clade with law $\mClade(\,\cdot \mid m_0 = a)$ by suitably rescaling $\bar N$ by $a$. Secondly, under these conditional laws $\mClade(\,\cdot \mid m_0 = a)$, the two parts $(N^-,N^+)$ of the decomposed bi-clade in Figure \ref{fig:bi-clade_decomp} are independent. Thirdly, let $\hat f$ denote a broken spindle with initial mass $\hat f(0) = a$, as in the assumed Markov property of the spindles, independent of $\bN \sim \PRM[\Leb\otimes\mBxc_{q,c}^{(-2\alpha)}]$. Then the point process formed by beginning with spindle $\hat f$ at time 0, continuing according to $\bN$, and stopping when the resulting scaffolding hits 0, $\delta_{(0,\hat f)} + \restrict{\bN}{[0,T_{-\zeta(\hat f)}]}$, has the law of $N^+$ under $\mClade(\,\cdot\mid m_0 = a)$; and the point process formed by rotating this scaffolding-and-spindles picture $180^\circ$, effectively reversing the left-to-right order of spindles and flipping each spindle vertically, has the law of $N^-$ under $\mClade(\,\cdot\mid m_0 = a)$.

%The existence of this natural disintegration by conditioning on $m_0$ follows from general theory, which we review in Section \ref{sec:disintegration}. 
To prove %the rest of 
this proposition, we consider the first bi-clade $N$ about level $0$ that has middle mass $m_0 \ge 1$. By Proposition \ref{prop:bi-clade_PRM}, since the bi-clades arise from a \PRM[\Leb\otimes\mClade], we get $N\sim\mClade(\,\cdot \mid m_0>1)$. %Lemma \ref{lem:clade:invariance} states scaling and reversal invariance properties of the bi-clade It\^o measure $\mClade$; 
 From the scaling invariance of $\mClade$ noted in Lemma \ref{lem:clade:invariance}, we find that a normalized bi-clade $\bar N$, obtained by scaling $N$ to have $m_0=1$, has law $\mClade(\,\cdot\mid m_0 = 1)$. The mid-spindle Markov property, Lemma \ref{lem:mid_spindle_Markov}, applied to $\bN$ at the time of the middle spindle of $N$ then proves that $N^+$ and $N^-$ are conditionally independent and $N^+$ has the claimed conditional law. The claimed conditional law for $N^-$ follows from this and the reversal invariance of $\mClade$ of Lemma \ref{lem:clade:invariance}.

%Let $\bF^y$ denote the \PRM\ of bi-clades about level $y$. 
Via the mappings $N\mapsto N^+$ and $N\mapsto N^-$ sending a bi-clade to its \em clade \em and \em anti-clade \em parts, we map the bi-clade \PRM\ $\bF^y$ to \PRM s $\bF^{\ge y}$ of clades and $\bF^{\le y}$ of anti-clades. %For the sake of simplicity, in the present discussion we will ignore the first passage of $\bX$ to level $y$, which precedes all of the bi-clades of $\bN$ about level $y$. 
We define a \emph{level filtration} $(\cF^y,y\ge0)$ so that all spindles and parts of spindles arising below level $y$ in the scaffolding-and-spindles picture are $\cF^y$-measurable; this is made precise in Definition \ref{def:filtrations}. Informally, $\cF^y$ is also generated by the point process of anti-clades, $\bF^{\le y}$.%, though this may not be exactly right.

\emph{Proposition \ref{prop:PRM:Fy-_Fy+}} is a Markov-like property of $\bN$ in the level filtration: $\bF^{\ge y}$ is conditionally independent of $\cF^y$ %(and thus $\bF^{\le y}$) 
given $\skewer(y,\bN,\bX)$. This is akin to a Markov property in that $\bF^{\ge y}$ encodes the future of the skewer process beyond time $y$, while $\cF^y$ describes its past. Moreover, the conditional law of $\bF^{\ge y}$ given $\skewer(y,\bN,\bX) = \beta$ is the law of $\sum_{U\in\beta} \delta\big((\IPLT^{\alpha}_\beta(U),N_U)\big)$, where each $N_U$ is an independent clade with law $\mClade(N^+\in\cdot\mid m_0 = \Leb(U))$.% and $\IPLT_\beta(U)$ denotes the $\alpha$-diversity in $\beta$ to the left of block $U$.

A subtle point is that the $\alpha$-diversity in the interval partition $\IPLT^{\alpha}_\beta(U)$ corresponds to local times in the scaffolding-and-spindles picture, and therefore to the (local) time index in the point process of clades $\bF^{\ge y}$; this was studied in \cite{Paper0}, see also Theorem \ref{thm:LT_property_all_levels} below. Informally, we prove Proposition \ref{prop:PRM:Fy-_Fy+} by appealing to: (1) Proposition \ref{prop:bi-clade_PRM}, which asserts that $\bF^y$ is a \PRM, so bi-clades are independent, and (2) Proposition \ref{prop:clade_splitting}, which says that each bi-clade $N_U$ can be split into a clade and anti-clade, which are conditionally independent given their middle mass, which is the mass of the block $U$ in $\skewer(y,\bN,\bX)$ corresponding to that bi-clade. Then, gathering all of the clades together into $\bF^{\ge y}$ and the anti-clades into $\bF^{\le y}$, these are conditionally independent given the skewer.

From this proposition, most of the remaining work to prove Theorem \ref{thm:diffusion} involves passing results from the setting of the \PRM\ of spindles, $\bN$, to a more general setting of a point process of spindles $\bN_\beta$ engineered to describe an IP-evolution from any given deterministic or random initial state $\beta$. Proposition \ref{prop:type-1:Fy-_Fy+} restates Proposition \ref{prop:PRM:Fy-_Fy+} in that setting. From there, we easily conclude the simple Markov property of the IP-evolutions resulting from taking the skewer; this appears in Corollary \ref{cor:type-1:simple_Markov_1}. We extend this to a strong Markov property by proving continuity in the initial condition, in Proposition \ref{prop:type-1:cts_in_init_state}.

\section{Ingredients for the Poissonian construction}
\label{sec:sample_space}

In this section, we introduce the state space for IP-evolutions and formalise the scaffolding-and-spindles set-up. %MW

\subsection{Technical remarks}\label{sec:techrem}

\emph{A. Disintegrations and scaling.} We will require disintegrations of $\sigma$-finite excursion measures. Informally, for our purposes, if $(\cS,\ScS,\mu)$ is a measure space and $\phi\colon\cS\to\cT$ a measurable function, then a $\phi$-disintegration of $\mu$ is a stochastic kernel, which we will denote by $t\mapsto \mu(\,\cdot \mid \phi = t)$, with the following properties. Firstly, for $t\in \cT$, the law $\mu(\,\cdot \mid \phi = t)$ is supported on the pre-image $\phi^{-1}(t)$; and secondly
 \begin{equation}
  \mu(A) = \int \mu(A \mid \phi = t)\mu(\phi \in dt) \quad \text{for }A\in\ScS.\label{eq:scl_ker:integration}
 \end{equation}

In \ref{SuppMeas}, we observe that if $\mu$ satisfies an invariance identity with respect to a scaling operation on $\cS$ and $\phi$ interacts well with this scaling operation, then there is a canonical choice of disintegration so that an object with law $\mu(\,\cdot \mid \phi = t)$ can be obtained by suitably scaling an object with law $\mu(\,\cdot \mid \phi = 1)$. Throughout this paper, all disintegrations are of this kind.

\emph{B. Measurability of random counting measures.} As outlined in Section \ref{sec:overview}, random counting measures, sometimes called ``point processes,'' play a key role in this work. 
%MW Because we have abstracted these processes a level beyond what is common in the probability literature, we will check topological and measure-theoretic details. 
We use \cite{DaleyVereJones1, DaleyVereJones2} as our reference. In that framework, random counting measures $\mathbf{M}$ must always be boundedly finite on some complete, separable metric space $(X,d)$, meaning it must be a.s.\ the case that every bounded Borel set in $(X,d)$ has finite measure under $\mathbf{M}$. 
%MW As this would be a technical distraction in several places of the present work, 
We collect 
associated topological and measure-theoretic %MW 
%MW these 
arguments in \ref{SuppMeas}.

\subsection{The state space $(\IPspace,\dI)$: interval partitions with diversity}\label{sec:IPspace}

This section discusses the metric topology on interval partitions proposed in \cite{IPspace}.

%MW \begin{definition}\label{def:diversity_property}
 Let $\HIPspace$ denote the set of all interval partitions in the sense of Definition \ref{def:IP_1}. %For $q>0$, let $\HIPspace^q=\{\beta\in\HIPspace\colon\sum_{U\in\beta}(\Leb(U))^{1/q}<\infty\}$. 
 We say that an interval partition $\beta\in\HIPspace$ of a finite interval $[0,M]$ has the \emph{$\alpha$-diversity property}, or that $\beta$ is an \emph{interval partition with diversity}, if the following limit exists for every $t\in [0,M]$:\vspace{-0.2cm}
 \begin{equation}
  \IPLT_{\beta}^\alpha (t) := \Gamma(1-\alpha)\lim_{h\downto 0}h^\alpha\#\{(a,b)\in \beta\colon\ |b-a|>h,\ b\leq t\}.\label{eq:IPLT}\vspace{-0.2cm}
 \end{equation}
 For $\alpha\in(0,1)$, we denote by $\IPspace_\alpha \subset\HIPspace$ the set of interval partitions $\beta$ that possess the $\alpha$-diversity property. We call $\IPLT_{\beta}^\alpha (t)$ the $\alpha$-\emph{diversity} of the interval partition up to $t\in[0,M]$. For $U\in\beta$, $t\in U$, we write $\IPLT_{\beta}^\alpha (U)=\IPLT_{\beta}^\alpha (t)$, and we write $\IPLT_{\beta}^\alpha (\infty) := \IPLT_{\beta}^\alpha (M)$ to denote the \emph{total ($\alpha$-)diversity} of $\beta$. We often fix $\alpha\in(0,1)$ and simplify notation $\IPspace:=\IPspace_\alpha $ and $\IPLT_\beta:=\IPLT_\beta^\alpha $. 
%  We also write $\IPspace_\alpha^q=\IPspace_\alpha\cap\HIPspace^q$. 
%MW \end{definition}

Note that $\IPLT_\beta^\alpha (U)$ is well-defined, since $\IPLT_{\beta}^\alpha $ is constant on each interval $U \in \beta$, as the intervals of $\beta$ are disjoint. 

\begin{proposition}\label{prop:IP:Stable} For a \Stable[\alpha] subordinator $Y=(Y(s),\;s\geq 0)$ with Laplace
  exponent $\Phi(\lambda) = \lambda^{\alpha}$ and any $T>0$, the interval partition \vspace{-0.2cm}
 \begin{equation}\label{eq:IPs_from_subord}
  \beta := \{(Y(s-),Y(s))\colon\ s\in [0,T),\ Y(s-)<Y(s)\}% \quad \text{and} \quad \bar\beta := \frac{1}{Y(T)} \scaleI \beta\vspace{-0.2cm}
 \end{equation}
 has $\alpha$-diversity $\IPLT_\beta^\alpha (\infty)=T$ a.s.. We call $\beta$ a \Stable[\alpha] interval partition. 
\end{proposition}
\begin{proof}
 This follows from the Strong Law of Large Numbers for the Poisson process of jumps and the monotonicity of $\IPLT_{\beta}^\alpha (t)$ in $t$.
\end{proof}

We adopt the standard discrete mathematics notation $[n] := \{1,2,\ldots,n\}$.
To define metrics on $\HIPspace$ and on $\IPspace_\alpha$, we say that a \emph{correspondence} from $\beta\in\HIPspace$ to $\gamma\in\HIPspace$ is a finite sequence of pairs of intervals $(U_1,V_1),\ldots,(U_n,V_n) \in \beta\times\gamma$, $n\geq 0$, where the sequences $(U_j)_{j\in [n]}$ and $(V_j)_{j\in [n]}$ are each strictly  
increasing in the left-to-right ordering of the interval partitions.

Now fix $\alpha\in(0,1)$. For $\beta,\gamma\in\IPspace_\alpha $, the $\alpha$-\emph{distortion} $\dis_\alpha (\beta,\gamma,(U_j,V_j)_{j\in [n]})$
of a correspondence $(U_j,V_j)_{j\in [n]}$ from $\beta$ to $\gamma$ is defined to be the maximum of the following four quantities:
 \begin{enumerate}[label=(\roman*), ref=(\roman*)]
  \item $\sum_{j\in [n]}|\Leb(U_j)-\Leb(V_j)| + \IPmag{\beta} - \sum_{j\in [n]}\Leb(U_j)$, \label{item:IP_m:mass_1}
  \item $\sum_{j\in [n]}|\Leb(U_j)-\Leb(V_j)| + \IPmag{\gamma} - \sum_{j\in [n]}\Leb(V_j)$. \label{item:IP_m:mass_2}
  \item $\sup_{j\in [n]}|\IPLT_{\beta}^\alpha (U_j) - \IPLT_{\gamma}^\alpha (V_j)|$,
  \item $|\IPLT_{\beta}^\alpha (\infty) - \IPLT_{\gamma}^\alpha (\infty)|$,
 \end{enumerate}
Similarly, the \em Hausdorff distortion \em is defined to be the maximum of (i)-(ii). \pagebreak

\begin{definition} \label{def:IP:metric}%\label{def:IP:correspondence}%\label{def:IP:distortion}
 For $\beta,\gamma\in\HIPspace$ we define\vspace{-0.1cm}
 \begin{equation} d_H^\prime(\beta,\gamma):=\inf_{n\ge 0, (U_j,V_j)_{j\in[n]}}\dis_H(\beta,\gamma,(U_j,V_j)_{j\in[n]},\vspace{-0.1cm}
 \end{equation}
 where the infimum is over all correspondences from $\beta$ to $\gamma$.
 
 For $\beta,\gamma\in\IPspace_\alpha$ we similarly define\vspace{-0.1cm}
 \begin{equation}\label{eq:IP:metric_def}
  d_\alpha (\beta,\gamma) := \inf_{n\ge 0,\,(U_j,V_j)_{j\in [n]}}\dis_\alpha \big(\beta,\gamma,(U_j,V_j)_{j\in [n]}\big).\vspace{-0.1cm}
 \end{equation}
 We often fix $\alpha\in(0,1)$ and use notation $(\IPspace,d_\IPspace):=(\IPspace_\alpha,d_\alpha)$.%MW
\end{definition}

\begin{theorem}[Theorems 2.2 and 2.3 of \cite{IPspace}]\label{thm:Lusin}
 $(\HIPspace,d_H^\prime)$ is a complete metric space. $(\IPspace_\alpha ,d_\alpha )$ is Lusin, i.e.\ a metric space that is homeomorphic to a Borel subset of a compact metric space. 
\end{theorem}

We further show in \cite{IPspace} that the $d_H^\prime$-topology is the same as the topology generated by the Hausdorff distance $d_H$ between (the complements such as  
$C_\beta:=[0,\IPmag{\beta}]\setminus\bigcup_{U\in\beta}U$ of) interval partitions. We give a detailed account of the topological properties of $(\IPspace_\alpha ,d_\alpha )$ and 
$(\HIPspace,d_H^\prime)$ in \cite{IPspace}. It follows from Theorem \ref{thm:Lusin} that they are Borel spaces, i.e. bi-measurably in bijective correspondence 
with Borel subsets of $[0,1]$. In this setting, regular conditional distributions exist; see Kallenberg \cite[Theorem A1.2, Theorem 6.3]{Kallenberg}. 

There are various natural operations for interval partitions. We define a \emph{scaling map} $\scaleI\colon (0,\infty)\times\HIPspace\rightarrow\HIPspace$ by saying, for $c>0$ and $\beta \in\HIPspace$,
 \begin{equation}
  \scaleI[c][\beta] = \{(ca,cb)\colon (a,b)\in \beta\}.\label{eq:IP:scale_def}
 \end{equation}
 Let $(\beta_a)_{a\in\mathcal{A}}$ denote a family of interval partitions indexed by a totally ordered set $(\cA,\preceq)$. For the purpose of this definition, let 
% $S_{\beta}(a) := \sum_{b\preceq a}\IPmag{\beta_b}$ and 
 $S(a-) := \sum_{b\prec a}\IPmag{\beta_b}$ for $a\in\mathcal{A}$.
 If $S(a-) < \infty$ for every $a\in \mathcal{A}$, then we define the \emph{concatenation}\vspace{-0.1cm}
 \begin{equation}\label{eq:IP:concat_def}
  \Concat_{a\in\mathcal{A}}\beta_a := \{(x+S(a-),y+S(a-)\colon\ a\in\mathcal{A},\ (x,y)\in \beta_a\}.\vspace{-0.1cm}
 \end{equation}
 When $\mathcal{A}=\{a_1,a_2\}$, $a_1\prec a_2$, we denote this by $\beta_{a_1}\concat\beta_{a_2}$. 
 If $\sum_{a\in\mathcal{A}}\IPmag{\beta_a} < \infty$, we call $(\beta_a)_{a\in\mathcal{A}}$ \emph{summable}, then  
 \emph{strongly summable} if the concatenated partition has the diversity property \eqref{eq:IPLT}. 
%MW\end{definition}

The following lemmas record some elementary properties of $d_H^\prime$ and $d_\alpha$.

\begin{lemma}\label{lem:IP:concat}
  If $(\beta_a)_{a\in \mathcal{A}},(\gamma_a)_{a\in \mathcal{A}} \in \IPspace_\alpha^{\mathcal{A}}$ are strongly summable, then\vspace{-0.1cm}
  \begin{equation}\label{eq:IP:concat_dist}
   d_\alpha\left(\Concat_{a\in\mathcal{A}}\beta_a,\ \Concat_{a\in\mathcal{A}}\gamma_a\right) \leq \sum_{a\in\mathcal{A}}d_\alpha(\beta_a,\gamma_a).\vspace{-0.1cm}
  \end{equation}
  The same holds for $d_H'$ when $(\beta_a)_{a\in \mathcal{A}},(\gamma_a)_{a\in \mathcal{A}}\in \HIPspace^{\mathcal{A}}$ are summable.
\end{lemma}

\begin{lemma}\label{lem:IP:scale}
 For $\beta,\gamma\in\IPspace_\alpha$ and $c>0$,
 \begin{gather}
  \dH'(\beta,\scaleI[c][\beta]) = \left|c-1\right|\IPmag{\beta}, \qquad \dH'(c\scaleI\beta,c\scaleI\gamma) = c\dH'(\beta,\gamma),\label{eq:IP:Haus_scale}\\
  d_\alpha(\beta,\scaleI[c][\beta]) \leq \max\left\{\left|c^\alpha - 1\right|\IPLT_{\beta}(\infty), \left|c-1\right|\IPmag{\beta}\right\},\label{eq:IP:scaling_dist_1}\\
  \min\{c,c^\alpha\}d_\alpha(\beta,\gamma) \leq d_\alpha(\scaleI[c][\beta],\scaleI[c][\gamma]) \leq \max\{c,c^\alpha\}d_\alpha(\beta,\gamma).\label{eq:IP:scaling_dist_2}
 \end{gather}
\end{lemma}

\subsection{Spindles: excursions as block size evolutions}
\label{sec:BESQ}

Let $(\cD,d_\cD)$ denote the Skorokhod space of real-valued c\`adl\`ag functions. Recall that its Borel $\sigma$-algebra $\Sigma(\cD)$ is generated by the evaluation
maps $g\mapsto g(t)$, $t\in\mathbb{R}$; see \cite[Theorem 14.5]{Billingsley}. Let $\Exc$ be the subset of non-negative real-valued excursions that are continuous, possibly excepting \cadlag\ jumps at their times of birth (time 0 as elements of $\Exc$) and death:
 \begin{equation}
  \Exc := \left\{f\in\cD\ \middle| \begin{array}{c}
    \displaystyle \exists\ z\in(0,\infty)\textrm{ s.t.\ }\restrict{f}{(-\infty,0)\cup [z,\infty)} = 0,\\[0.2cm]
    \displaystyle f \text{ positive and continuous on } (0,z)%\restrict{f}{(0,z)} > 0\text{, and }\restrict{f}{[0,z)}\textrm{ is continuous}
   \end{array}\right\}.\label{eq:cts_exc_space_def}
 \end{equation}
Let $\SExc$ denote the Borel $\sigma$-algebra on $\Exc$ generated by $d_\cD$. We define the \emph{lifetime} and \emph{amplitude} 
$\life,A\colon \Exc \to (0,\infty)$ via
 \begin{equation}
  \life(f) = \sup\{s\!\geq\! 0\colon f(s)\!>\!0\},\ \mbox{and}\  A(f)=\sup\{f(s),0\!\le\! s\!\le\!\zeta(f)\}.
  \label{eqn:lifeamplitude}
 \end{equation}

Squared Bessel processes (\BESQ) are a family of diffusions solving
$$dZ_s=\delta\,ds+2\sqrt{Z_s}dB_s,\qquad Z_0=y,$$
for $s\le\zeta(Z)=\inf\{r\ge 0\colon Z_r=0\}$. Here, $\delta\in\mathbb{R}$ is a parameter. We make the boundary state 0 absorbing by setting $Z_s=0$ for $s\ge\zeta(Z)$. These diffusions contain the Feller diffusion, which is a continuous-state 
branching process, when the \em dimension \em parameter is $\delta=0$, with immigration when $\delta>0$. The squared norm of a $\delta$-dimensional Brownian 
motion is a \BESQ[\delta] starting from 0, when $\delta\in\BN$. Let $\alpha\in(0,1)$. The case $\delta=-2\alpha$ can be interpreted as emigration. In this case (as when $\delta=0$), 
the boundary point 0 is not an entrance boundary, while exit at 0 happens almost surely. See \cite{PitmYor82,GoinYor03,Pal13}. 

\begin{lemma}[Equation (13) in \cite{GoinYor03}]\label{lem:BESQ:length}
  Let $Z=(Z_s,s\ge 0)$ be a \BESQ[-2\alpha] process starting from $z >0$. Then $\zeta(Z)$ has law  
  \InvGammaDist[1+\alpha,z/2], i.e.\ $z/2\zeta(Z)$ has density $(\Gamma(1+\alpha))^{-1}x^{\alpha}e^{-x}$, $x\in(0,\infty)$. 
\end{lemma}

Pitman and Yor \cite{PitmYor82} constructed excursion measures $\Lambda$ for diffusions even when there is no reflecting extension (to replace absorption at 0) that has $\Lambda$ as its 
It\^o excursion measure. They gave several descriptions, the first of which yields the following for the special case of \BESQ[-2\alpha]. We define \emph{first passage times} $H^a\colon\Exc\to[0,\infty]$ via $H^a(f)=\inf\{s\!\ge\! 0\colon f(s)\!=\!a\}$, $a\!>\!0$.

\begin{lemma}[Section 3 of \cite{PitmYor82}]\label{lem:BESQ:existence}
  There is a measure $\Lambda$ on $\Exc$ such that $\Lambda\{f\in\Exc\colon f(0)\neq 0\}=0$, 
  $\Lambda\{H^a<\infty\}=a^{-1-\alpha}$, $a>0$, and under $\Lambda(\,\cdot\;|\,H^a<\infty)$, the restricted canonical process 
  $f|_{[0,H^a]}$ is a \BESQ[4+2\alpha] process starting from 0 and stopped at the first passage time of $a$, independent of $f(H^a+\cdot\,)$, which is a \BESQ[-2\alpha] process 
  starting from $a$. 
\end{lemma}

We will consider a constant multiple of $\Lambda$ as an intensity of a Poisson random measure on $[0,\infty)\times\Exc$. %NOTE: could insert ref to SuppMeas here
%, using the framework of \cite{DaleyVereJones1, DaleyVereJones2} as our reference. %Here, $[0,\infty)$ is an auxiliary space, while we will find ``time'' in the $\Exc$-component of each atom of the Poisson random measure. 
In the setting of scaffoldings and spindles discussed in the introduction, changing the intensity by a constant corresponds to time-changing the scaffolding, which will not impact the skewer map of Definition \ref{def:skewer} or our interval partition diffusions. We %during the construction 
make the following choice (so that $\Phi(\lambda)=\lambda^{\alpha}$ in Proposition \ref{prop:agg_mass_subord}).
%MW
%MW\begin{definition}\label{def:excmeas}
 We define 
\begin{equation}\label{excmeas}\mBxc_{\tt BESQ}^{(-2\alpha)}:=(2\alpha(1+\alpha)/\Gamma(1-\alpha))\Lambda\end{equation}%MW
 as our \BESQ[-2\alpha] \emph{excursion measure}, where 
 $\Lambda$ is the Pitman--Yor excursion measure of Lemma \ref{lem:BESQ:existence}. We call continuous elements of $\Exc$ such as $\mBxc_{\tt BESQ}^{(-2\alpha)}$-a.e.\ 
 $f\in\Exc$ \emph{spindles} and elements of $\Exc$ with a discontinuity at birth and/or death \emph{broken spindles}.
%\end{definition}
While every spindle $f\in\Exc$ has an intrinsic lifetime $\life(f)\in[0,\infty)$, the scaffolding of Section \ref{sec:prelim:JCCP} will shift spindles to non-zero birth
times that are not intrinsic to each spindle.   

\begin{lemma}\label{lem:BESQ:exc_length} For the excursion measure %MW of Definition \ref{def:excmeas},
 \eqref{excmeas}, %MW 
  we have for $m\!>\!0$, $y\!>\!0$, 
  $$\mBxc_{\tt BESQ}^{(-2\alpha)}\{A\!>\!m\}=\frac{2\alpha(1\!+\!\alpha)m^{-1-\alpha}}{\Gamma(1-\alpha)} %\label{eq:BESQ:exc_amp}
  \mbox{ and }\mBxc_{\tt BESQ}^{(-2\alpha)}\{\life\!>\!y\}= \frac{\alpha y^{-1-\alpha}}{2^\alpha\Gamma(1\!-\!\alpha)\Gamma(1\!+\!\alpha)}.$$%\label{eq:BESQ:exc_length}
   %\mBxc\{X : \len(X)\in ds\} = (2\pi)^{-\frac12}s^{-\frac52}ds.\label{eq:BESQ:exc_length}
\end{lemma}
\begin{proof} The first formula follows straight from Lemma \ref{lem:BESQ:existence}. We can calculate the second one using \cite[Description (3.2)]{PitmYor82} to express
  $\Lambda\{\life>s\}$ in terms of a
  \BESQ[4+2\alpha] process $Z$ starting from 0, whose probability density function at time $s$ is given in \cite[Equation (50)]{GoinYor03}: 
  $$\Lambda(\life\!>\!s)=\EV[Z_s^{-1-\alpha}]=\!\int_0^\infty\!\! y^{-1-\alpha}\frac{(2s)^{-2-\alpha}y^{1+\alpha}}{\Gamma(2+\alpha)}e^{-y/2s}dy
                                         =\frac{s^{-1-\alpha}}{2^{1+\alpha}\Gamma(2\!+\!\alpha)}.$$\vspace{-0.8cm}
  
\end{proof}

%MW \begin{definition}\label{def:BESQ:scaling_def}
 We define a \emph{reversal involution} $\reverseexc\colon\Exc\rightarrow\Exc$ and, for any fixed $q>0$ that we suppress notationally, a \emph{spindle scaling map} $\scaleB\colon (0,\infty)\times\Exc\rightarrow\Exc$, by saying, for $a>0$ and $f\in\Exc$,
 \begin{equation}%MW
  \reverseexc (f)\!:=\!\big(f\big((\life(f)\!-\!y)\!-\!\big),y\!\in\!\BR\big)  \text{ and }  \scaleB[a][f]\!:=\!\left(a^qf(y/a),y\!\in\!\BR\right)\!.\label{eq:BESQ:scaling_def}%MW
 \end{equation}%MW
%MW \end{definition}

\begin{lemma}\label{lem:BESQ:invariance}
 For $B\in \cExc$, $a>0$, and for spindle scaling with $q=1$
 \begin{equation*}
  \mBxc_{\tt BESQ}^{(-2\alpha)}(\reverseexc(B)) = \mBxc_{\tt BESQ}^{(-2\alpha)}(B)\ \text{and}\ %\label{eq:BESQ:reversal_inv}
  \mBxc_{\tt BESQ}^{(-2\alpha)}(\scaleB[a][B]) = a^{-1-\alpha}\mBxc_{\tt BESQ}^{(-2\alpha)}(B).%\label{eq:BESQ:scaling_inv}
 \end{equation*}
\end{lemma}
\begin{proof} Time reversibility can be read from \cite[(3.3)]{PitmYor82}. The scaling relation follows from Lemma \ref{lem:BESQ:existence} and the scaling
  properties of \BESQ[-2\alpha] and \BESQ[4+2\alpha] as noted e.g.\ in \cite[A.3]{GoinYor03}. 
\end{proof}

Scaling as in \eqref{eq:BESQ:scaling_def} and Lemma \ref{lem:BESQ:invariance}, the pair $(\mu,\phi) = \big(\mBxc_{\tt BESQ}^{(-2\alpha)} , \zeta\big)$ falls into the setting of Section \ref{sec:techrem}. This yields the following.

%By \eqref{eq:BESQ:scaling_def} %MW
%and Lemma \ref{lem:BESQ:invariance}, the scaling and lifetime maps $\scaleB$ and $\life$ fall into the setting of Section \ref{sec:techrem}. This yields the following.

\begin{corollary}\label{cor:BESQ:scl_ker}
  There exists a $\life$-disintegration of $\mBxc_{\tt BESQ}^{(-2\alpha)}$, denoted by $\mBxc_{\tt BESQ}^{(-2\alpha)}(\,\cdot\,|\,\life)$, that is unique with the following property. For every $a,b\!>\!0$, if $\mathbf{f}$  has law $\mBxc_{\tt BESQ}^{(-2\alpha)}(\,\cdot\;|\,\life\!=\!a)$ then $\scaleB[(b/a)][\mathbf{f}]$, for $q\!=\!1$, has law $\mBxc_{\tt BESQ}^{(-2\alpha)}(\,\cdot\;|\,\life\!=\!b)$.
\end{corollary}

\begin{lemma}[e.g.\ Corollary 3 of \cite{Paper0}]\label{lem:BESQ:Holder}
 For every $\theta\in (0,\frac12)$, $\mBxc_{\tt BESQ}^{(-2\alpha)}$-a.e.\ excursion is H\"older-$\theta$.
\end{lemma}

%In Lemma \ref{lem:spindle_piles}, we exhibit subfamilies of the \BESQ[-2\alpha] excursions of a \PRM[\Leb\otimes\mBxc_{\tt BESQ}^{(-2\alpha)}] that are uniformly H\"older-$\theta$. 
%Here, \PRM[\mu] stands for Poisson random measure with intensity $\mu$.

Any continuous $\Exc$-valued random excursion $\bff$ of length $\life(\bff)=1$ provides a model for a block size evolution that lasts one time unit. 
Fix $\alpha\in(0,1)$ and $q>\alpha$. Assume $\int_0^1\EV[(\bff(y))^{\alpha/q}]dy<\infty$. For any $x>0$, denote by $\kappa_q(x,\cdot)$ the distribution of $\scaleB[x][\bff]=(x^q\bff(y/x),y\in\BR)$ and let\vspace{-0.1cm}
\begin{equation}\label{eqn:generalnu}
 \nu=\int_0^\infty c_{\nu} x^{-\alpha-2}\kappa_q(x,\cdot\,)dx,\quad c_{\nu}=\frac{\alpha}{\Gamma(1\!-\!\alpha/q)\int_0^1\EV[(\bff(y))^{\alpha/q}]dy}.
\end{equation} 
By Lemma \ref{lem:BESQ:exc_length} and Corollary \ref{cor:BESQ:scl_ker}, the excursion measure $\mBxc_{\tt BESQ}^{(-2\alpha)}$ is a special case of this 
general construction, for which $c_{\nu}=\alpha(1+\alpha)/2^\alpha\Gamma(1-\alpha)\Gamma(1+\alpha)$, $q=1$ and 
$\kappa_1(x,\cdot)=\nu_{\tt BESQ}^{(-2\alpha)}(\,\cdot\,|\,\zeta=x)$. In the present paper we discuss 
$\dI$-continuity in models based on the general construction \eqref{eqn:generalnu} as stated in Theorem \ref{thm:diffusion_0}. We also investigate the Markov property that is
key to Theorem \ref{thm:diffusion}. The 
natural generality for a Markov 
property are (multiples of) Pitman--Yor excursion measures $\nu$ of suitable self-similar diffusions. 

Indeed, it follows e.g.\ from Lamperti's 
\cite{Lamperti72} characterization of positive self-similar Markov processes as time-changed exponential L\'evy processes, that, in our case of 
continuous sample paths, every positive self-similar Markov process that is absorbed when reaching 0 can be obtained from a squared Bessel process 
of some dimension by a power transformation of space by $x\!\mapsto\! cx^q$. For our
purposes we will also need $\alpha\!\in\!(0,1)$, $q\!>\!\alpha$ and $c\!>\!0$ so that \eqref{eqn:generalnu} and \eqref{eq scaffold} are well-defined. We refer to such a diffusion as an 
$(\alpha,q,c)$-block diffusion. By similarly transforming $\nu_{\tt BESQ}^{(-2\alpha)}$, we also associate an excursion measure $\nu_{q,c}^{(-2\alpha)}$. We develop this in 
detail in Section \ref{IPevolgeneral}.

\subsection{Scaffolding: \StableA\ processes to describe births and deaths of blocks}
\label{sec:prelim:JCCP}

For $\nu$ as in \eqref{eqn:generalnu}, let $\bN$ denote a \PRM[\Leb\otimes\mBxc] on $[0,\infty)\times \Exc$. 
%The idea of the pair $(N,X)$ and the skewer map of Definition \ref{def:skewer} is to 
%associate with each atom $(t,f)$ of $N$ a spindle birth time $X(t\minus)$ and to extract $f(y-X(t\minus))$ as its mass at level $y\in\BR$. Also, the scaffolding $X$ has $\Delta X(t):= X(t)-X(t\minus)=\zeta(f)$. 
By mapping a spindle $f$ to its lifetime $\zeta(f)$, we obtain 
%Recall that we refer to the excursions $f$ arising in points $(t,f)$ of this point process as ``spindles.'' 
the associated point process of spindle lifetimes $\int\delta(s,\zeta(f))d\bN(s,f)$, which is a \PRM[\Leb\otimes\mBxc(\zeta\in\cdot\,)]. Note that    
$$\int_{(z,\infty]}x\,\mBxc(\zeta\in dx)=\int_\Exc\cf\{\zeta(f)>z\}\zeta(f)d\mBxc(f) = \frac{1}{\alpha}c_{\nu}z^{-\alpha}\longrightarrow\infty\quad\text{as }z\downarrow 0.$$
Thus, if we take these spindle lifetimes to be the heights of jumps for a \cadlag\ path, as in the introduction, %in the manner of 
%\eqref{eq:discrete_JCCP_eg}, 
then these jumps are almost surely not summable. To define a path $\bX$ associated with $\bN$ in this manner, we require a limit with
compensation. We give a definition generalising \eqref{eq scaffold} that will also apply to random measures constructed from independent copies of $\bN$. %Our reference for  
%measures on Polish metric spaces is \cite{DaleyVereJones1,DaleyVereJones2}.

%MW \begin{definition}\label{def:JCCP}
 For a complete, separable metric space $(\cS,d_\cS)$, denote by $\cN(\cS)$ the set of counting measures $N$ on $\cS$ that are boundedly finite: $N(B)<\infty$ for all bounded Borel sets $B\subset\cS$. We equip $\cN(\cS)$ with the $\sigma$-algebra $\Sigma(\cN(\cS))$ generated by evaluation maps $N\mapsto N(B)$. 
 
 % Recall $(\Exc,d_{\cD}^A)$ from Proposition \ref{prop:E_bdedly_finite}. We metrize $[0,\infty)\times\Exc$ via the sum of the Euclidean metric in the first coordinate and $d_{\cD}^A$ in the second. 
 For $N\in \cNRE$, we define the \emph{length} of $N$ as \vspace{-0.1cm}
 \begin{equation}
  \len(N) := \inf\Big\{t>0\colon N\big([t,\infty)\times\Exc\big) = 0\Big\} \in [0,\infty].\label{eq:assemblage:len_def}\vspace{-0.1cm}
 \end{equation}
 When the following limit exists for $t\in [0,\len(N)]\cap [0,\infty)$, we define \vspace{-0.1cm}
 \begin{equation}
  \xi_N(t) := \lim_{z\downto 0}\left(\int_{[0,t]\times\{g\in\Exc\colon\zeta(g) > z\}}\life(f)dN(s,f) - t\frac{1}{\alpha}c_{\nu}z^{-\alpha}\right).\label{eq:JCCP_def}\vspace{-0.1cm}
 \end{equation}   %\int_{s\in [0,t],\;f\in\Exc\colon\zeta(f) > z}
 We also set $\xi_N(t)=0$ for $t>\len(N)$ and write\vspace{-0.1cm} %MW
 $$\xi(N) := \big( \xi_N(t),\,t\ge 0 \big).\vspace{-0.1cm}$$ %MW
%MW \end{definition}

%We will take $\xi(N)$, or some variant, to be the \emph{scaffolding} associated with a point process of spindles $N$, as described in the introduction.
The limit in \eqref{eq:JCCP_def} only exists for a rather specific class of measures $N$.

\begin{proposition}\label{prop:stable_JCCP}
  For $\bN$ a \PRM[\Leb\otimes\mBxc] on $[0,\infty)\times \Exc$, the convergence in \eqref{eq:JCCP_def} holds a.s.\ uniformly in $t$ on any bounded interval.   
  Moreover, the scaffolding $\xi(\bN)$ is a spectrally positive stable L\'evy process of index $1+\alpha$, with L\'evy measure and Laplace exponent given by\vspace{-0.1cm} 
 \begin{equation}
  \mBxc(\zeta\in dx)=c_{\nu}x^{-2-\alpha}dx\quad\mbox{and}\quad\psi(\lambda) = c_{\nu}\frac{\Gamma(1-\alpha)}{\alpha(1+\alpha)}\lambda^{1+\alpha}.\label{eq:JCCP:Laplace}\vspace{-0.1cm}
 \end{equation}
\end{proposition}

\begin{proof} By Lemma \ref{lem:BESQ:exc_length} and elementary Poisson random measure arguments, the pre-limiting quantity is a compensated compound Poisson process.
  By the L\'evy--It\^o decomposition of L\'evy processes, e.g.\ in \cite[Theorem 19.2]{Sato}, the remaining conclusions follow. Specifically, we use%\vspace{-0.2cm} 
  $$\psi(\lambda)\!=\!\!\int_0^\infty\!\!(e^{-\lambda x}\!-\!1\!+\!\lambda x)\mBxc(\life\!\in\! dx),\ \ \int_0^\infty\!\!(e^{-\lambda x}\!-\!1\!+\!\lambda x)\frac{\alpha x^{-2-\alpha}}{\Gamma(1\!-\!\alpha)}dx\!=\!\frac{\lambda^{1+\alpha}}{1\!+\!\alpha}. \qedhere$$%; maybe ex I.10 in Bertoin.
  %\vspace{-0.2cm} 
\end{proof}

%MW \begin{notation}
 Henceforth we write ``\StableA'' to refer exclusively to L\'evy processes with the Laplace exponent specified in \eqref{eq:JCCP:Laplace}. In particular, such processes are spectrally positive. 
  We write $\bX := \xi(\bN)$.
%MW \end{notation}

\begin{definition}[$\H$, $\Hfin$, point processes of spindles]\label{def:assemblage_m}
 Let $\Hfin\subset \cNRE$ denote the set of all counting measures $N$ on $[0,\infty)\times \Exc$ with the following additional properties:
 \begin{enumerate}[label=(\roman*), ref=(\roman*)]
  \item $N\big( \{t\}\times\Exc \big)\leq 1$ for every $t\in [0,\infty)$,
  \item $N\big( [0,t]\times\{f\in\Exc\colon \life(f) > z\} \big) < \infty$ for every $t,z > 0$, 
  \item the length of $N$, defined in \eqref{eq:assemblage:len_def} %MW to be $\len(N)$ in the sense of Definition \ref{def:JCCP}, 
    is finite and the convergence in (\ref{eq:JCCP_def}) holds uniformly in $t\in[0,\len(N)]$.
 \end{enumerate}
 
 We define $\H\subset\cNRE$ by saying $N\in\H$ if and only if the restriction $\restrict{N}{[0,t]}$ of $N$ to $[0,t]\times\Exc$ is in $\Hfin$ for every $t>0$; 
 here, we abuse notation and consider $\restrict{N}{[0,t]}$ as a measure on $[0,\infty)\times\Exc$ that equals $N$ on $[0,t]\times\Exc$ and vanishes on $(t,\infty)\times\Exc$. %MW
 We call the members of $\Hfin$ and $\H$ \emph{point processes of spindles}. We denote by $\SH$ and $\cHfin$ the restrictions of $\Sigma\left(\mathcal{N}\!\left([0,\infty)\!\times\!\Exc\right)\right)$ to subsets of $\H$ and $\Hfin$.
\end{definition}

\begin{proposition}\label{prop:JCCP_meas}
 The map $\xi\colon \H \to \cD$ specified in Definition \ref{def:assemblage_m} and %MW \ref{def:JCCP} 
 \eqref{eq:JCCP_def} %MW
 is well-defined and measurable, where $\cD$ is the Skorokhod space of real-valued \cadlag\ functions $g\colon[0,\infty)\rightarrow\BR$.
\end{proposition}

\begin{proof}
 %Definition \ref{def:assemblage_m} (iii) ensures that for $N\in\H$, the convergence to $\xi_N(t)$ in \eqref{eq:JCCP_def} holds uniformly in $t$ for any bounded interval. Uniform limits preserve the \cadlag\ property, so $\xi(N)$ exists and is \cadlag. By definition of $\SH$, the function $N\mapsto\xi_N(t)$ is measurable for each $t\geq 0$. By \cite[Theorem 14.5]{Billingsley}, the Borel $\sigma$-algebra $\ScD$ associated with the Skorokhod topology is generated by evaluation maps $g\mapsto g(t)$, $t\ge 0$, so we conclude that $\xi$ is measurable.
 This follows from definitions and an appeal to \cite[Theorem 14.5]{Billingsley} concerning the Skorokhod topology.
\end{proof}

Most of the constructions in this paper begin with a point process $N\in\H$ and from there obtain a scaffolding $X=\xi(N)$. However, it is useful to be able to go in the other direction, to begin with a scaffolding $X$ and to define a point process $N\in\H$ by marking the jumps of $X$ with continuous excursions (which we call spindles, see Section \ref{sec:BESQ}). %In this setting, Proposition \ref{prop:marked_PRM} has the following consequence.

\begin{proposition}[The \PRM\ of spindles via marking jumps]\label{prop:marking_jumps}
 Let $\bX$ denote a \StableA\ process with Laplace exponent as in \eqref{eq:JCCP:Laplace}. Let $\bM = \sum_{t\geq 0\colon\Delta \bX(t)>0}\Dirac{t,\Delta\bX(t)}$. 
 Use the marking kernel $x\mapsto\mBxc(\,\cdot\;|\,\zeta=x)$ to mark each point $(t,\Delta\bX(t))$ of $\bM$ by a spindle $f_t$ with length $\life(f_t)=\Delta\bX(t)$.
 Then $\bN := \sum_{t\geq 0\colon\Delta\bX(t)>0}\Dirac{t,f_t}$ is a \PRM[\Leb\otimes\mBxc] and $\bX = \xi(\bN)$.
\end{proposition}

\begin{comment}
\cite[Proposition 9.1.XII]{DaleyVereJones2}.
 Let $\bX$ denote a \StableA\ process with Laplace exponent as in \eqref{eq:JCCP:Laplace}. Let $\bM = \sum_{j\geq 1}\Dirac{t_j,Z_j}$ denote the associated \PRM\ of jumps of $\bX$ with $((t_j,Z_j),\ j\geq 1)$ a measurable enumeration of points of $\bM$, as in \cite[Proposition 9.1.XII]{DaleyVereJones2}. Let $(f_j,\ j\geq 1)$ be conditionally independent given $((t_j,Z_j),\ j\geq 1)$, with respective conditional laws $f_j \sim \mBxc(\cdot\ |\ \life = Z_j)$ for each $j$. Then $\bN := \sum_{j\geq 1}\Dirac{t_j,f_j}$ is a \PRM[\Leb\otimes\mBxc] and $\bX = \xi(\bN)$.
\end{comment}

\begin{proof} 
 Since $\bX$ is a L\'evy process, $\bM$ is a \PRM. By \eqref{eq:JCCP:Laplace}, its intensity is $\mBxc(\life\in\cdot\,)$. 
 It is well-known that marking constructions like that above result in \PRM s; see \cite[Proposition 6.4.VI]{DaleyVereJones1}. Thus, $\bN$ is a \PRM. 
%By Proposition \ref{prop:stable_JCCP}, the L\'evy measure associated with $\bX$ is $\mBxc(\life\in\cdot\,)$. 
 Since $\mBxc(\,\cdot\;|\,\life)$ is a disintegration of $\mBxc$,
 %it follows from the defining property \eqref{eq:scl_ker:integration} of disintegrations 
 we conclude that $\bN$ has intensity $\Leb\otimes\mBxc$.
\end{proof}

%MW \begin{definition}\label{def:LT}
 For a \cadlag\ function $g\colon [0,\infty) \to\BR$, a bivariate measurable function $(y,t)\mapsto \ell^y_g(t)$ from $\BR\times [0,\infty]$ to $[0,\infty]$, is an \emph{(occupation density) local time} for $g$ if $t\mapsto\ell^y_g(t)$ is increasing for all $y\in\BR$ and if for every bounded and measurable $h\colon \BR\to[0,\infty)$,
 \begin{equation}
  \int_{-\infty}^{\infty} h(y)\ell^y_g(t)dy = \int_0^t h\big(g(s)\big)ds.\label{eq:LT_int_ident}
 \end{equation}
% the \emph{(occupation density) local time} at level $y$, up to time $t$, is
 We call $t$ the \emph{time parameter} and $y$ the \emph{space parameter} and say $\ell^y_g(t)$ is the local time of $g$ at level $y$, up to time $t$.
%MW \end{definition}

\begin{theorem}[Boylan \cite{Boylan64}, equations (4.4) and (4.5)]\label{thm:Boylan} As stable process 
 $\bX\sim\StableA$ has an a.s.\ unique jointly continuous local time process $\ell_\bX=(\ell_\bX^y(t);\;y\in\BR,\,t\geq 0)$.
 Moreover, for every $\theta\in (0,\alpha/(2+\alpha))$, $\theta'\in (0,\alpha/2)$, and each bounded space-time rectangle $R$, the restriction of $(y,t)\mapsto\ell^y_\bX(t)$ to $R$ is uniformly H\"older-$\theta$ in the time coordinate and uniformly H\"older-$\theta'$ in the space coordinate.
\end{theorem}
We denote the \em %MW(right-continuous) 
inverse local time \em by $\tau^y_\bX(s):=\inf\{t\!\ge\! 0\colon\ell_\bX^y(t)\!>\!s\}$, $s\!\ge\! 0$.

\section{$\dI$-path-continuity of $(\nu,T)$-IP-evolutions}
\label{sec:cont}

\subsection{The skewer map}

We now make a slight modification to Definition \ref{def:skewer} of the aggregate mass process, with the aim of having it apply nicely when some spindles are broken.

\begin{definition}\label{def:skewer_2}\!\!\!\!
 The \emph{aggregate mass process} of $N\!\in\!\H$ at level $y\!\in\!\BR$ is
 $$M^y_{N,\xi(N)}(t) := \int_{[0,t]\times\Exc} \!\!\max\!\Big\{ f\big((y-\xi_N(u\minus))\minus\big), f\big(y-\xi_N(u\minus)\big)\Big\}dN(u,f)$$
 for $t\geq 0$. We leave the definition of the skewer map unchanged but abbreviate it $\skewer(y,N) := \skewer(y,N,\xi(N))$, so that
 $$\skewer(y,N) = \left\{\left(M^y_{N}(t-),M^y_{N}(t)\right)\colon t\geq 0,\, M^y_{N}(t-) < M^y_{N}(t)\right\}.$$
 We abbreviate $M^y_N(t)\!:=\!M^y_{N,\xi(N)}(t)$ and $\skewerP(N)\!:=\!\skewerP(N,\xi(N))\!=\!\big(\skewer(y,N),\,y\!\geq\! 0\big)$.
\end{definition}

Recall the inverse local time $\big(\tau^y_\bX(s),\ s\geq 0\big)$ of $\bX=\xi(\bN)$ at level $y$. 

\begin{proposition}[Aggregate mass]\label{prop:agg_mass_subord} Let $\bN$ be a \PRM[\Leb\otimes\mBxc], where $\nu$ is as in \eqref{eqn:generalnu}. Then  
 $\big(M^y_{\bN}\circ\tau^y_{\bX}(s) -M^y_{\bN}\circ\tau^y_{\bX}(0),\ s\geq 0\big)$ is a \Stable[\alpha/q] subordinator 
 with Laplace exponent $\Phi(\lambda) = \lambda^{\alpha/q}$, for each fixed $y\in\BR$.
\end{proposition}

\begin{proof}
 This was proved in \cite[Proposition 8(i)]{Paper0}. See also \cite[Section 6.4]{Paper0} to show that our choice of intensity $c_\nu$ in
 \eqref{eqn:generalnu} is such that $\Phi(\lambda)=k\lambda^{\alpha/q}$ has $k=1$ here. 
\end{proof}

\begin{theorem}[Scaffolding local time equals skewer diversity everywhere; Theorem 1 of \cite{Paper0}]\label{thm:LT_property_all_levels}
 Let $\bN$ be a \PRM[\Leb\otimes\mBxc] and $\bX = \xi(\bN)$, where $\nu$ is as in \eqref{eqn:generalnu}. Suppose that $\bff\sim\nu(\,\cdot\,|\,\zeta=1)$ is $\theta$-H\"older 
 for some $\theta\in(0,q)$ and that the H\"older constant $D_\theta=\sum_{0<x<y<1}|\bff(y)-\bff(x)|/|y-x|^\theta$ has moments of all orders. Then
 there is an event of probability 1 on which, for every $y\in\BR$ and $s\geq 0$, the partition $\beta^y_s := \skewer\big(y,\,\restrict{\bN}{[0,\tau_\bX^y(s)]}\big)$ 
 possesses the $\alpha/q$-diversity property of \eqref{eq:IPLT}, and
 \begin{equation}
  \ell^y_\bX(t) = \IPLT^{\alpha/q}_{\beta^y_s}\left(M^y_{\bN}(t)\right)\qquad\mbox{for all }t\in[0,\tau^y_\bX(s)],\ s\ge 0,\ y\in\BR. 
%= \Gamma(1-\alpha/q)\lim_{h\downto 0} h^{\alpha/q}\int \cf\{m^0(N) > h,\tau^y(r)\le t\}d\bF^y(r,N).
\label{eq:inft_skewer_LT_cnvgc}
 \end{equation} 
\end{theorem}

The strength of the preceding result is that it holds a.s.\ simultaneously at every level $y$, so $\skewerP(\restrict{\bN}{[0,\tau^y(s)]})$ is
an $\IPspace$-valued process for $\IPspace:=\IPspace_{\alpha/q}$. Proposition \ref{prop:agg_mass_subord} implies the weaker result that \eqref{eq:inft_skewer_LT_cnvgc} holds a.s.\ for every $s\geq 0$, for any fixed $y$. Recall Definition \ref{def:assemblage_m} of the measurable spaces $(\H,\SH)$ and $(\Hfin,\SHfin)$. We are interested in diffusions on $(\IPspace,\dI)$. To that end we require measures $N\in \Hfin$ for which $\skewerP(N)$ is path-continuous in $(\IPspace,\dI)$. 

\begin{definition}[$\Hs,\ \Hfins$]\label{def:domain_for_skewer}
 Let $\Hs$ denote the set of all $N\in\H$ with the following additional properties.
 \begin{enumerate}[label=(\roman*), ref=(\roman*)]
  \item The aggregate mass $M^y_N(t)$ is finite for every $y\in\BR$ and $t\geq 0$.\label{item:d_f_s:fin}
  \item The occupation density local time $(\ell^y_{\xi(N)}(t))$ is bi-continuous on $t\geq 0$, $y\in (\inf_u \xi_N(u), \sup_u\xi_N(u))$, and for every $(y,t)$ in this range,\label{item:d_f_s:div}
   \begin{equation}
    %\IPLT_{\skewer(y,\restrict{N}{[0,t]})}\big( M^y_N(u) \big) = \ell^y_{\xi(N)}(u), \qquad u\in [0,t].\label{eq:div_LT_condition}
    \IPLT_{\skewer(y,\restrict{N}{[0,t]})}(\infty) = \ell^y_{\xi(N)}(t).\label{eq:div_LT_condition}
   \end{equation}
  \item For every $t>0$, the skewer process $\skewerP\left(\restrict{N}{[0,t]}\right)$ is continuous in $(\IPspace,\dI)$.\label{item:d_f_s:cts}
 \end{enumerate}
 
 Let $\Hfins := \Hfin\cap\Hs$. Let $\Sigma(\Hs) := \{A\cap\Hs\colon A\in\SH\}$, and correspondingly define $\Sigma(\Hfins)$.
\end{definition}

In condition (ii) above, we restrict $y$ away from boundary values because \eqref{eq:div_LT_condition} can fail at $y = 0$ for the point processes $\bN_{\beta}$ constructed in a clade construction; see Definition \ref{constr:type-1}.
%$N$ of the form $\cutoffG{0}{N'}$, where the latter is as defined in \eqref{eq:cutoff_process_def}.

\begin{proposition}\label{prop:skewer_measurable} The map 
 $\skewerP$ of Definitions \ref{def:skewer_2} and \ref{def:domain_for_skewer} is measurable from $(\Hfins,\SHfins)$ to the space $\mathcal{C}([0,\infty),\IPspace)$ of continuous functions, under the Borel $\sigma$-algebra generated by uniform convergence.
\end{proposition}

We prove this proposition in \ref{SuppMeas}.

\subsection{Proof of Theorem \ref{thm:diffusion_0}}
\label{sec:type-1:PRM}

%Let $\bN$ be a \PRM[\Leb\otimes\mBxc] living on a probability space $(\Omega,\cA,\Pr)$. We continue to use the notation of the first paragraph of Section \ref{sec:biclade_PRM} for objects related to $\bN$. Let $(\ol\cF_t,\,t\ge0)$ and $(\ol\cF^y,\,y\ge0)$ denote $\Pr$-completions of the time- and level-filtrations on $(\Omega,\cA)$ generated by $\bN$, as in Definition \ref{def:filtrations}. That is, these are $\Pr$-completions of the pullbacks, via $\bN\colon \Omega\to\H$, of the time- and level-filtrations on $\H$. 
Let $\nu$ be as in \eqref{eqn:generalnu} and assume there is $\theta\in(0,q-\alpha)$ such that $\bff\sim\nu(\,\cdot\,|\,\zeta=1)$ has a $\theta$-H\"older
constant with moments of all orders. Let $\bN$ be a \PRM[\Leb\otimes\mBxc] living on a probability space $(\Omega,\cA,\Pr)$. We define $\tdN := \restrict{\bN}{[0,T)}$, where $T\in(0,\infty)$. We take ``twiddled versions'' of our earlier notation to denote the corresponding objects for $\tdN$; for instance, $\tdl$ will denote the jointly H\"older continuous version of the local time process associated with $\tdX := \xi(\tdN)$. 
It follows from Proposition \ref{prop:agg_mass_subord} that for each $y\geq 0$ we have $\td\beta^y := \skewer(y,\tdN)\in\IPspace$ almost surely. I.e.\ $\td\beta^y$ is almost surely a finite interval partition with the diversity property. %For the purpose of the following, for $y\geq 0$ we define the level $y$ \emph{local time filtration} $(\ol\cG^{\leq y}_s,\ s\geq 0)$ via
%\begin{equation}\label{eq:filtration_LT_def}
% \ol\cG^{\leq y}_s := \ol\cF^y \cap \ol\cF_{\tau^y(s)} \qquad \text{for }y,s\geq 0.
%\end{equation}
 %\ol\cG^{\leq y}_s := \ol\cF^y \cap \bigcap_{r>s}\ol\cF_{\tau^y(r)} \qquad \text{for }y,s\geq 0.

\begin{lemma}[Proposition 6 of \cite{Paper0}]\label{lem:spindle_piles}
 For each spindle $(t,f)$ of $\tdN$, let $\bar f(y) := f(y-\tdX(t-))$, $y\in \BR$. These translated spindles can a.s.\ be partitioned into sequences $(g^n_j,\,j\geq 1)$, for $n\geq 1$, in such a way that in each sequence $(g^n_j,\,j\geq 1)$:
 (i) the spindles have disjoint support, and (ii) they are uniformly H\"older-$\theta$ with some constants $D_n$. Furthermore, we may choose $D_n$ so that
 $\sum_{n\ge1}D_n < \infty$ a.s..
\end{lemma}

\begin{corollary}\label{cor:PRM:type-1_wd}
 It is a.s.\ the case that $M^y_{\tdN}(T)$ is finite for all $y\in\BR$. Moreover, $(\td\beta^y,\ y\in\BR)$ a.s.\ takes values in $\IPspace$ for all $y\in\BR$.
\end{corollary}

\begin{proof}
 It suffices to check the first assertion: by Theorem \ref{thm:LT_property_all_levels}, this implies that $\td\beta^y\in\IPspace$ for all $y$ simultaneously, almost surely. For each $n\geq 1$ let the $(g^n_j,\ j\geq 1)$ and $D_n$ be as in Lemma \ref{lem:spindle_piles}. Definition \ref{def:skewer_2} gives\vspace{-0.2cm}
 \begin{equation}
  M^y_{\tdN}(T) = \sum_{n\geq 1}\sum_{j\geq 1} g^n_j(y) \qquad \text{for }y\in\BR.\label{eq:agg_mass_sum_piles}
 \end{equation}
 Let $g^n := \sum_{j\geq 1} g^n_j$ for each $n\geq 1$. Since the $g^n_j$ in each sequence have disjoint support, $g^n$ is H\"older-$\theta$ with constant $D_n$. Proposition \ref{prop:agg_mass_subord} implies that $M^0_{\tdN}(T)$ is a.s.\ finite. Thus, by \eqref{eq:agg_mass_sum_piles}, $y\mapsto M^{y}_{\tdN}(T)$ is almost surely H\"older-$\theta$ with constant bounded by $\sum_{n\geq 1}D_n$.
\end{proof}

Recall Definition \ref{def:domain_for_skewer} of $\Hfins$, the subspace of $\H$ on which the skewer map measurably produces a continuously evolving interval partition. The following result implies Theorem \ref{thm:diffusion_0}.

\begin{proposition}\label{prop:PRM:cts_skewer}
 Suppose that $\theta<\min\{\alpha/2,q-\alpha\}$. Then $\tdN$ and $\bN$ a.s.\ belong to $\Hfins$ and $\Hs$, respectively. In particular, $(\td\beta^y,\,y\ge0)$ is a.s.\ H\"older-$\theta$ in $(\IPspace,\dI)$.%\pagebreak
\end{proposition}

%NOTE: Holder constant moments??

\begin{proof}
 We have already shown in Theorem \ref{thm:LT_property_all_levels} and Corollary \ref{cor:PRM:type-1_wd} that $\tdN$ (resp.\ $\bN$) satisfies the first two conditions in Definition \ref{def:domain_for_skewer} for membership in $\Hfins$ (resp.\ $\Hs$). It remains to prove the claimed H\"older continuity.
 
 For $n,j\geq 1$, let $g^n_j$ and $D_n$ be as in Lemma \ref{lem:spindle_piles} and let $g^n := \sum_{j\geq 1} g^n_j$. Since $\tdN$ is stopped at an a.s.\ finite time, the path of $\tdX$ lies within a random bounded space-time rectangle. We restrict our attention to the intersection of the almost sure events posited by Lemma \ref{lem:spindle_piles}, Corollary \ref{cor:PRM:type-1_wd}, and Theorem \ref{thm:Boylan}: that the H\"older constants $D_n$ are summable, the process $(\td\beta^y)$ lies in $\IPspace$, and the local times $(\tdl^y(t))$ are uniformly H\"older-$\theta$ in level and continuous in time. Let\vspace{-0.2cm}
 \begin{equation*}
  C := \sup_{-\infty<y<z<\infty}\sup_{t\in [0,T]}\frac{|\tdl^z(t) - \tdl^y(t)|}{|z-y|^{\theta}} \qquad \text{and} \qquad D := \sum_{n=1}^\infty D_n.\vspace{-0.1cm}
 \end{equation*}
 
 Fix $y,z\in \BR$ with $y<z$. Let $A := \{n\geq 1\colon \Restrict{g^n}{[y,z]} > 0\}$. That is, $A$ is the set of indices $n$ for which a single spindle in the sequence $(g^n_j)_{j\geq 1}$ survives the interval $[y,z]$. For each $n\in A$, let $t_n$ denote the time at which that particular spindle arises as a point in $\tdN$. Recall Definition \ref{def:IP:metric} of $\dI$ based on correspondences between interval partitions. Consider the correspondence from $\td\beta^y$ to $\td\beta^z$ that, for each $n\in A$, pairs the block $U_n^y\in\td\beta^y$ with $U_n^z\in\td\beta^z$, where there are the blocks corresponding to $g^n$. This is indeed a correspondence, respecting order in the two interval partitions, since each paired block corresponds to the same spindle as its partner.
 
 Note that for $n\!\notin\! A$ there is some $x\!\in\! [y,z]$ for which $g^n(x) \!=\! 0$. By its H\"older continuity, both $g^n(y)$ and $g^n(z)$ are bounded by $D_n(z-y)^{\theta}$. Thus,\vspace{-0.1cm}
 \begin{equation*}
 \begin{split}
  \sum_{n\in A} |g^n(z) - g^n(y)| + \max\left\{\sum_{n\notin A}g^n(y),\sum_{n\notin A}g^n(z)\right\} &\leq \sum_{n\ge 1} D_n(z-y)^{\theta},\\
  \sup_{n\in A}\left|\IPLT_{\td\beta^z}(U_n^z) - \IPLT_{\td\beta^y}(U_n^y)\right| = \sup_{n\in A}\left|\tdl^z(t_n) - \tdl^y(t_n)\right| &\leq C(z-y)^{\theta},\\
  \text{and} \quad \left|\IPLT_{\td\beta^z}(\infty) - \IPLT_{\td\beta^y}(\infty)\right| = \left|\tdl^z(T) - \tdl^y(T)\right| &\leq C(z-y)^{\theta}.\vspace{-0.2cm}
 \end{split}
 \end{equation*}
 By Definition \ref{def:IP:metric} of $\dI$, we conclude that $(\td\beta^y,\ y\in\BR)$ is H\"older-$\theta$ with H\"older constant bounded by $\max\{C,D\}$.
\end{proof}

%\pagebreak

\section{Excursion theory for scaffolding and spindles} 

A $\nu_{q,c}^{(-2\alpha)}$-IP-evolution starting from $\beta\in\IPspace$ can be obtained by concatenating 
excursions of scaffolding with spindles. To establish the Markov property of a $\nu_{q,c}^{(-2\alpha)}$-IP-evolution, we will decompose the 
scaffolding and spindles at the corresponding level. In this section, we first formalise the definition of $\nu_{q,c}^{(-2\alpha)}$-IP-evolutions 
and then study excursion theory for the \StableA\ scaffolding, enriched by the spindles marking the jumps of the scaffolding.

\subsection{Definition of $\mBxc_{q,c}^{(-2\alpha)}$-IP-evolutions}

Fix $\alpha\in(0,1)$, $q>\alpha$ and $c>0$. We again abbreviate notation to $\IPspace:=\IPspace_{\alpha/q}$.

%MW \begin{definition}\label{def:concat}
 Let $(N_a)_{a\in\mathcal{A}}$ denote a family of elements of $\H$ indexed by a totally ordered set $(\cA,\preceq)$. For the purpose of the following, set
 \begin{equation}\label{eq:length_partial_sums}
  S(a) := \sum_{b\preceq a}\len(N_b) \quad \text{and} \quad S(a-) := \sum_{b\prec a}\len(N_b) \quad \text{for each }a\in \mathcal{A}.
 \end{equation}
 If $S(a-) < \infty$ for every $a\in \mathcal{A}$ and if for every consecutive $a\prec b$ in $\mathcal{A}$ we have $N_a(\{\len(N_a)\}\times\Exc) + N_b(\{0\}\times\Exc)\leq 1$, then we define the \emph{concatenation} of $(N_a)_{a\in\mathcal{A}}$ to be the counting measure
 \begin{equation}\label{eqn:concat}%MW
   \Concat_{a\in\mathcal{A}}N_a := \sum_{a\in\mathcal{A}}\int\Dirac{S(a-)+t,f}dN_a(t,f).
 \end{equation}%MW
%MW \end{definition}

We now formalize the construction of IP-evolutions started from any $\beta\!\in\!\IPspace$.

\begin{definition}[$\nu_{q,c}^{(-2\alpha)}$-IP-evolution, $\Pr^{\alpha,q,c}_{\beta}$]\label{constr:type-1}
 Take $\beta\in\IPspace$. If $\beta = \emptyset$ then $\bN_{\beta} := 0$. Otherwise, for each $U\in\beta$ we carry out the following construction independently. Let $\bN$ denote a \PRM[\Leb\otimes\mBxc_{q,c}^{(-2\alpha)}], let $\bff$ be an independent $(\alpha,q,c)$-block diffusion started from $\Leb(U)$ and absorbed at 0, and consider the hitting time $T := \inf\{t>0\colon \xi_{\bN}(t) = -\life(\bff)\}$. Let $\bN_U := \Dirac{0,\bff}+\restrict{\bN}{[0,T]}$. If $\sum_{U\in\beta}\len(\bN_U)<\infty$, let $\bN_{\beta} := \ConcatIL_{U\in\beta}\bN_U$. 
 Recalling Definition \ref{def:skewer_2}, %MW 
 we call $(\beta^y,\,y\ge0) := \skewerP(\bN_{\beta})$ %MW
 a \emph{$\nu_{q,c}^{(-2\alpha)}$-IP-evolution} starting from $\beta$.
\end{definition}

  We write $\Pr^{\alpha,q,c}_{\beta}$ to denote the law of $\bN_{\beta}$ on $\Hfin$. For probability distributions $\mu$ on $\IPspace$, we write $\Pr^{\alpha,q,c}_{\mu} := \int \Pr^{\alpha,q,c}_{\beta}\mu(d\beta)$ to denote the $\mu$-mixture of the laws $\Pr^{\alpha,q,c}_{\beta}$. 

We remark that the measurability of the map $\omega\mapsto \beta^y(\omega)$ is a subtle point.  Propositions \ref{prop:skewer_measurable} and \ref{prop:PRM:cts_skewer} prove this when the initial state comes from a L\'evy process as in Section \ref{sec:cont}.  This is also sufficient for Section \ref{sec:easyinitialdata}. Arbitrary initial data is addressed in Section \ref{sec:arbitraryinitialdata}, where Lemma \ref{lem:finite_survivors}, Proposition \ref{prop:type-1:LT_diversity}, and \ref{SuppMeas} Lemma 8 show that it is measurable when $q=c=1$. In this case, we confirm in Proposition \ref{prop:clade_lengths_summable} that 
$\sum_{U\in\beta}\len(\bN_U)\!<\!\infty$ a.s.\ for each $\beta\in\IPspace$ and that $\beta\mapsto\Pr^{\alpha,q,c}_{\beta}$ is a stochastic kernel. 
In Lemma \ref{lem:type-1:wd} we show that $\beta^y\in\IPspace$ for every $y\geq 0$, a.s.. In Proposition \ref{prop:type-1:cts}, we show 
that $\nu_{\tt BESQ}^{(-2\alpha)}$-IP-evolutions are a.s.\ continuous. This is extended to $\nu_{q,c}^{(-2\alpha)}$-IP-evolutions in Section \ref{IPevolgeneral}.%, on a 
%state space that is smaller when $q>1$.

\subsection{Scaffolding levels: excursion theory for \StableA\ processes}
\label{sec:exc}

Excursion theory for Markov processes was first developed by It\^o \cite{Ito72}. Bertoin \cite[Chs.\ IV-V]{BertoinLevy} offers a nice treatment in the setting of L\'evy processes. %Greenwood and Pitman \cite{GreePitm80} give a variant approach with a more discrete flavor, through nested arrays of excursions.
% Recall Definitions \ref{def:JCCP} and \ref{def:assemblage_m} of $\xi$ and $\H$.

%MW \begin{definition}\label{def:stable_exc}
 We define two subsets of the space $\cD$ of c\`adl\`ag functions $g\colon[0,\infty)\rightarrow\BR$:
 \begin{align}\label{eq:stable_exc_space_def}
  \DS &:= \xi(\H)=\left\{\xi(N)\colon N\in\H\right\},\\[0.2cm]
  \DSxc &:= \left\{\xi(N)\colon N\in\Hfin,\ \xi_N\neq 0 \text{ on }(0,\len(N)),\ \xi_N(\len(N))=0\right\}.%\len(N) = \inf\{t > 0\ :\ \xi_N(t) = 0\}\right\}.
 \end{align}
 We take $\SDS$ and $\SDSxc$ to denote the $\sigma$-algebras on these spaces generated by the evaluation maps. We say that members of $\DS$ are \emph{\StableA-like processes} and members of $\DSxc$ are \emph{\StableA-like excursions}.
 
 For $g=\xi(N)\in\DS$, we define the \emph{length} of $g$ to be $\len(g)=\len(N)$. We will use the convention $g(0-)=0$.  
% \begin{align}
%  \len(g) := \sup\{t > 0\ :\ g(t) \neq 0\}.\label{eq:len_D_def}
% \end{align}
%MW \end{definition}

%Note that if $g\in \DSxc$ then $\len(g) = \sup\{t>0\colon g(t)\neq 0\}.$ 

%\begin{proposition}\label{prop:stable_Skor_meas}
% $\DS$ and $\DSxc$ are Borel sets in $(\cD,d_{\cD})$.
%\end{proposition}

%This is a straightforward exercise in topology and measure theory, starting by expressing the conditions on $N$ of Definition \ref{def:assemblage_m} in terms of the \cadlag\ function $g=\xi(N)$.

%MW \begin{definition}\label{def:inverse_LT}
 Fix $y\in\BR$ and $g\in\DS$ and recall \eqref{eq:LT_int_ident}. %MW
 If the level $y$ local time associated with $g$ exists, in the sense that for all $t\ge 0$ the limits 
 \begin{equation}
  \ell^y_g(t)\! =\! \lim_{h\downto 0} \frac{1}{h}\!\int_0^t\! \cf\{y\!-\!h\!<\!g(s)\!<\!y\}ds=\lim_{h\downto 0} \frac{1}{h}\!\int_0^t\! \cf\{y\!<\!g(s)\!<\!y\!+\!h\}ds \label{eq:LT_def}
 \end{equation}
exist, the \emph{inverse local time} is defined as $\tau^y_g(s) := \inf\{t \geq 0\colon\ell^y_g(t) > s\}$ for $s\geq 0$. In this notation we may replace $g$ with $N\in\H$ to denote the corresponding object with $g = \xi(N)$.
%MW \end{definition}

In the sequel, we will suppress the `$g$' in the above notations when we refer to these objects applied to $g = \bX$, where $\bX$ is the 
\StableA\ scaffolding $\xi(\bN)$ of Proposition \ref{prop:stable_JCCP}. Let $\ell$ denote the jointly H\"older continuous version of local time 
specified in Theorem \ref{thm:Boylan}.
 Note that for $y\in\BR$ fixed, $T^y = \tau^y(0)$ a.s., but this is not simultaneously the case for all $y\in\BR$.% Let $V^y$ denote the interval partition specified by the excursions of $\bX$ about level $y$:

%MW \begin{definition}\label{def:shifted_restriction}%\label{def:stable:exc_PRM} 
  Let $[a,b]\subset[0,\infty)$ be an interval, $g\colon[0,\infty)\to\cS$ a function and $N\in\cN([0,\infty)\times\cS)$ a counting measure. We define \emph{shifted restrictions} by setting for $x\in[0,\infty)$, $B\in\ScS$ and $I\subset[0,\infty)$ Borel
 \begin{equation}\label{eq:shift_restrict_defn}
  \ShiftRestrict{g}{[a,b]}(x):=\Restrict{g}{[a,b]}(x\!+\!a),\ \ \mbox{and}\ \ \ShiftRestrict{N}{[a,b]}(I\!\times\! B):=\Restrict{N}{[a,b]}((I\!+\!a)\!\times\! B).
 \end{equation}
%\ShiftRestrict{f}{[a,b]}(x):=f(x+a)\cf(x\in[0,b-a]),\quad\mbox{and}\quad\ShiftRestrict{N}{[a,b]}(I\times B):=N(((I+a)\cap[a,b])\times B).
 We use similar notation with open/half-open intervals $(a,b)$, $[a,b)$, and $(a,b]$. As in Definition \ref{def:assemblage_m}, we abuse notation and consider $\ShiftRestrict{g}{[a,b]}$ as a 
 function on $[0,\infty)$ that vanishes on $(b-a,\infty)$, and $\ShiftRestrict{N}{[a,b]}$ as a measure on $[0,\infty)\times\Exc$.

%  Adapting the definition of $d_\cD^A$ of Proposition \ref{prop:E_bdedly_finite}, we equip $\DSxc\setminus\{0\}$ with the metric $d_\cD^A(f,g)=d_\cD(f,g)+|A(f)^{-1}-A(g)^{-1}|$, where $A(g):=\sup\{|g(t)|,0\le t\le\len(g)\}$ with $\len(g)$ as in Definition \ref{def:stable_exc}.
  We denote by $V^y$ the set of intervals of complete excursions of $\bX$ about level $y$; this is defined formally in Appendix \ref{app:exc_intervals}. 
  We define an excursion counting measure and an associated intensity
 \begin{align}
  \mathbf{G}^y &:= \sum_{[a,b]\in V^y} \Dirac{\ell^y(a),\ShiftRestrict{\bX}{[a,b]}};\label{eq:stable:exc_PRM}\\
  \mSxc(B) &:= \EV[\mathbf{G}^0([0,1]\times B)] \quad \text{for }B\in\Sigma(\DSxc\setminus\{0\}).\label{eq:stable_exc_measure}
 \end{align}
%MW \end{definition}

\begin{proposition}[%Blumenthal and Getoor \cite{BlumenthalGetoor}; 
see e.g.\ \cite{BertoinLevy} Theorems IV.8 and IV.10]\label{thm:excursion_PRM} %\label{thm:inverse_LT}
 For each $y\in\BR$, the inverse local-time $\tau^y(s)$ of the \StableA\ process $\bX$ is a.s.\ finite for every $s\ge 0$, and $\mathbf{G}^y$ is a \PRM[\Leb\otimes\mSxc] on $[0,\infty)\times\DSxc$.
\end{proposition}

This proposition has the consequence that, for fixed $y\in\BR$,\vspace{-0.1cm}
\begin{align}
 V^y = \big\{[\tau^y(t-),\tau^y(t)]\colon\ t>0,\ \tau^y(t-)\neq \tau^y(t)\big\}\quad\text{a.s..}\label{eq:inv_LT_intervals}\vspace{-0.1cm}
\end{align}

%\begin{definition}\label{def:stable:xforms}
 Let $\reverseincr$ denote the increment-reversal involution on excursions $g\in \DSxc$:\vspace{-0.1cm}
 \begin{align}
  \reverseincr(g) = \big(-g\big((\len(g)-t)-\big),\ t\in [0,\len(g)]\big).\label{eq:stable:reversal_def}\vspace{-0.1cm}
 \end{align}
 Let $\scaleS$ denote the \StableA\ scaling map from $(0,\infty)\times\DS$ to $\DS$:\vspace{-0.1cm}
 \begin{align}
  c \scaleS g := \left(cg\left(c^{-1-\alpha}t\right),\ t\in [0,c^{1+\alpha}\len(g)]\right).\label{eq:stable:scaling_def}\vspace{-0.1cm}
 \end{align}
%\end{definition}

\begin{lemma}[Invariance of \Stable $(1+\alpha)$ excursions]
\label{lem:stable:invariance}
 For $B\in \SDSxc$ and $c>0$, the \Stable $(1+\alpha)$ excursion measure $\mSxc$ of \eqref{eq:stable_exc_measure} satisfies\vspace{-0.1cm}
 \begin{align}
  \mSxc(\reverseincr(B)) = \mSxc(B)\quad \text{and} \quad \mSxc(\scaleS[c][B]) = c^{-\alpha}\mSxc(B).\label{eq:stable:scaling_inv}\vspace{-0.1cm}
 \end{align}
\end{lemma}
\begin{proof} The increment-reversal invariance was obtained by Getoor and Sharpe \cite[Theorem (4.8)]{GetoShar81}. The scaling invariance follows from the scaling invariance of $\bX$. 
\end{proof}

\subsection{Bi-clades: level filtrations via excursions of scaffolding with spindles}
\label{sec:clade_filtration}

In the preceding section, we have looked at excursions of the \StableA\ scaffolding process. In this section, we consider such excursions with jumps marked by $(\alpha,q,c)$-block evolutions.

%\begin{definition}\label{def:bi-clade}
 We define\vspace{-0.1cm}
 %$$\Hxc{\pm} := \Big\{N\in \Hfin\colon \inf\{t > 0\colon \xi_N(t)=0\} = \len(N)\Big\} = \big\{N\in \Hfin\colon \xi(N)\in\DSxc\big\},$$
 %$$\Hxc{+}	:= \big\{N\in \Hxc{\pm}\colon \inf\nolimits_t\xi_N(t)=0\big\}, \quad \text{and} \quad \Hxc{-}	:= \big\{N\in \Hxc{\pm}\colon \sup\nolimits_t\xi_N(t)=0\big\},$$
 \begin{align*}
  \Hxc{\pm} &:= \big\{N\in \Hfin\colon \xi(N)\in\DSxc\big\},\\
  \Hxc{+}	&:= \big\{N\in \Hxc{\pm}\colon \inf\nolimits_t\xi_N(t)=0\big\},\\
  \text{and}\ \Hxc{-}	&:= \big\{N\in \Hxc{\pm}\colon \sup\nolimits_t\xi_N(t)=0\big\},\vspace{-0.1cm}
 \end{align*}
 where $\Hfin$ is as in Definition \ref{def:assemblage_m}. Let $\Sigma(\Hxc{\pm})$, $\Sigma(\Hxc{+})$, and $\Sigma(\Hxc{-})$ denote the restrictions of $\ScNRE$ to subsets of $\Hxc{\pm}$, $\Hxc{+}$, and $\Hxc{-}$ respectively. We call the members of $\Hxc{+}$ the \emph{clades} and those of $\Hxc{-}$ the \emph{anti-clades}. Members of $\Hxc{\pm}$ are called \emph{bi-clades}.
%\end{definition}

\begin{definition}\label{def:typical_exc}
 We call an excursion $g\in\DSxc$ \emph{typical} if there exists some time $T_0^+(g) \in (0,\len(g))$ such that: (i) $g(t) < 0$ for $t\in (0,T_0^+(g))$, and (ii) $g(t) > 0$ for $t\in [T_0^+(g),\len(g))$. We call $g$ \emph{degenerate} if it is not typical.
\end{definition}

Proposition \ref{prop:nice_level} in Appendix \ref{app:exc_intervals} observes that almost all excursions are typical. A typical excursion in this sense can be decomposed into an initial escape downwards, a jump up across zero, and a final first-passage descent. Correspondingly, as illustrated in Figure \ref{fig:bi-clade_decomp}, a bi-clade $N\in\Hxc{\pm}$ for which $\xi(N)$ is typical may be split into two components around the jump of $\xi(N)$ across zero. The initial component, corresponding to the negative path-segment of $\xi(N)$, is an anti-clade, and the subsequent component, on which $\xi(N)$ is positive, is a clade. In order to make this decomposition, we must break the spindle marking the middle jump of the excursion into two parts, above and below level zero.

\begin{figure}
 \centering
 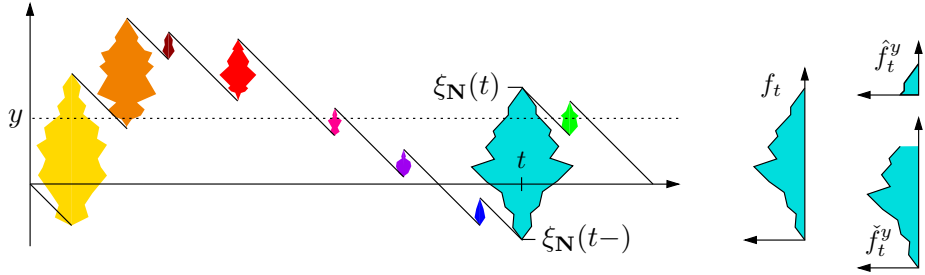
 \caption{Splitting a spindle of some $N\in \H$ about some level $y$, as in \eqref{eq:spindle_split}.\label{fig:spindle_split}}
\end{figure}

If a measure $N\in \H$ has a point $(t,f_t)$ with $y\in (\xi_N(t-),\xi_N(t))$, then we define
\begin{equation}
 \left.\begin{array}{r@{\ :=\ }l}
     \check f_{t}^y(z) & f_{t}(z)\cf\{z \in [0,y-\xi_N(t-))\}\\[0.2cm]
     \text{and}\quad\hat f_{t}^y(z) & f_{t}(y - \xi_N(t-) + z)\cf\{z \in [0,\xi_N(t)-y]\}.
 \end{array}\right.\label{eq:spindle_split}
\end{equation}
This splits the spindle $f_{t}$ into two parts, corresponding to the part of the jump of $\xi(N)$ that goes from $\xi_N(t-)$ up to $y$, and the part extending from $y$ up to $\xi_N(t)$. This is illustrated in Figure \ref{fig:spindle_split}. Following Definition \ref{def:typical_exc}, for $N\in\Hxc{\pm}$ the \emph{crossing time} is
\begin{equation}
 T_0^+(N) := \inf\{t \in (0,\len(N)] \colon \xi_N(t) \geq 0\}.\label{eq:clade:crossing_def}
\end{equation}

Fix $N\in\Hxc{\pm}\setminus(\Hxc{-}\cup\Hxc{+})$. For the purposes of the following definitions, we abbreviate the crossing time $T_0^+ := T_0^+(N)$. We split the bi-clade into anti-clade and clade components, denoted by $(N^-,N^+)$, as follows:
\begin{equation}\label{eq:bi-clade_split}
 N^- := \Restrict{N}{[0,T_0^+)} + \Dirac{T_0^+,\check f_{T_0^+}^0}, \quad N^+ := \Dirac{0,\hat f_{T_0^+}^0} + \ShiftRestrict{N}{(T_0^+,\infty)}.
\end{equation}
We define $(N^-,N^+)$ to be $(N,0)$ if $N\in\Hxc{-}$ or $(0,N)$ if $N\in\Hxc{+}$.
%For $N\in\Hxc{-}$, let $(N^-,N^+) := (N,0)$. For $N\in\Hxc{+}$, let $(N^-,N^+) := (0,N)$.

More generally, we may define scaffolding and spindles \emph{cut off} above and below a level $y\in\BR$. These processes are illustrated in Figure \ref{fig:cutoff}. For the purpose of the following, for $N\in\H$ and $y\in\BR$, let
\begin{equation*}
 \sigma^y_N(t) := \Leb\{u\leq t\colon \xi_N(u)\leq y\} \qquad \text{for }t\ge0.
 %,\qquad \rho^y_N(s) := \sup\{t\ge 0\colon \sigma^y_N(t) \le s\}.
\end{equation*}
In other words, $\sigma^y_N(t)$ is the amount of time that $\xi(N)$ spends below level $y$, up to time $t$. Then for $t\ge 0$,%, while $\rho^y_N$ is its right-continuous right inverse. 
%\begin{equation}
% \cutoffL{y}{\xi(N)}(t) := \xi_N\!\left( \rho^y_N(t) \right) - \min\{y,0\}, \quad \cutoffG{y}{\xi(N)}(t) := \xi_N\!\left( t - \rho^y_N(t-) \right) - \max\{y,0\}.
%\end{equation}
%Correspondingly,
\begin{equation}
\begin{split}
  &\cutoffL{y}{\xi(N)}(t) := \xi_N\big( \sup\{t\ge 0\colon \sigma^y_N(t) \le s\} \big) - \min\{y,0\},\\
  &\cutoffG{y}{\xi(N)}(t) := \xi_N\big( \sup\{t\ge 0\colon t - \sigma^y_N(t) \le s\} \big) - \max\{y,0\},\\
 \label{eq:cutoff_def}
  &\cutoffL{y}{N} := \!\!\!\sum_{\text{points }(t,f_t)\text{ of }N} \!\left(\!\!\begin{array}{r@{\,}l}
 		\cf\big\{y\in (\xi_N(t\minus),\xi_N(t))\big\}&\DiracBig{\sigma^y_N(t),\check f^y_t}\\[3pt]
 		+\; \cf\big\{\xi_N(t) \leq y\big\}&\DiracBig{\sigma^y_N(t),f_t}
 		\end{array}\!\!\right)\!,\\
  &\cutoffG{y}{N} := \!\!\!\sum_{\text{points }(t,f_t)\text{ of }N} \!\left(\!\!\!\begin{array}{r@{\,}l}
 		\cf\big\{y\in (\xi_N(t\minus),\xi_N(t))\big\}&\DiracBig{t\!-\!\sigma^y_N(t),\hat f^y_t}\\[3pt]
 		+\; \cf\big\{\xi_N(t\minus) \geq y\big\}&\DiracBig{t-\sigma^y_N(t),f_t}
 		\end{array}\!\!\!\right)\!.
\end{split} 
\end{equation}

\begin{figure}
 \centering
 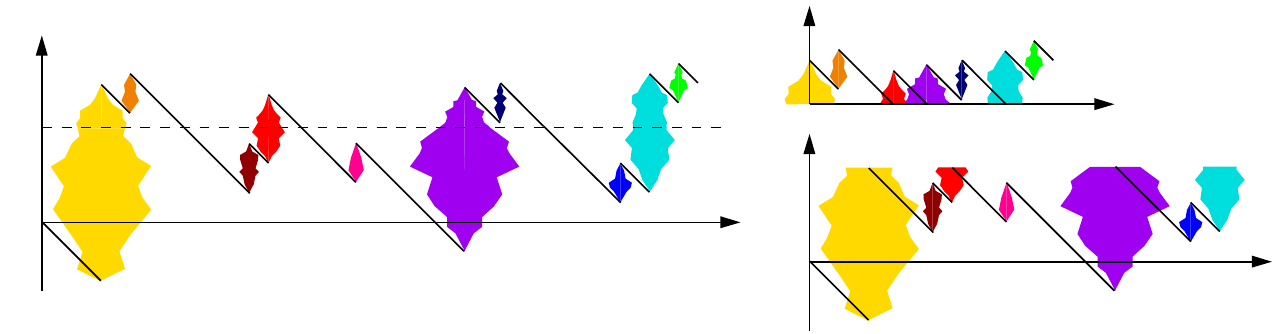\vspace{-.3cm}
 \caption{Left: $(N,\xi(N))$. Right: $\big(\cutoffG{y}{N},\cutoffG{y}{\xi(N)}\big)$ (above) and $\big(\cutoffL{y}{N}, \cutoffL{y}{\xi(N)}\big)$, as in \eqref{eq:cutoff_def}.\label{fig:cutoff}}
\end{figure}

We note the following elementary result.

\begin{lemma}\label{lem:cutoff_meas}
 The four cutoff processes defined in \eqref{eq:cutoff_def} are measurable functions of $N$, for $y$ fixed.
\end{lemma}

Recall Figure \ref{fig:skewer_1} and Definition \ref{def:skewer} of the skewer map. We are ultimately interested in the processes $\big(\skewer(y,\bN_{\beta},\xi(\bN_{\beta})),\ y\geq 0\big)$, where $\bN_{\beta}$ is as in Definition \ref{constr:type-1}. We view such processes as evolving in \emph{level} rather than in \emph{time}, as the parameter $y$ of this process corresponds to values, or levels, in the scaffolding function $\xi(\bN_{\beta})$. From this standpoint, $\cutoffL{y}{N}$ describes the past up to level $y$, and $\cutoffG{y}{N}$ describes the future beyond level $y$. This motivates the following. Throughout, superscripts refer to level whereas subscripts refer to time.

\begin{definition}\label{def:filtrations}
 \begin{enumerate}[label=(\roman*), ref=(\roman*)]
  \item We define the \emph{filtration in level} $(\cF^y)_{y\geq 0}$ on $\H$ to be the least right-continuous filtration under which the maps $N\!\mapsto\! \left(\cutoffL{y}{\xi(N)},\cutoffL{y}{N}\right)$ are $\cF^y$-measurable for $y\!\geq\! 0$; see \cite[Section 1.3]{DuquLeGall02} for a similar definition on Skorokhod space.
  
  \item The \emph{filtration in time} on $\H$, denoted by $(\cF_t,\ t\geq 0)$, is defined to be the least right-continuous filtration under which $N\mapsto \restrict{N}{[0,t]}$ is $\cF_t$-measurable for every $t\geq 0$.
	
  \item We write $(\cF^{y-})$, $(\cF_{t-})$ for left-continuous versions of the filtrations.
 \end{enumerate}
\end{definition}

Generalizing the notation $V^y$ introduced above \eqref{eq:stable:exc_PRM}, we write $V^y(N)$ to denote the set of intervals of complete excursions of $\xi(N)$ about level $y$ and $V^y_0(N)$ to denote this set including the incomplete first and last excursion intervals. These sets are defined formally in Definition \ref{def:exc_intervals}.

\begin{definition}\label{def:bi-clade_PP}
 Take $N\in\H$ and $y\in\BR$. If the level-$y$ local time $\ell^y_N(t):=\ell^y_{\xi(N)}(t)$, $t\geq 0$, of \eqref{eq:LT_def} exists, we define the following counting measures of (bi-/anti-)clades.\vspace{-0.1cm}
  \begin{gather*}
   F^{y}(N) := \sum_{I\in V^y(N)}\!\Dirac{\ell^y_N(a),\shiftrestrict{N}{I}}\!,\ \ 
   F^{\geq y}(N) := \sum_{I\in V^y(N)}\!\Dirac{\ell^y_N(a),\left(\shiftrestrict{N}{I}\right)^+}\!,\\
   F^y_0(N) := \sum_{I\in V^y_0(N)}\!\Dirac{\ell^y_N(a),\shiftrestrict{N}{I}}\!,\ \ 
   F^{\leq y}(N) := \sum_{I\in V^y(N)}\!\Dirac{\ell^y_N(a),\left(\shiftrestrict{N}{I}\right)^-}\!.\vspace{-0.1cm}
  \end{gather*}
  and\vspace{-0.7cm}
  \begin{equation*}
   \begin{split}
    F^{\leq y}_0(N) := F^{\leq y}(N) &+ \cf\left\{N^{\leq y}_{\textnormal{first}} \neq 0\right\}\Dirac{0,N^{\leq y}_{\textnormal{first}}}\\
    		&+ \cf\left\{N^{\leq y}_{\textnormal{last}} \neq 0\right\}\Dirac{\ell^y_N(\len(N)),N^{\leq y}_{\textnormal{last}}},\vspace{-0.1cm}
   \end{split}
  \end{equation*}
  where $N^{\le y}_{\textnormal{first/last}}$ denotes the initial/final (possibly incomplete) anti-clade of $N$ below level $y$; this is defined formally in Appendix \ref{app:exc_intervals}. We define $F^{\ge y}_0(N)$ just as $F^{\le y}_0(N)$, but with all instances of ``$\le$'' replaced by ``$\ge$.'' 
  If $\ell^y_N(t)$ is undefined for some $t\in [0,\len(N)]$, we set all six of these measures to zero.
\end{definition}

%Recall the \PRM\ $\bN$ studied in Sections \ref{sec:prelim:JCCP} and \ref{sec:exc}. We plan to use standard techniques from the study of counting measures, as in \cite{DaleyVereJones1,DaleyVereJones2}, to manipulate the measures of Definition \ref{def:bi-clade_PP}. To justify the use of such techniques, we require the following.

%MW \begin{definition}\label{def:scaff_concat}
 Let $(N_a)_{a\in\mathcal{A}}$ denote a family of elements of $\Hfin$ indexed by a totally ordered set $(\cA,\preceq)$, with all but finitely many being bi-clades. Let $S\colon \mathcal{A}\to[0,\infty]$ be as in \eqref{eq:length_partial_sums}. We require that: (i) $S(a-)<\infty$ for all $a\in\mathcal{A}$ and (ii) there is no infinite $\mathcal{B}\subseteq\mathcal{A}$ with $\inf_{b\in \mathcal{B}}\sup_t|\xi_{N_b}(t)| > 0$. 
 %for which with the amplitudes of $\xi(N_b)$ are bounded away from zero, for $b\in B$.
 Then we define the \emph{concatenation of scaffoldings} by setting\vspace{-0.1cm}
 \begin{equation}\label{eqn:concatscaf}
  %\text{for }t\in \left[0,\sum_{a\in\mathcal{A}}\len(N_a)\right],\quad
  \Big(\Concat_{a\in\mathcal{A}}\xi(N_a)\Big)(t) := \sum_{a\in\mathcal{A}}\left\{\begin{array}{ll}
  		\xi_{N_a}(\len(N_a))	& \text{if }S(a)\le t,\\[2pt]
  		\xi_{N_a}(t - S(a-))	& \text{if }t\in [S(a-),S(a)),\\[2pt]
  		0						& \text{otherwise.}
  	\end{array}\right.\vspace{-0.1cm}
 \end{equation}
 for $t\in \left[0,\sum_{a\in\mathcal{A}}\len(N_a)\right]$. Note that $\xi_{N_a}(\len(N_a))=0$ if $N_a$ is a bi-clade.
%MW \end{definition}

In the following, we need to restrict to ``nice'' levels, to exclude certain degeneracies that $\xi(N)$ may have in general. As these are 
well-known properties when applied to L\'evy processes, we refer to the appendix for details. 

\begin{lemma}\label{lem:cutoff_vs_PP}
 Take $N\in\H$ and $y\in\BR$. If level $y$ is nice for $\xi(N)$, as defined in Proposition \ref{prop:nice_level}, then the cutoff processes of \eqref{eq:cutoff_def} satisfy\vspace{-0.4cm}
 \begin{equation*}
 \begin{split}
  \cutoffL{y}{N} = \Concat_{(s,N^{\le y}_s)}N^{\le y}_s,&\ \ 
  \cutoffL{y}{\xi(N)} = \Concat_{(s,N^{\le y}_s)}\xi\big(N^{\le y}_s\big),\\[-0.2cm]
  \cutoffG{y}{N} = \Concat_{(s,N^{\ge y}_s)}N^{\ge y}_s,&\ \ 
  \cutoffG{y}{\xi(N)} = \Concat_{(s,N^{\ge y}_s)}\xi\big(N^{\ge y}_s\big),\vspace{-0.2cm}
 \end{split}
 \end{equation*}
%  \cutoffL{y}{N} &= \Concat_{\text{points }(s,N^-_s)\text{ of }F^{\leq y}_0(N)}N^-_s,\\
%  \cutoffL{y}{\xi(N)} &= \Concat_{\text{points }(s,N^-_s)\text{ of }F^{\leq y}_0(N)}\xi(N^-_s),\\
%  \cutoffG{y}{N} &= \Concat_{\text{points }(s,N^+_s)\text{ of }F^{\geq y}_0(N)}N^+_s,\\
%  \cutoffG{y}{\xi(N)} &= \Concat_{\text{points }(s,N^+_s)\text{ of }F^{\geq y}_0(N)}\xi(N^+_s).
 where the concatenations in the first line are over points $\big(s,N^{\le y}_s\big)$ of $F^{\le y}_0(N)$, and those in the second line are over points $\big(s,N^{\ge y}_s\big)$ of $F^{\ge y}_0(N)$. 
 On this event, $\cutoffL{y}{N}$ is a measurable function of $F^{\leq y}_0(N)$, and likewise for $F^{\geq y}_0(N)$ and $\cutoffG{y}{N}$. Moreover, $F^{\leq y}_0$ generates $\cF^y$ up to sets that are \PRM[\Leb\otimes\mBxc_{q,c}^{(-2\alpha)}]-null.
\end{lemma}

The proofs of the formulas are straightforward. We prove the measurability assertion in \ref{SuppMeas}.

We will find in \eqref{eq:cutoff_scaffold_commute} that $\cutoffL{y}{\xi(\bN_\beta)} = \xi\big(\cutoffL{y}{\bN_\beta}\big)$ a.s., where $\bN_\beta$ is as in Definition \ref{constr:type-1}. A more general result of this type may be true, but we do not need it here. A challenge to proving such a result is that in general, the scaffolding map does not commute with concatenation: $\xi(\ConcatIL_a N_a) \neq \ConcatIL_a \xi(N_a)$. This can be seen in the case of concatenating bi-clades corresponding to excursions of a \StableA\ process $\xi(\bN)$ above its running minimum. %This decomposition into positive excursions corresponds to a decomposition of a \PRM\ $\bN$ of spindles into degenerate clades: clades lacking a (broken) leftmost spindle, unlike the clade in Figure \ref{fig:bi-clade_decomp}. For each of these degenerate clades $N$, the associated scaffolding $\xi(N)$ is a degenerate \StableA-like excursion above level zero; but when we concatenate these clades, in the manner of Definition \ref{def:concat}, to recover the full point process $\bN$, the resulting scaffolding $\xi(\bN)$ sets these clades above progressively lower levels.

%We would like to say that $\cutoffL{y}{\xi(N)} = \xi\big(\cutoffL{y}{N}\big)$. We will find in \eqref{eq:cutoff_scaffold_commute} that this holds a.s.\ in the setting of the processes $\bN_{\beta}$ of Definition \ref{constr:type-1}. A more general result of this type may be true, but we do not need it here. A challenge to proving such a result is that in general, the scaffolding map does not commute with concatenation: $\xi(\ConcatIL_a N_a) \neq \ConcatIL_a \xi(N_a)$.  Consider, for example, the excursions of a \StableA\ process $\xi(\bN)$ above its past minimum. This decomposition into positive excursions corresponds to a decomposition of a \PRM\ $\bN$ of spindles into degenerate clades: clades lacking a (broken) leftmost spindle, unlike the clade in Figure \ref{fig:bi-clade_decomp}. For each of these degenerate clades $N$, the associated scaffolding $\xi(N)$ is a degenerate \StableA-like excursion above level zero; but when we concatenate these clades, in the manner of Definition \ref{def:concat}, to recover the full point process $\bN$, the resulting scaffolding $\xi(\bN)$ sets these clades above progressively lower levels.

%\section{The type-1 evolution in a \StableA\ process with spindles}

\subsection{Bi-clade It\^o measure and invariance}
\label{sec:biclade_PRM}

Let $\bN\sim\PRM[\Leb\otimes\mBxc_{q,c}^{(-2\alpha)}]$ and $\bX:=\xi(\bN)$. As in Section \ref{sec:exc}, we adopt the convention of suppressing the parameter $\bX$ when referring various functions of $\bX$, including the local time $(\ell^y(t))$, inverse local time $(\tau^y(s))$, hitting and crossing times $T^y$ and $T^{\geq y}$, and sets of excursion intervals $V^y$ and $V^y_0$. Refer back to Section \ref{sec:exc} for definitions of these objects. 
%We restrict to the a.s.\ events that: (a) $V^y$ and $V^y_0$ have the properties enumerated in Proposition \ref{prop:exc_intervals} for every $y\in\BR$, (b) $(\ell^y(t),\ y\in\BR,\ t\geq 0)$ is H\"older continuous in $(y,t)$, as in Theorem \ref{thm:Boylan}, and (c) $\bN\in\td\cN^{\textnormal{sp}}$, as in Lemma \ref{lem:bi-clade_PP_meas} \ref{item:bcPPm:PRM}, so that the counting measures of Definition \ref{def:bi-clade_PP} are all boundedly finite. 
We use notation such as $\bF^y := F^y(\bN)$ and $\bF^{\geq y} := F^{\geq y}(\bN)$ for the counting measures of Definition \ref{def:bi-clade_PP}.

\begin{table}
 \centering
 \begin{tabular}{|c||c|c||Sc|c|}\hline
  		&	PRM of spindles				&	intensity				&	PRM of bi-clades				&	intensity\\\hline\hline
  $(\bN,\xi(\bN))$	&	$\bN = \sum\Dirac{t,f_t}$	&	$\nu=\nu_{q,c}^{(-2\alpha)}$&	$\bF^y = \sum \Dirac{s,N^y_s}$	&	$\mClade$\\[1pt]\hline
  $\bX$				&	$\sum\Dirac{t,\Delta X_t}$	&	$c_\nu x^{-2-\alpha}dx$	    &	$\bG^y = \sum \Dirac{s,g^y_s}$	&	$\mSxc$\\[1pt]\hline
 \end{tabular}\vspace{3pt}
 \caption{Objects from L\'evy process (excursion) theory and analogous objects in the setting of bi-clades, where 
  $\nu$ and $c_\nu$ are given in \eqref{eqn:generalnu}, see also \eqref{eqn:cnu}.\label{tbl:exc_vs_cld}}
\end{table}

%MW \begin{definition}\label{def:bi-clade_Ito}
 We define the \emph{It\^o measures on bi-clades}, \emph{clades}, and \emph{anti-clades} respectively by saying that for $A\in\cHxc{\pm}$, $B\in\cHxc{+}$, and $C\in\cHxc{-}$,\vspace{-0.1cm}
 \begin{equation}\label{eqn:nuclade}
  \left.\begin{array}{c}%MW
  \mClade(A) := \EV\big[\bF^0([0,1]\times A)\big],\\[0.2cm]
  \mClade^+(B) := \EV\big[\bF^{\geq 0}([0,1]\times B)\big], \quad\mClade^-(C) := \EV\big[\bF^{\leq 0}([0,1]\times C)\big].\vspace{-0.1cm}
  \end{array}\right.%MW
 \end{equation}% \qquad \text{for } A\in\cHxc{\pm}.
%  \begin{array}{c}
%  \mClade(A) := \EV\big[\bF^0([0,1]\times A)\big],\\[.2cm]
%  \mClade^+(B) := \EV\big[\bF^{\geq 0}([0,1]\times B)\big], \quad \text{and} \quad \mClade^-(C) := \EV\big[\bF^{\leq 0}([0,1]\times C)\big].
%  \end{array}
%MW \end{definition}

In Proposition \ref{prop:marking_jumps} we construct $\bN$ by marking jumps of the scaffolding $\bX$ with independent $(\alpha,q,c)$-block 
diffusions. 

\begin{comment} 

After an auxiliary lemma, we give a similar description of $\mClade$.

\begin{lemma}\label{lem:meas_marking}
 Consider two complete and separable metric spaces $(\cS,d_\cS)$ and $(\cT,d_\cT)$, equipped with their Borel $\sigma$-algebras $\ScS$ and $\ScT$, respectively. Let $\kappa\colon\cS\times\ScT\rightarrow[0,1]$ denote a stochastic kernel. Let $M\in\cNS$ be a counting measure on $\cS$. Consider for each point $x$ of $M$ an independent mark $m_x$ with distribution $\kappa(x,\cdot)$. Then the map that associates with $M$ the distribution of the marked point process $\int\delta(x,m_x)dM(x)\in\cN(\cS\times\cT)$ is Borel measurable.
\end{lemma}

\begin{proof}
 This is a direct consequence of the following results. First, there exists a measurable enumeration map $\phi\colon\cNS\rightarrow\cS^* := \bigcup_{0\le n\le\infty}\cS^n$ that sends a counting measure to a list of all its points; see \cite[Proposition 9.1.XII]{DaleyVereJones2}. Second, the marking kernel $\kappa$ induces a natural kernel $\kappa^*\colon\cS^*\times\Sigma((\cS\times\cT)^*)\to [0,1]$ marking each of the points in the sequence independently. Finally, the map $\Lambda\colon(\cS\times\cT)^*\rightarrow\cN(\cS\times\cT)$ that sends $((x_j,m_j)) \mapsto \sum_j\delta(x_j,m_j)$ is measurable.
\end{proof}

\end{comment}

\begin{proposition}[Bi-clade It\^o measure via marking jumps]\label{prop:bi-clade_PRM}
 For $g\in\DSxc$, let $\bN_g$ be derived from $g$ like $\bN$ is derived from $\bX$ in Proposition \ref{prop:marking_jumps} -- i.e.\ by passing from a \cadlag\ path to a point process of jumps and marking jumps of height $z$ with spindles with law $\mBxc_{q,c}^{(-2\alpha)}(\,\cdot\;|\;\life = z)$. Then the map that assigns with $g\in\DSxc$ the law
 $\mu_g$ of $\bN_g$ is a marking kernel from $\DSxc$ to $\Hxc{\pm}$, we have $\mClade=\int\mu_g\,\mSxc(dg)$, and 
% \begin{enumerate}[label=(\roman*), ref=(\roman*)]
%  \item For every $g\in\DSxc$, this $\bN_g$ is a random member of $\Hxc{\pm}$. Let $\mu_g$ denote its law.
%  \item The map $g\mapsto\mu_g$ on $\DSxc$ is a $\xi$-disintegration of $\mClade$, in the sense of Definition \ref{def:disintegration}. As in Lemma \ref{lem:scl_ker}, we denote this disintegration by $\mClade(\,\cdot\;|\;\xi)$.
%  \item 
for every $y\in\BR$, the point measure $\bF^y$ is a \PRM[\Leb\otimes\mClade] on $[0,\infty)\times \Hxc{\pm}$.
% \end{enumerate}
\end{proposition}

The reader may find Table \ref{tbl:exc_vs_cld} helpful regarding the counting measures that we have introduced.

\begin{proof}
By definition, $$\bF^y=\sum_{I\in V^y}\delta(\ell^y(I),\ShiftRestrict{\bN}{I})\quad\mbox{and}\quad\bG^y=\sum_{I\in V^y}\delta(\ell^y(I),\ShiftRestrict{\bX}{I}).$$ By Proposition \ref{thm:excursion_PRM}, $\bG^y$ is a \PRM[\Leb\otimes\mSxc].  
By Proposition \ref{prop:marking_jumps}, we may think of $\bN$ as being obtained by marking the \PRM\ of jumps of $\bX$. In particular, $\ShiftRestrict{\bN}{I}$ is obtained from 
$\ShiftRestrict{\bX}{I}$ in the same way independently for all $I\in V^y$. Hence, $\bF^y$ can be obtained from $\bG^y$ using the marking kernel $g\mapsto\mu_g$ and is therefore a
$\PRM[\Leb\otimes\mu^\prime]$ where $\mu^\prime=\int\mu_g\,\mSxc(dg)$. In particular, $\mu^\prime(A)=\EV\big[\bF^0([0,1]\times A)\big]$. By \eqref{eqn:nuclade}, we conclude that %MW \ref{def:bi-clade_Ito},
$\mu^\prime=\mClade$. 
\end{proof}

\begin{corollary}\label{cor:clade_PRM}
 $\bF^{\geq y}$ is a \PRM[\Leb\otimes\mClade^+] on $[0,\infty)\times \Hxc{+}$. Correspondingly, $\bF^{\leq y}$ is a \PRM[\Leb\otimes\mClade^-] on $[0,\infty)\times \Hxc{-}$.
\end{corollary}

We define a time-reversal involution and a scaling operator via
\begin{equation}\label{eq:clade:xform_def}
 \begin{split}
  \reverseH(N) &:= \int \Dirac{\len(N)-t,\reverseexc(f)}dN(t,f) \qquad \text{for }N\in\Hfin\\%\label{eq:clade:reversal}
  %\text{and} \quad 
  a\scaleH N &:= \int \Dirac{a^{1+\alpha}t,\scaleB[a][f]}dN(t,f) \qquad \text{for }N\in\H,\ a>0,
%  \text{and} \quad \scaleM[a][F] &:= \int \Dirac{a^{1/2}s,\scaleH[a][N]}dF(s,N) \qquad \text{for }F\in \cNRHf,\ a>0.
\end{split}
\end{equation}
where $\scaleB$ and $\reverseexc$ are as in %MW Definition \ref{def:BESQ:scaling_def}. 
\eqref{eq:BESQ:scaling_def}. %MW 
The map $\reverseH$, in particular, reverses the order of spindles and reverses time within each spindle.

\begin{lemma}[Bi-clade invariance properties]\label{lem:clade:invariance}
 For $A\in\cHxc{\pm}$, $b>0$,
 $$\mClade(\reverseH(A)) = \mClade(A) \quad \text{and} \quad \mClade(\scaleH[b][A]) = b^{-\alpha}\mClade(A).$$
 Moreover, for $\bN$ a \PRM[\Leb\otimes\nu_{q,c}^{(-2\alpha)}], we have %and $\bF^y$ \PRM s as above, for any $y\geq 0$,
 $\scaleH[b][\bN] \stackrel{d}{=} \bN$. %and $\scaleM[c][\bF^y] \stackrel{d}{=} \bF^0$.
\end{lemma}

\begin{proof}
 This can be derived straightforwardly from Proposition \ref{prop:bi-clade_PRM} and the invariance properties of Bessel processes and stable L\'evy processes noted in Lemmas \ref{lem:BESQ:invariance} and \ref{lem:stable:invariance}. We leave the details to the reader.
\end{proof}

\subsection{Mid-spindle Markov property and conditioning bi-clade It\^o measure}
\label{sec:clade_anti-clade}

Take $N\in\H$. A spindle $f_t$ that arises at time $t$ in $N$ is said to be born at level $\xi_N(t-)$ and die at level $\xi_N(t)$. Thus, at each level $z\in\BR$ it has mass $f_t(z-\xi_N(t-))$. In particular, the spindle crosses level $z$ only if $f_t(z-\xi_N(t-)) > 0$. In a bi-clade $N$ for which $\xi(N)$ is typical, in the sense of Definition \ref{def:typical_exc}, there is a single spindle that crosses level $0$. Otherwise, if $\xi(N)$ is degenerate, there is no such spindle. The following formula isolates the level-$0$ mass of this unique spindle, when it exists. Moreover, the formula is sufficiently general that it may be applied to clades and anti-clades as well. The \emph{(central spindle) mass} of $N\in\Hxc{\pm}$ is
\begin{equation}
 m^0(N) := \int \max\Big\{\ f\big((-\xi_N(s-))-\big),\ f\big(-\xi_N(s-)\big)\ \Big\}dN(s,f).\label{eq:clade:mass_def}
\end{equation}
%See Figure \ref{fig:clade_stats} for an illustration highlighting this and other quantities. 
Consider $N\in\Hxc{\pm}$ for which $\xi(N)$ is typical. Recalling the notation for broken spindles in \eqref{eq:spindle_split}, $f_{T_0^+}(-\xi_N(T_0^+\minus)) = \hat f_{T_0^+}^0(0) = \check f_{T_0^+}^0\big((-\xi_N(T_0^+\minus))\!\minus\!\big)$. Thus, $m^0(N) = m^0(N^+) = m^0(N^-)$.
%For the bi-clades that arise as points for the measures $\bF^y$, the ``max'' in this formula is superfluous, as each spindle mass evolves continuously; equation \eqref{eq:clade_mass_2} then becomes
%\begin{equation}
% m^0(N) = \int f(-\xi_N(t-))dN(t,f).\label{eq:clade_mass}
%\end{equation}

\begin{lemma}\label{lem:heavy_clades_discrete}
 Under $\mClade$, the variable $m^0$ satisfies $\mClade\{m^0 > 1\} < \infty$.
\end{lemma}

\begin{proof}
 Since $m^0(N)$ evaluates a single spindle in $N$ at a single point,
 \begin{equation*}
  \bF^0\big([0,1]\times \{m^0 > 1\}\big) \leq \bN\left( (0,\tau^0(1))\times\left\{f\in \Exc\colon \sup\nolimits_{y\in\BR}f(y) > 1\right\}\right).
 \end{equation*}
 Proposition \ref{thm:excursion_PRM} implies that $\tau^0(1)$ is a.s.\ finite. As $\bN$ is a \PRM[\Leb\otimes\mBxc_{q,c}^{(-2\alpha)}], by Lemma \ref{lem:BESQ:exc_length} the right-hand side is a.s.\ finite. The desired formula follows from the \PRM\ description of $\bF^0$ in Proposition \ref{prop:bi-clade_PRM}.
\end{proof}

Fix $y\in\BR$ and $n,j\in\BN$. For the purpose of the following, let
\begin{equation}\label{eq:Tnj}
 T^{y}_{n,j} := \inf\!\bigg\{ t > 0\;\bigg|\; \int_{[0,t]\times\Exc}\cf\!\left\{f\big( y\!-\!\bX(s-) \big) > \frac{1}{n}\bigg\} d\bN(s,f) \geq j \right\}\!.
\end{equation}
This is the $j^{\text{th}}$ time at which a spindle of $\bN$ crosses level $y$ with mass at least $1/n$.

\begin{lemma}[Mid-spindle Markov property]\label{lem:mid_spindle_Markov}
 Let $T$ be either the stopping time $T^{\geq y}$ for some $y > 0$ or $T^y_{n,j}$ for some $y\in\BR,\ n,j\in\BN$. Let $f_T$ denote the spindle of $\bN$ at this time. Let $\hat f^y_T$ and $\check f^y_T$ denote the split of this spindle about level $y$, as in \eqref{eq:spindle_split}. Then, given $f_T(y-\bX(T-)) = a > 0$,
  \begin{equation*}
   \left( \Restrict{\bN}{[0,T)},\check f^y_T\right) \quad \text{is conditionally independent of} \quad \left( \ShiftRestrict{\bN}{(T,\infty)},\hat f^y_T\right).
  \end{equation*}
  Under this conditional law, $\shiftrestrict{\bN}{(T,\infty)}$ is a \PRM[\Leb\otimes\mBxc_{q,c}^{(-2\alpha)}] independent of $\hat f^y_T$, which is a $(\alpha,q,c)$-block diffusion started at $a$ and killed upon hitting 0.
\end{lemma}

\begin{proof}
 We start by proving the case $T = T^y_{n,j}$. By the strong Markov property of $\bN$, it suffices to prove this with $j=1$. For the purpose of the following, let $\Exc_n := \big\{f\in\Exc\colon \sup_{u}f(u) > \frac1n\big\}$, where $\Exc$ is the space of spindles of \eqref{eq:cts_exc_space_def}. Lemma \ref{lem:BESQ:exc_length} asserts that $\mBxc(\Exc_n) < \infty$. Thus, we may sequentially list the points of $\bN$ in $\Exc_n$:
 \begin{equation*}
  \Restrict{\bN}{[0,\infty)\times \Exc_n} = \sum_{i=1}^\infty \Dirac{T_i,f_i} \quad \text{with} \quad T_1 < T_2 < \cdots.
 \end{equation*}
 
 First, note that each time $T_i$ is a stopping time in the time-filtration $(\cF_t)$; thus, by the Poisson property of $\bN$, each $f_i$ is independent of $\cF_{T_i-}$. Also the $(f_i)$ are i.i.d.\ with the law $\mBxc_{q,c}^{(-2\alpha)}(\,\cdot\,|\,\Exc_n)$. We define first passage times of $f_i$, $H_i := \inf\{z > 0\colon f_{i}(z) = 1/n\}$. Then by Lemma \ref{lem:BESQ:existence}, for each $i$ the process $\big(f_i(H_i+z),\ z\geq 0\big)$ is an $(\alpha,q,c)$-block diffusion starting from $1/n$. We define a stopping $\rho_i$ for $f_i$ as follows. If $f_i(y-\bX(T_i-)) > 1/n$, set $\rho_i := y-\bX(T_i-)$; otherwise, set $\rho_i := \zeta(f_i)$. Thus, $\rho_i$ is always greater than $H_i$, and hence $\rho_i-H_i$ is a stopping time for $\big(f_i(H_i+z),\ z\geq 0\big)$.
 
 Recall Definition \ref{def:filtrations} of $(\cF_t,\ t\geq 0)$. For the purpose of the following, for $i\!\geq\! 1$ let $\cG_i := \sigma(\cF_{T_i-},\restrict{f_{i}}{(-\infty,\rho_i)})$. The sequence of pairs $\big(\bX(T_i-),\ \restrict{f_i}{(-\infty,\rho_i)}\big)$ is a Markov chain in this filtration. Indeed, in the case $\rho_i = \life(f_i)$, the process $\bX$ simply runs forward from its value $\bX(T_i) = \bX(T_i-)+\life(f_i)$ until the $(\cF_t)$-stopping time $T_{i+1}$. In the case $\rho_i < \life(f_i)$, we have $f_i(\rho_i) > 1/n$. Then by the Markov property of $\big(f_i(H_i+z),\ z\geq 0\big)$ at $\rho_i-H_i$, conditionally given $\cG_i$ the process $\hat f^y_i = \shiftrestrict{f_i}{[\rho_i,\infty)}$ is an $(\alpha,q,c)$-block diffusion starting from $f_i(\rho_i)$. In particular, $\hat f^y_i$ is conditionally independent of $\cG_i$ given $f_i(\rho_i)$. Then $\bX(T_i) = y + \life\big(\hat f^y_i\big)$.
 %, and $(\bN,\bX)$ proceed forward from $T_i$ as before.
 %, and by the strong Markov property of $\bN$, $(\shiftrestrict{\bN}{(T_i,\infty)},\shiftrestrict{\bX}{(T_i,\infty)})$
 
 Let $J := \inf\big\{i \geq 1\colon \rho_i < \life(f_i)\big\}$, so $T_J = T$. This $J$ is a stopping time for $(\cG_i)$. Therefore, conditionally given $f_J(\rho_J)$, the process $\hat f^y_J$ is independent of $\cG_J$, distributed like an $(\alpha,q,c)$-block diffusion starting from $f_J(\rho_J)$. By the strong Markov property of $\bN$, the process $\ShiftRestrict{\bN}{(T,\infty)}$ is a \PRM[\Leb\otimes\nu_{q,c}^{(-2\alpha)}], independent of $\Restrict{\bN}{[0,T]}$, as desired.
 
 For the case $T = T^{\geq y}$, note that $T^{\geq y} = \inf_{n \geq 1}T^y_{n,1}$. It follows from Proposition \ref{prop:nice_level} that this infimum is almost surely attained by some $n$. Thus, the result in this case follows from the previous case.
\end{proof}

\begin{lemma}\label{lem:clade:mass_ker}
 The It\^o measure $\mClade$ admits a unique $m^0$-disintegration $\mClade(\,\cdot\;|\;m^0)$ with the scaling property that for $a>0$, $B\in\SHxc{\pm}$,
 \begin{equation}
  \mClade(B\;|\;m^0 = a) = \mClade\left(\scaleH[a^{-1/q}][B]\;\middle|\;m^0 = 1\right).\label{eq:clade_mass_ker}
 \end{equation}
 Likewise, $\mClade^+$ and $\mClade^-$ admit unique $m^0$-disintegrations with this same scaling property.
\end{lemma}

\begin{proof}
 Lemmas \ref{lem:clade:invariance} and \ref{lem:heavy_clades_discrete} and the scaling property $m^0(a^{1/q}\scaleH N) = a m^0(N)$ satisfy the hypotheses of Section \ref{sec:techrem}, which then yields the claimed result.
\end{proof}

\begin{proposition}\label{prop:clade_splitting}
 \begin{enumerate}[label=(\roman*), ref=(\roman*)]
  \item[(i)] Fix $a\!>\!0$. Let $\ol N_a$ have law $\mClade(\,\cdot\,|\,m^0\!=\!a)$, and let $(\ol N_a^+,\ol N_a^-)$ denote its decomposition into clade and anti-clade. Then $(\ol N_a^+,\ol N_a^-)$ has distribution
   \begin{equation}
    \mClade^+(\,\cdot\;|\;m^0 = a) \otimes \mClade^-(\,\cdot\;|\;m^0 = a);\label{eq:clade_split_joint_law}
   \end{equation}
   in particular, $\ol N_a^+$ and $\ol N_a^-$ are independent.\label{prop:c_s:split}
  \item[(ii)] Let $\hat f$ denote an $(\alpha,q,c)$-block diffusion started at $\hat f(0) = a$ and absorbed upon hitting $0$, independent of $\bN$, and let $\widehat T^0 := \inf\big\{t>0\!:$ $\xi_{\bN}(t) = -\zeta(\hat f)\big\}$. We define
   \begin{equation}
    \ol N_a^+ := \Dirac{0,\hat f} + \Restrict{\bN}{\left[0,\widehat T^0\right)}.
   \end{equation}
   Then $\ol N_a^+$ has the law $\mClade^+(\,\cdot\;|\;m^0 = a)$ and $\reverseH(\ol N_a^+)$ has the law $\mClade^-(\,\cdot\;|\;m^0 = a)$, where $\reverseH$ is time reversal as in \eqref{eq:clade:xform_def}.\label{prop:c_s:clade_law}
 \end{enumerate}
\end{proposition}

\begin{proof} 
 Let $T = T^y_{n,j}$ be as in \eqref{eq:Tnj}, with $y=0$, $n=j=1$, and let $S=\ell^0(T)$ and $N_S=\ShiftRestrict{\bN}{(\tau^0(S-)\tau^0(S))}$.
 Then $(S,N_S)$ is the earliest point of $\bF^0$ in $\{N\in\Hxc{\pm}\colon m^0(N) > 1\}$ 
  and $(T,f_T)$ is the spindle in $\bN$ that corresponds to the jump of $\xi(N_S)$ across level zero.  
 By the \PRM\ description of $\bF^0$, the bi-clade $N_S$ has law $\mClade(\,\cdot\;|\;m^0 > 1)$. 
 Moreover, as we are in the disintegration setting of Section \ref{sec:techrem}, $\ol N:= \scaleH[\frac{1}{m^0(N_S)}][N_S]$ has distribution $\mClade(\,\cdot\;|\;m^0 = 1)$, 
  and $\ol N_a:=\scaleH[a][\ol N]$ has law $\mClade(\,\cdot\;|\;m^0=a)$.
 
 The marginal distributions of $\ol N_a^+$ and $\ol N_a^-$ stated in (i) follow straight from the definitions of $\mClade$, $\mClade^+$ and $\mClade^-$ in \eqref{eqn:nuclade}. %MW 
 The independence of $\ol N_a^+$ and $\ol N_a^-$ asserted in (i) and the description of $\mClade^+$ stated in (ii) follow from Lemma \ref{lem:mid_spindle_Markov}. 
 For the corresponding description of $\mClade^-$ in (ii), observe that
 \begin{equation}
  (\reverseH(N))^- = \reverseH(N^+) \quad \text{and} \quad m^0(\reverseH(N)) = m^0(N).
 \end{equation}
 (We refer the reader back to Figure \ref{fig:bi-clade_decomp} for an illustration; $\reverseH$ time-reversal corresponds to holding the page upside down.) 
 By Lemma \ref{lem:clade:invariance}, if $\ol N_a$ has law $\mClade(\,\cdot\;|\;m^0=a)$ then so does $\reverseH(\ol N_a)$. 
 Thus, $\reverseH(\ol N_a^+)$ has law $\mClade^-(\,\cdot\;|\;m^0=a)$, as desired.
\end{proof}

\section{$\nu_{\tt BESQ}^{(-2\alpha)}$-IP-evolutions in stopped \StableA\ processes}\label{sec:easyinitialdata}

\subsection{$\nu_{\tt BESQ}^{(-2\alpha)}$-IP-evolutions}
\label{sec:type-1:def}

Let $\bN$ be a \PRM[\Leb\otimes\nu_{\tt BESQ}^{(-2\alpha)}]. 

\begin{proposition}[Aggregate mass from $F^y(N)$]\label{prop:agg_mass_subord_extra}
 Take $N\in\H$ and $y\in\BR$ and suppose that level $y$ is nice for $\xi(N)$, in the sense of Proposition \ref{prop:nice_level}. Suppose also that either $\len(N) = \infty$ or $\xi_N(\len(N)) < y$. We write $F^y_{N} := F^y(N)$. For every $s\geq 0$,
 \begin{gather}\label{eq:agg_mass_from_clades}
  M^y_{N}\circ\tau^y_N(s) = M^y_N\circ\tau^y_N(0) + \int_{(0,s]\times\Hxc{\pm}} m^0(N')dF^y_N(r,N')\\
  \skewer(y,N) = \big\{ \big(M^y_{N}\circ\tau^y_N(s\minus),M^y_{N}\circ\tau^y_N(s)\big) \colon s\geq 0,\ \tau^y(s) > \tau^y(s\minus) \big\},\notag
 \end{gather}
 where we take $\tau^y(0\minus) := 0$. In particular, for fixed $y\in\BR$ this holds for $N = \bN$ almost surely. 
\end{proposition}

\begin{proof}
 As noted in Proposition \ref{prop:partn_spindles}, the bi-clades of $F^y(N)$, along with the potential initial and final incomplete bi-clades, partition the spindles of $N$. At most one of the spindles in the initial incomplete bi-clade crosses level $y$. Each subsequent excursion interval $I^y_N(a,b)$ with $[a,b]\in V^y(N)$ includes at most one jump of $\xi(N)$ that crosses level $y$. If $N' := \shiftrestrict{N}{I^y_N(a,b)}$ then this spindle crosses with mass $m^0(N')$. Finally, our requirement that either $\len(N) = \infty$ or $\xi_N(\len(N)) < y$ implies that either there is no final incomplete bi-clade about $y$ or, if there is, then this bi-clade dies during the incomplete anti-clade $N^{\leq y}_{\textnormal{last}}$, without contributing mass at level $y$. This gives us the claimed description of $M^y_N\circ\tau^y_N$. The subsequent description of $\skewer(y,N)$ follows from our assumption that level $y$ is nice for $\xi(N)$: no two level $y$ bi-clades, complete or incomplete, arise at the same local time.
 
 If $N=\bN$ then by Proposition \ref{prop:nice_level}, level $y$ is nice for $\bX$ almost surely. 
\end{proof}

Recall Definition \ref{constr:type-1} of $\nu_{\tt BESQ}^{(-2\alpha)}$-IP-evolutions, with $\bN_{\beta} = \ConcatIL_{U\in\beta}\bN_U$. Comparing that construction to Proposition \ref{prop:clade_splitting}, we see that each $\bN_U$ has distribution $\mClade^+(\,\cdot\;|\;m^0=\Leb(U))$. 

\begin{definition}\label{constr:type-1_2} Let $\Pr^{(\alpha)}_\beta:=\Pr^{\alpha,1,1}_\beta$.
 For $\bN_{\beta}$ as in Definition \ref{constr:type-1}, we abuse notation to write\vspace{-0.2cm}
 $$F^{\geq 0}_0(\bN_{\beta}) := \sum_{U\in\beta}\Dirac{\IPLT_{\beta}(U),\bN_U},\vspace{-0.2cm}$$ 
 substituting diversities in the place of local times in Definition \ref{def:bi-clade_PP}. We write $\Pr^{(\alpha)}_{\beta}\{F^{\geq 0}_0\!\in\!\cdot\,\}$ for its distribution, and correspondingly for $\Pr^{(\alpha)}_{\mu}:=\Pr^{\alpha,1,1}_\mu$.
\end{definition}

We will find from Propositions \ref{prop:type-1:LT_diversity} and \ref{prop:cts_lt_at_0} that 
 almost surely for all $t\geq 0$, if $t$ falls within the segment of $\bN_{\beta}$ corresponding to $\bN_U$, then\vspace{-0.1cm}
$$\IPLT_{\beta}(U) = \lim_{y\downto 0} \ell^y_{\bN_{\beta}}(t) = \lim_{h\downto 0}h^{-1}\Leb\{u\in [0,t]\colon \xi_{\bN_{\beta}}(u)\in [0,h]\}.$$

\begin{proposition}\label{prop:clade_lengths_summable}
 \begin{enumerate}[label=(\roman*), ref=(\roman*)]
  \item For every $\beta\in\IPspace$, the point process $\bN_{\beta}$ of Definition \ref{constr:type-1} a.s.\ has finite length: in the notation of that definition, $\sum_{U\in\beta}\len(\bN_U) < \infty$ a.s., when $q=c=1$.\label{item:CLS:CLS}
  \item The map $\beta\mapsto \Pr^{(\alpha)}_{\beta}$ is a stochastic kernel.\label{item:CLS:kernel}
  \item We have $\xi(\bN_{\beta}) = \ConcatIL_{U\in\beta}\,\xi(\bN_U)$, where concatenation is as in %MW Definition \ref{def:scaff_concat}.
    \eqref{eqn:concatscaf}. 
    \label{item:CLS:scaffold}
 %is formed by concatenating the positive excursions $\xi(\bN_U)$:
%  $$\xi_{\bN_{\beta}}(t) = \xi_{\bN_U}\!\left(t-\sum\nolimits_{V\in\beta\colon V<U}\len(\bN_V)\right) \qquad \text{for }t\in [0,\len(\bN_{\beta})],$$
%  for any $U$ for which $0\le t - \sum\nolimits_{V\in\beta\colon V<U} \le \Leb(U)$, and $\xi_{\bN_{\beta}}(t)=0$ if there is no such $U$.
  \item \label{item:CLS:PP}
  The map $\beta \mapsto \Pr^{(\alpha)}_{\beta}\{F^{\geq 0}_0\in \cdot\,\}$ is a stochastic kernel. Moreover, there exists a measurable function $\phi\colon \Hfin\to\cNRHf$ such that $F^{\geq 0}_0(\bN_{\beta})=\phi(\bN_{\beta})$ a.s..
  %For the purpose of the following, let $\cS := \{N\in\Hfin\colon \skewer(0,N)\in\IPspace\}$. The laws $\Pr^{(\alpha)}_{\beta}$ are supported on $\cS$. There exists a measurable map $\phi\colon \cS\to \cNRHf$ such that $F^{\geq 0}_0(\bN_{\beta})=\phi(\bN_{\beta})$. Moreover, the map $\beta \mapsto \Pr^{(\alpha)}_{\beta}\{F^{\geq 0}_0\in \cdot\,\}$ is a stochastic kernel.
 \end{enumerate}
\end{proposition}

\begin{proof}
 \ref{item:CLS:CLS} Let $(\bff_U,\,U\in\beta)$ denote an independent family of \BESQ[-2\alpha] processes absorbed at $0$, with each $\bff_U$ starting from $\Leb(U)$. By Lemma \ref{lem:BESQ:length},
 $$\EV[\life(\bff_U)] = \frac{1}{\Gamma(1+\alpha)} \int_0^{\infty} \frac{\Leb(U)}{2x}x^{\alpha}e^{-x}dx =\frac{\Leb(U)}{2\alpha}.$$
 For each $U\in \beta$, let $S(U-) := \sum_{V\in\beta\colon V<U}\life(\bff_V)$ and $S(U) := S(U-)+\life(\bff_U)$. Let $L:= \sup_{U\in\beta}S(U)$. Then
 \begin{equation}
  \EV[L] = \sum_{U\in\beta} \EV\big[\life(\bff_U)\big] =\frac{1}{2\alpha} \sum_{U\in\beta} \Leb(U) = \IPmag{\beta} < \infty.
 \end{equation} 
  Let $\bN$ denote a \PRM[\Leb\otimes\mBxc_{\tt BESQ}^{(-2\alpha)}], independent of $(\bff_U,\,U\in\beta)$. Let $H_U$ denote the first hitting time of $-S(U-)$ in $\xi(\bN)$, let $T$ denote the hitting time of $-L$, and set\vspace{-0.2cm}
 $$\bN_{\beta}' := \Restrict{\bN}{[0,T]} + \sum_{U\in\beta}\DiracBig{S(U),\bff_U}.\vspace{-0.2cm}$$
 It follows from the strong Markov property of $\bN$ that $\bN_{\beta}'$ has law $\Pr^{(\alpha)}_{\beta}$. Thus, in the setting of Definition
 \ref{constr:type-1}, the total length $\sum_{U\in\beta}\len(\bN_U)$ has the same distribution as $T$ in the construction here, which is 
 a.s.\ finite.
 
 \ref{item:CLS:kernel} This is straightforward from standard marking kernel methods.
 %
 %For $\beta\in\IPspace$, the counting measure $\sum_{(a,b)\in\beta}\Dirac{a,|b-a|}$ is a measurable function of $\beta$. By Lemma \ref{lem:meas_marking}, the map from $\beta$ to the law of $\bG_{\beta} := \sum_{(a,b)\in\beta}\Dirac{a,\bN_{(a,b)}}$ is a kernel, as we have marked the points via the kernel $|b-a|\mapsto \mClade^+(\,\cdot\,|\; m^0=|b-a|)$. And finally, the map that takes $\bG_{\beta}$ to $\bN_{\beta} = \ConcatIL \bN_{(a,b)}$ by concatenating over points of $\bG_{\beta}$ is measurable.
 
 \ref{item:CLS:scaffold} In the construction of \ref{item:CLS:CLS}, adding spindles $\DiracBig{S(U),\bff_U}$ to $\bN$ with summable lifetimes modifies the associated scaffolding $\xi(\bN_{\beta}')$ only by adding jumps of the corresponding heights. In particular, $\xi(\bN_{\beta}')$ is formed by concatenating the paths of the excursions $\xi\big(\shiftrestrict{\bN_{\beta}'}{[S(U-),S(U)]}\big)$. Thus, the claimed identity holds a.s.\ under $\Pr^{(\alpha)}_{\beta}$. We remark that, in light of Lemma \ref{lem:cutoff_vs_PP} connecting cutoff processes to point processes of (anti-)clades and Corollary \ref{cor:clade_PRM} asserting that these are Poisson point processes, this also proves
 \begin{equation}\label{eq:cutoff_scaffold_commute}
 \begin{split}
  \xi\left(\cutoffL{z}{\restrict{\bN}{[0,t]}}\right) &= \cutoffL{z}{\restrict{\xi(\bN)}{[0,t]}}\text{ and}\\
  \xi\left(\cutoffL{z}{\restrict{\bN_{\beta}}{[0,t]}}\right) &= \cutoffL{z}{\restrict{\xi(\bN_{\beta})}{[0,t]}} \quad \text{for }t\ge0.
 \end{split}
 \end{equation}

 \ref{item:CLS:PP} We prove this assertion in \ref{SuppMeas}.
\end{proof}

We now relate point processes of clades to the skewer process. Recall the cutoff processes of \eqref{eq:cutoff_def}.

\begin{lemma}\label{lem:cutoff_skewer}
 Take $N\in\H$, $y,z\geq 0$, and suppose $M^y_N(\len(N)) < \infty$.%For the purpose of this lemma, we abuse notation and write $\ConcatIL_{(s,N_s)\in F}$ to denote concatenation over the points of a point process $F$, and likewise for sums.
 \begin{enumerate}[label=(\roman*), ref=(\roman*)]
  \item \label{item:CPS:cutoff_skewer}
  $\skewer(y,N)\,$equals $\skewer(y,\cutoffL{z}{N}\!,\cutoffL{z}{\xi(N)})$ for $y < z$ or 
  $\skewer\left(y-z,\cutoffG{z}{N},\cutoffG{z}{\xi(N)}\right)$ for $y > z$.%MW \pagebreak
  \item \label{item:CPS:clades_skewer}
  	If level $z$ is nice for $\xi(N)$, in the sense of Proposition \ref{prop:nice_level}, then
  	\begin{equation}\label{eq:skewer_from_clade_PP}
  	 \skewer(y,N) = \Concat_{\text{points }(s,N^+_s)\text{ of }F^{\geq z}_0(N)} \skewer(y-z,N^+_s)\vspace{-0.3cm}
  	\end{equation}
  	for $y\ge z$.
  \item \label{item:CPS:clades_skewer_0}
  	Suppose $\beta\in\IPspace$ is \emph{nice} in the sense that, for $U,V\in\beta$, if $U\neq V$ then $\IPLT_{\beta}(U)\neq \IPLT_{\beta}(V)$. Let $\bN_{\beta}$ and $F^{\geq 0}_0(\bN_{\beta})$ be as in Definition \ref{constr:type-1_2}. In the event that $M^y_{\bN_{\beta}}(t) < \infty$ for all $t < \len(\bN_{\beta})$, \eqref{eq:skewer_from_clade_PP} holds with $z=0$ and $N = \bN_{\beta}$.
%  \item For $s\geq 0$, the map $\restrict{F^{\leq z}_0}{[0,\tau^z(s))}$ generates $\ol\cG^{\leq z}_s$ up to $\PRM[\Leb\otimes\mBxc]$-null sets. Likewise, $\restrict{F^{\leq z}_0}{[0,\tau^z(s-))}$ generates $\ol\cG^{\leq z}_{s-}$ up to $\PRM[\Leb\otimes\mBxc]$-null sets.
 \item The process $(\skewer(y,N),\,y\ge0)$, defined on $N\in \Hfins$, is adapted to the restriction of the filtration $(\cF^y,\,y\ge0)$ to $\Hfins$.
 \end{enumerate}
\end{lemma}

We prove this in \ref{SuppMeas}.

%\begin{proof}
 %These formulas are straightforward consequences of the definitions and Lemma \ref{lem:cutoff_vs_PP}. We leave the details to the reader.
%\end{proof}

%\begin{figure}
% \centering
% \caption{NOTE: Placeholder figure. Needed to illustrate \eqref{eq:skewer_by_concatenating_clades_at_y}.\label{fig:skewer_by_concatenating_clades_at_y}}
%\end{figure}

%In light of Lemma \ref{lem:cutoff_skewer}, the skewer process $\skewerP$ as a map on $\Hfins$ is adapted to the filtration $(\cF^{y,*},\,y\geq 0)$, where we take $\cF^{y,*} := \{A\cap\Hfins\colon A\in\cF^y\}$.

\subsection{$\nu_{\tt BESQ}^{(-2\alpha)}$-IP-evolutions with \Stable[\alpha] initial state}

As before, let $\bN$ be a \PRM[\Leb\otimes\mBxc_{\tt BESQ}^{(-2\alpha)}] living on a probability space $(\Omega,\cA,\Pr)$. We continue to use the notation of the first paragraph of Section \ref{sec:biclade_PRM} for objects related to $\bN$. Let $(\ol\cF_t,\,t\ge0)$ and $(\ol\cF^y,\,y\ge0)$ denote $\Pr$-completions of the time- and level-filtrations on $(\Omega,\cA)$ generated by $\bN$, as in Definition \ref{def:filtrations}, augmented to allow an independent random variable $S$ measurable in $\ol\cF^0\cap\ol\cF_0$. That is, these are formed by augmenting the $\Pr$-completions of the pullbacks, via $\bN\colon \Omega\to\H$, of the time- and level-filtrations on $\H$. 
We define $\tdN := \restrict{\bN}{[0,T)}$, where $T$ is an a.s.\ finite $(\ol\cF_t)$-stopping time. We again take ``twiddled versions'' of our earlier notation to denote the corresponding objects for $\tdN$. %For the purpose of the following, for $y\geq 0$ we define the level $y$ \emph{local time filtration} $(\ol\cG^{\leq y}_s,\ s\geq 0)$ via
%\begin{equation}\label{eq:filtration_LT_def}
% \ol\cG^{\leq y}_s := \ol\cF^y \cap \ol\cF_{\tau^y(s)} \qquad \text{for }y,s\geq 0.
%\end{equation}
 %\ol\cG^{\leq y}_s := \ol\cF^y \cap \bigcap_{r>s}\ol\cF_{\tau^y(r)} \qquad \text{for }y,s\geq 0.

\begin{proposition}\label{prop:PRM:Fy-_Fy+}
 Suppose $T$ has the properties: (a) $S^0 := \ell^0(T)$ is measurable in $\ol\cF^0$, and (b) $\bX < 0$ on the time interval $(\tau^0(S^0-),T)$. Then for each $y\geq 0$, the measure $\tdF^{\geq y}_0 = F^{\geq y}_0(\tdN)$ is conditionally independent of $\ol\cF^{y}$ given $\td\beta^y$, with the regular conditional distribution (r.c.d.) $\Pr^{(\alpha)}_{\td\beta^y}\{F^{\geq 0}_0\in\cdot\,\}$ of Definition \ref{constr:type-1_2}.
\end{proposition}

In light of Lemma \ref{lem:cutoff_skewer} \ref{item:CPS:cutoff_skewer}, this proposition is very close to a simple Markov property for $(\td\beta^y,\ y\geq 0)$. In order to minimize our involvement with measure-theoretic technicalities, we will postpone pinning this connection down until Corollary \ref{cor:type-1:simple_Markov}.

\begin{proof}
 Step 1 of this proof establishes the claimed result at a fixed level $y\geq 0$ when $T = \tau^y(s-)$, where $s>0$ is fixed. Note that this time does not satisfy conditions (a) and (b). In Step 1, $T$ is specific to a fixed level $y$, whereas in the proposition, the result holds at each level for a single time $T$. 
 %falls outside of the regime of the proposition, which is stated so as to apply for all $y$. 
 In Step 2, we extend this to describe the unstopped point process $\bF^{\geq y}_0$. Finally, in Step 3, we extend our results to the regime of the proposition.
 
 \emph{Step 1}: Assume $T = \tau^y(s-)$. Note that $\tdF^y_0 = \restrict{\bF^y_0}{[0,s)}$. The strong Markov property of $\bN$ tells us that $\restrict{\bN}{[0,T^y)}$ is independent of $\shiftrestrict{\bN}{[T^y,\infty)}$. Rephrasing this in the notation of Definition \ref{def:bi-clade_PP}, $(\bN^{\leq y}_{\textnormal{first}},\bN^{\geq y}_{\textnormal{first}})$ is independent of $(\bF^{\leq y},\bF^{\geq y})$. This will allow us to consider conditioning separately for the first pair and the second. Let $m^y\colon \H\to [0,\infty)$ denote the mass of the leftmost spindle at level $y$:
 \begin{equation}\label{eq:LMB_def}
  m^y(N) := M^y_N\big(\inf\{ t\geq 0\colon M^y_N(t) >0\}\big) \qquad \text{for }N\in\H.
  %\inf\big( \textnormal{range}(M^y_N)\cap (0,\infty) \big) \quad \text{for }N\in\H.
 \end{equation}
 We apply the mid-spindle Markov property, Lemma \ref{lem:mid_spindle_Markov}, at time $T^{\geq y}$. Together with the description of $\mClade^+$ in Proposition \ref{prop:clade_splitting}, this implies that the clade $\bN^{\geq y}_{\textnormal{first}}$ has conditional law $\mClade^+(\,\cdot\;|\;m^0 = m^y(\bN^{\leq y}_{\textnormal{first}}))$ given $\bN^{\leq y}_{\textnormal{first}}$, as desired.
 
 Now, let $\td\gamma^y$ denote $\td\beta^y$ minus its leftmost block, so that $\td\beta^y\!=\!\{(0,m^y(\tdN))\}\concat\td\gamma^y$. Proposition \ref{prop:agg_mass_subord} indicates two properties of $\td\gamma^y$: (a) it is a \Stable[\alpha/q] interval partition with total diversity $s$, in the sense of Proposition \ref{prop:IP:Stable}, and (b) it a.s.\ equals a function of $\tdF^y$. For $\beta\in\IPspace$ let $(N^{\pm}_U,\ U\in\beta)$ denote a family of independent bi-clades with respective laws $N^{\pm}_U \sim \mClade\{\,\cdot\;|\,m^0=\Leb(U)\}$,
 $$G_{\beta} := \sum_{U\in\beta} \Dirac{\IPLT_{\beta}(U),\Leb(U)}, \quad \text{and} \quad G^{\pm}_{\beta} := \sum_{U\in\beta} \Dirac{\IPLT_{\beta}(U),N^{\pm}_U}.$$
 Then $G_{\td\gamma^y}$ is a \PRM[\Leb\otimes \mClade\{m^0\in \cdot\,\}] on $[0,s)\times (0,\infty)$. Moreover, $G^{\pm}_{\td\gamma^y}$ is a \PRM[\Leb\otimes\mClade] on $[0,s)\times \Hxc{\pm}$, as it may be obtained by marking the points of $G_{\beta}$ via the stochastic kernel $a\mapsto \mClade\{\,\cdot\;|\,m^0 = a)$, and this is an $m^0$-disintegration of $\mClade$. By Proposition \ref{prop:bi-clade_PRM}, $\tdF^{y}$ has the same \PRM\ distribution as $G^{\pm}_{\td\gamma^y}$. Thus, the distribution of $G^{\pm}_{\td\gamma^y}$ is a regular conditional distribution for $\tdF^y$ given $\td\gamma^y$.
 
 Extending the preceding construction of $G^{\pm}_\beta$, for each $U\in\beta$ let $(N^+_U,N^-_U)$ denote the clade and anti-clade components of $N^{\pm}_U$, respectively. By Proposition \ref{prop:clade_splitting} these are independent. Thus,
 $$G^+_{\beta} := \sum_{U\in\beta} \Dirac{\IPLT_{\beta}(U),N^{+}_U} \quad \text{is independent of} \quad G^{-}_{\beta} := \sum_{U\in\beta} \Dirac{\IPLT_{\beta}(U),N^{-}_U}.$$
 Moreover, $G^+_{\beta}$ has law $\Pr^{(\alpha)}_{\beta}\{F^{\geq 0}_0\in\cdot\,\}$, as in Definition \ref{constr:type-1_2}. Thus, given $\td\gamma^y$, the measure $\tdF^{\geq y}$ is conditionally independent of $\tdF^{\leq y}$ with regular conditional distribution $\Pr^{(\alpha)}_{\td\gamma^y}\{F^{\geq 0}_0\in\cdot\,\}$. By another application of the strong Markov property of $\bN$ at time $T$, this conditional independence extends to conditional independence between $\tdF^{\geq y}$ and $\bF^{\leq y}$. 
 Now note the general principle that from $\cF_1\indep_{\cH_1}\cG_1$, $\cF_2\indep_{\cH_2}\cG_2$, and $(\cF_1,\cG_1,\cH_1)\indep (\cF_2,\cG_2,\cH_2)$, we may deduce $(\cF_1,\cF_2)\indep_{\cH_1,\cH_2}(\cG_1,\cG_2)$; see e.g.\ \cite[Propositions 6.6-6.8]{Kallenberg}. Thus, $\tdF^{\geq y}_0$ is conditionally independent of $\ol\cF^y$ given $\td\beta^y$, with regular conditional distribution $\Pr^{(\alpha)}_{\td\beta^y}\{F^{\geq 0}_0\in\cdot\,\}$.
 
 %Finally, recalling the first arguments in Step 1 via mid-spindle and strong Markov properties, we conclude by standard methods \cite[Propositions 6.6-6.8]{Kallenberg} that given $\td\alpha^y$, $\tdF^{\geq y}_0$ is conditionally independent of $\ol\cF^y$, with regular conditional distribution $\Pr^1_{\td\alpha^y}\{F^{\geq 0}_0\in\cdot\,\}$. The general principle is that given $A_1\indep_{C_1}B_1$, $A_2\indep_{C_2}B_2$, and $(A_1,B_1,C_1)\indep (A_2,B_2,C_2)$, it follows that $(A_1,A_2)\indep_{C_1,C_2}(B_1,B_2)$.
 
 \emph{Step 2}: For $s>0$ let $\gamma^y_s := \skewer(y,\restrict{\bN}{[0,\tau^y(s-)})$. We write $\gamma^y_\infty := (\gamma^y_n,\ n\in\BN)$; this takes values in the subset of $\IPspace^{\BN}$ comprising projectively consistent sequences. We equip $\IPspace^{\BN}$ with the product $\sigma$-algebra. In the regime of such projectively consistent sequences, Definition \ref{constr:type-1_2} extends naturally to define a kernel $\beta_\infty = (\beta_n,\ n\geq 1) \mapsto \Pr^{(\alpha)}_{\beta_\infty}\{F^{\geq 0}_0\in\cdot\,\}$; i.e.\ a point process $G$ has this law if $\restrict{G}{[0,n)} \stackrel{d}{=} \sum_{U\in\beta_{n}} \Dirac{\IPLT_{\beta_{n}}(U),N^+_U}$ for every $n\geq 0$, 
 where the $(N^+_U)$ are as above. Extending the conditioning in the conclusion of Step 1, we find that $\restrict{\bF^{\geq y}_0}{[0,n)}$ is conditionally independent of $\ol\cF^y$ given $\gamma^y_\infty$. By consistency, $\bF^{\geq y}_0$ is conditionally independent of $\ol\cF^y$ given $\gamma^y_\infty$, with r.c.d.\ $\Pr^{(\alpha)}_{\gamma^y_\infty}\{F^{\geq 0}_0\in\cdot\,\}$.
 
 \emph{Step 3}: Assume $T$ satisfies conditions (a) and (b) stated in the proposition. We now show that $S^y := \ell^y(T)$ is measurable in $\ol\cF^y$. For $y=0$, this is exactly condition (a), so assume $y>0$. From condition (b), $S^y = \ell^y(\tau^0(S^0-))$. Thus, $\tau^0(S^0-)\in (\tau^y(S^y-),\tau^y(S^y))$. By monotonicity of $\ell^0$ we have $S^0\in [\ell^0(\tau^y(S^y-)),\ell^0(\tau^y(S^y))]$. In fact, we cannot have $S^0 = \ell^0(\tau^y(S^y-))$, since then we would have $\tau^y(S^y-)\in (\tau^0(S^0-),T)$ while $\bX(\tau^y(S^y-)) = y > 0$, which would violate condition (b). We conclude that $S^y = \inf\{s\ge0\colon \ell^0(\tau^y(s)) \ge S^0\}$. Finally,
 $$\ell^0(\tau^y(s)) = \ell^0_{N^{\leq y}_{\textnormal{first}}}(\infty) + \int_{[0,s]\times\Hxc{-}}\ell^{-y}_N(\infty)d\bF^{\leq y}(r,N),$$
 which is measurable in $\ol\cF^y$, as desired.
 
 Condition (b) has the additional consequence that time $T$ occurs in the midst of a (possibly incomplete) bi-clade about level $y$ at local time $S^y$, no later than the jump across level $y$. Thus, the clade that follows at local time $S^y$ is entirely excluded from $\tdN$, so $\tdF^{\geq y}_0 = \restrict{\bF^{\geq y}_0}{[0,S^y)}$.
 
 Appealing to the result of Step 2, $S^y$ is conditionally independent of $\bF^{\geq y}_0$ given $\gamma^y_\infty$. Thus, $\Pr^{(\alpha)}_{\gamma^y_\infty}\{F^{\geq 0}_0\in\cdot\,\}$ is a regular conditional distribution for $\bF^{\geq y}_0$ given $(\gamma^y_\infty,S^y)$. Consequently, for $f$ non-negative and measurable on the appropriate domain,
 $$\EV\big[ f\big(\bF^{\geq y}_0,S^y\big) \big] = \int f(G,s)\Pr^{(\alpha)}_{\beta_\infty}\{F^{\geq 0}_0\in dG\}\Pr\{\gamma^y_\infty\in d\beta_\infty,S^y\in ds\}.$$
 For the purpose of the following, for $(G,s)$ as above we will write $G_{< s} := \restrict{G}{[0,s)}$ and $G_{\geq s} := \shiftrestrict{G}{[s,\infty)}$. Similarly, modifying our earlier notation, for $\beta_\infty = (\beta_n,\ n\geq 1)$ as in Step 2, we will write $\beta_{<s}$ to denote the set of blocks of $\beta_{\infty}$ prior to diversity $s$, and $\beta_{\geq s}$ will denote the remainder, shifted to start at left endpoint zero. More formally, $\beta_{<s} := \{U\in \beta_{\lceil s\rceil}\colon \IPLT_{\beta_{\lceil s\rceil}}(U) < s\}$ and $\beta_{\geq s} := (\beta_{\geq s,n},\ n\geq 1)$ where, for $n\geq 1$,
 $$\beta_{\geq s,n} := \big\{(a-\IPmag{\beta_{<s}},b-\IPmag{\beta_{<s}})\colon (a,b)\in \beta_{\lceil s\rceil + n},\ \IPLT_{\beta_{\lceil s\rceil + n}}(a) \in [s,s+n)\big\}.$$
 Now, consider $f(G,s) := h(G_{<s})$ in our earlier disintegration calculation:
 \begin{equation*}
 \begin{split}
  \EV\big[ h\big(\tdF^{\geq y}_0\big) \big] &= \int h(G_{<s})\Pr^{(\alpha)}_{\beta_\infty}\{F^{\geq 0}_0\in dG\}\Pr\{S^y\in ds,\gamma^y_\infty\in d\beta_\infty\}\\
  		&= \int h(G_{\!<s})\Pr^{(\alpha)}_{\beta_{<s}}\!\{F^{\geq 0}_0\!\in\! dG_{\!<s}\}\Pr^{(\alpha)}_{\beta_{\geq s}}\!\{F^{\geq 0}_0\!\in\! dG_{\geq s}\}
  			\Pr\!\left\{\!\!\!\begin{array}{c}S^y\in ds,\\ \gamma^y_\infty\!\in\! d\beta_\infty\end{array}\!\!\!\right\}\\
  		&= \int h(G_{<s})\Pr^{(\alpha)}_{\beta_{<s}}\{F^{\geq 0}_0\in dG_{<s}\}\Pr\{S^y\in ds,\gamma^y_\infty\in d\beta_\infty\}.
 \end{split}
 \end{equation*}
 The second line above comes from noting that $\restrict{F^{\geq 0}_0}{[0,s)}$ is independent of $\shiftrestrict{F^{\geq 0}_0}{[s,\infty)}$ under $\Pr^{(\alpha)}_{\beta_\infty}$, and the third line comes from integrating out the $\Pr^{(\alpha)}_{\beta_{\geq s}}$ term. Noting that $\td\beta^y = \gamma^y_{< S^y}$, we conclude that $\Pr^{(\alpha)}_{\td\beta^y}\{F^{\geq 0}_0\in\cdot\,\}$ is a regular conditional distribution for $\tdF^{\geq y}_0$ given $(\gamma^y_\infty,S^y)$. We already have the desired conditional independence from $\ol\cF^y$. Finally, since this r.c.d.\ depends only on $\td\beta^y$, it is also an r.c.d.\ given $\td\beta^y$.
\end{proof}

We can now study %prove the existence of certain 
$\nu_{\tt BESQ}^{(-2\alpha)}$-IP-evolutions from certain initial states. It follows from Proposition \ref{prop:PRM:cts_skewer} that $\tdN$ belongs to $\Hfins$ a.s.. 
%MW For the purpose of the following, let $\tdN^*$ denote an $\Hfins$-version of $\tdN$ and 
Let $(\td\beta^{y},\ y\geq 0) := \skewerP\big(\tdN\big)$.

\begin{corollary}\label{cor:type-1:cts_from_Stable}
 Let $S>0$ be independent of $\bN$ 
% \begin{enumerate}[label=(\roman*), ref=(\roman*)]
%  \item 
 and $T := \tau^0(S)$. Then $\td\beta^{0}$ is a \Stable[\alpha]\ interval partition with total diversity $S$ %\label{item:c_f_S:determinstic}
%  \item If $T := \inf\big\{t > 0\colon M^0_{\bN}(t) > S\big\}$ and $S\sim\ExpDist[\rho]$ for some $\rho>0$ then $\alpha^{0,*}$ is distributed like a \PDIP[\frac12,\frac12] multiplied by an independent $\GammaDist[\frac12,\rho]$ scaling factor.\label{item:c_f_S:mass}
% \end{enumerate}
% In either case, 
 and $(\td\beta^{y},\ y\geq 0)$ is a path-continuous $\nu_{\tt BESQ}^{(-2\alpha)}$-IP-evolution.
\end{corollary}

\begin{proof}
 First, the claimed distributions for $\td\beta^0$ %MW and therefore for $\beta^{0,*}$ 
follow from the \Stable[\alpha] description of $M^y_\bN$ in Proposition \ref{prop:agg_mass_subord} and the definitions of the \Stable\ %and \PDIP\ 
interval partition laws in Proposition \ref{prop:IP:Stable}. % and \ref{prop:PDIP}.
 
 Next, note that %in either of the cases, 
 $\tdF^{\geq 0} = \restrict{\bF^{\geq 0}}{[0,\ell^0(T))}$ almost surely. %In case (i) 
 There is a.s.\ no bi-clade of $\bN$ about level 0 at local time $\ell^0(T) = S$. %In case (ii), time $T$ occurs at the middle spindle of a bi-clade, so $\tdN$ cuts off before the clade component of that final incomplete bi-clade. 
 From Proposition \ref{prop:PRM:Fy-_Fy+} applied at level $0$, we see that $\tdF^{\geq 0}_0$ has regular conditional distribution $\Pr^{(\alpha)}_{\td\beta^0}\{F^{\geq 0}_0\in\cdot\,\}$ given $\td\beta^0$. Thus, it has law $\Pr^{(\alpha)}_{\mu}\{F^{\geq 0}_0\in\cdot\,\}$, where $\mu$ is the law of $\td\beta^0$. Therefore, $F^{\geq 0}_0(\tdN)$ has law $\Pr^{(\alpha)}_{\mu}\{F^{\geq 0}_0\in\cdot\,\}$. From Lemma \ref{lem:cutoff_skewer} \ref{item:CPS:clades_skewer_0} and Proposition \ref{prop:nice_level}, since level 0 is a.s.\ nice for $\bN$ and thus for $\tdN$, we conclude that $(\td\beta^{y},\ y\geq 0)$ has law $\Pr^{(\alpha)}_{\mu}\{\skewerP\in\cdot\,\}$. Therefore, it satisfies Definition \ref{constr:type-1} of a $\nu_{\tt BESQ}^{(-2\alpha)}$-IP-evolution and is path-continuous.
\end{proof}

\subsection{$\nu_{\tt BESQ}^{(-2\alpha)}$-IP-evolutions starting from a single block}
\label{sec:type-1:clade}

%Let $\bN$ be a \PRM[\Leb\otimes\mBxc] living on a probability space $(\Omega,\cA,\Pr)$. We continue to use the notation of the first paragraph of Section \ref{sec:biclade_PRM} for objects related to $\bN$. Let $(\ol\cF_t,\,t\ge0)$ and $(\ol\cF^y,\,y\ge0)$ denote $\Pr$-completions of the time- and level-filtrations on $(\Omega,\cA)$ generated by $\bN$, as in Definition \ref{def:filtrations}. That is, these are $\Pr$-completions of the pullbacks, via $\bN\colon \Omega\to\H$, of the time- and level-filtrations on $\H$.

$\!\!$On a suitable probability space $(\Omega,\cA,\Pr)$ let $\bN$ be a \PRM[\Leb\otimes\mBxc_{\tt BESQ}^{(-2\alpha)}]. We continue to use the notation of the first paragraph of Section \ref{sec:biclade_PRM} for objects related to $\bN$. 
%
%Let $\bN$ be a \PRM[\Leb\otimes\mBxc] living on a probability space $(\Omega,\cA,\Pr)$. We continue to use the notation of the first paragraph of Section \ref{sec:biclade_PRM} for objects related to $\bN$. Let $(\ol\cF_t,\,t\ge0)$ and $(\ol\cF^y,\,y\ge0)$ denote $\Pr$-completions of the time- and level-filtrations on $(\Omega,\cA)$ generated by $\bN$, as in Definition \ref{def:filtrations}, augmented to allow an independent random variable $S$ measurable in $\ol\cF^0\cap\ol\cF_0$. That is, these are formed by augmenting the $\Pr$-completions of the pullbacks, via $\bN\colon \Omega\to\H$, of the time- and level-filtrations on $\H$. 
%
%We carry on with the notation $\bN$, $\bX$, and so on, of Section \ref{sec:biclade_PRM}. 
%Fix $a>0$ and let $\bff$ be a \BESQ[-2\alpha] starting from $a$ and absorbed upon hitting zero, independent of $\bN$. 
Fix $a>0$ and let $\bff$ be a \BESQ[-2\alpha] starting from $a$ and absorbed upon hitting zero, independent of $\bN$. 
Let $\olN := \Dirac{0,\bff} + \bN$. We use barred versions of our earlier notation to refer to the corresponding objects associated with $\olN$. For example, $\olX = \bX + \zeta(\bff)$. Let $\ol T^0 = T^{-\life(\bff)}$ denote the first hitting time of 0 by $\olX$ and set $\widehat\bN := \restrict{\olN}{[0,\ol T^0)}$. By Proposition \ref{prop:clade_splitting}, $\widehat\bN$ has distribution $\mClade^+(\,\cdot\;|\;m^0 = a)$. We use hatted versions of our earlier notation to refer to the corresponding objects associated with $\widehat\bN$. Set $(\widehat\beta^y,\,y\geq 0) := \skewerP(\widehat\bN)$.

Let $(\ol\cF_t,\,t\ge0)$ and $(\ol\cF^y,\,y\ge0)$ denote %$\Pr$-completions of the time- and level-filtrations on $(\Omega,\cA)$ generated by $\olN$, as in Definition \ref{def:filtrations}. That is, these are the 
$\Pr$-completions of the pullbacks, via $\olN\colon \Omega\to\H$, of the time- and level-filtrations on $\H$, as in Definition \ref{def:filtrations}.

%MW \begin{definition}\label{def:clade_stats}
 The \emph{lifetime} of a bi-clade $N\in\Hxc{\pm}$ is 
\begin{equation}\label{eqn:life}\life^+(N) := \sup_{t\in[0,\len(N)]} \xi_N(t).
\end{equation}
%MW Its \emph{anti-clade lifetime} is $\life^-(N) := -\inf_{t\in[0,\len(N)]}\!\xi_N(t)$. 
%MW \end{definition}
%MW 
%\begin{figure}
% \centering
% \input{Fig_clade_stats.pdf_t}
% \caption{A bi-clade, with the statistics of \eqref{eq:clade:mass_def} and Definition \ref{def:clade_stats} labeled.\label{fig:clade_stats}}
%\end{figure}
%MW
We call $\life^+$ ``lifetime'' rather than ``maximum'' since values in the scaffolding function play the role of times in the evolving interval partitions $(\skewer(y,N),\ y\geq 0)$ that we ultimately wish to study. %The above quantities appear labeled in Figure \ref{fig:clade_stats}. The rates at which they scale under $\scaleH$ are listed in Table \ref{tbl:clade_scaling}. 
%By Lemma \ref{lem:scl_ker}, $\mClade$ admits unique kernels with scaling properties that allow us to condition on the exact value of any one of these quantities and get a resulting probability distribution.

%\begin{table}
% \centering
% \begin{tabular}{|c|Sc|c|}\hline
%  $J^+(\scaleH[c][N]) = cJ^+(N)$ & $J^-(\scaleH[c][N]) = cJ^-(N)$ & $J(\scaleH[c][N]) = cJ(N)$\\\hline
%  $\life^+(\scaleH[c][N]) = c\life^+(N)$ & $\life^-(\scaleH[c][N]) = c\life^-(N)$ & $m^0(c\scaleH N) = cm^0(N)$\\\hline
%  $T_0^+(c\scaleH N) = c^{1+\alpha}T_0^+(N)$ & $\len(c \scaleH N) = c^{1+\alpha}\len(N)$ & $\ell^y_{c \scaleH N}(t) = c^{\alpha}\ell^{y/c}_{N}(tc^{-1-\alpha})$\\\hline
% \end{tabular}\vspace{4pt}
% \caption{How statistics of \eqref{eq:clade:mass_def} and Definition \ref{def:clade_stats} scale as $N$ is scaled, as in \eqref{eq:clade:xform_def}.\label{tbl:clade_scaling}}
%\end{table}

\begin{corollary}\label{cor:clade:cts_skewer}
 The process $\widehat\bN$ a.s.\ belongs to $\Hfins$. In particular, $(\widehat\beta^y,\ y\geq 0)$ is a $\nu_{\tt BESQ}^{(-2\alpha)}$-IP-evolution starting from $\{(0,a)\}$, and it is a.s.\ H\"older-$\theta$ in $(\IPspace,\dI)$ for every $\theta\in \big(0,\alpha/2\big)$.
% a.s.\ well-defined and continuous in $(\IPspace,\dI)$. Moreover, it is a.s.\ H\"older-$\theta$ in $(\IPspace,\dI)$ for every $\theta\in \big(0,\frac14\big)$.
\end{corollary}

\begin{proof}
 For the purpose of the following let $\td\bN := \restrict{\bN}{[0,T^{-\life(\bff)}]} = \widehat\bN - \Dirac{0,\bff}$. Note that $\td\bN$ is in the regime of processes considered in Section \ref{sec:type-1:PRM}. %, albeit with additional randomization in the form of the independent $\life(\bff)$.
 By Proposition \ref{prop:PRM:cts_skewer}, $\td\bN \in \Hfins$ almost surely. Let $\td\bX := \xi(\tdN)$ and $(\td\beta^y,\,y\ge0) := \skewerP(\tdN)$. Then
 \begin{gather*}
  \widehat\bX = \td\bX + \life(\bff), \quad M^y_{\widehat\bN}\left(\ol T^0\right) = M^{y-\life(\bff)}_{\td\bN}\left(\ol T^0\right) + \bff(y),\\
  \text{and} \quad \widehat\beta^y = \big\{(0,\bff(y))\big\} \concat \td\beta^{y-\life(\bff)}.
 \end{gather*}
 By Definition \ref{def:domain_for_skewer}, in order to have $\widehat\bN\in\Hfins$ we require that $M^y_{\widehat\bN}\left(\ol T^0\right) < \infty$, 
 that $(\widehat\beta^y,\,y\ge0)$ is continuous in $y$, and that  
 $$\wh\ell^y(t) = \IPLT_{\widehat\beta^y}\left(M^y_{\widehat\bN}(t)\right)\qquad \text{for }t\geq 0,\ y\in\left(0,\life^+\left(\whN\right)\right),$$
 where $\life^+$ denotes clade lifetime, as in %MW Definition \ref{def:clade_stats}.
\eqref{eqn:life}.  
 % (i) $M^y_{\widehat\bN}(\ol T^0) < \infty$ for all $y\in\BR$, (ii) $\hat\ell^y(t) = \IPLT_{\widehat\alpha^y}\big(M^y_{\widehat\bN}(t)\big)$ for all $y\in\BR$ and $t\geq 0$, and (iii) $(\widehat\alpha^y)$ evolves continuously in $y$.
 In light of the connections between $\widehat\bN$ and $\td\bN$ mentioned above, these three properties follow from the corresponding properties for $\td\bN$, noted in Proposition \ref{prop:PRM:cts_skewer}. That proposition further implies that $(\td\beta^y)$ is a.s.\ H\"older-$\theta$ for $\theta\in \big(0,\alpha/2\big)$. By Lemma \ref{lem:BESQ:Holder}, $\bff$ is a.s.\ H\"older-$\theta$ for $\theta\in \big(0,\frac12\big)$. Thus, $(\widehat\beta^y)$ is a.s.\ formed by concatenating two H\"older-$\theta$ processes, so the claimed H\"older continuity follows from Lemma \ref{lem:IP:concat} on concatenation. 
\end{proof}

\begin{lemma}[Lifetime of a $\nu_{\tt BESQ}^{(-2\alpha)}$-IP-evolution from $\{(0,a)\}$]\label{prop:type-1:transn}
 The lifetime of $(\widehat\beta^y,y\ge 0)$ has \InvGammaDist[1,a/2] distribution, i.e.\vspace{-0.1cm}
 \begin{equation}
  \Pr\big\{\life^+\big(\whN\big)> y\big\} = \Pr(\widehat\beta^y \not= \emptyset) = 1 - e^{-a/2y} \qquad \text{for }y>0.\label{eq:transn:lifetime}\vspace{-0.1cm}
 \end{equation}
\end{lemma}

\begin{proof}
 By construction, $\life(\bff)$ is independent of $\bN$. Thus, by Proposition \ref{prop:nice_level}, level $y$ is a.s.\ nice for $\olX$; henceforth we restrict to that event. By Proposition \ref{prop:bi-clade_PRM} and the aforementioned independence, the point process $\olF^y = \bF^{y-\life(\bff)}$ is a \PRM[\Leb\otimes\mClade]. Let $\widehat S^y := \oll^y(\ol T^0)$. If $\widehat\bN$ survives past level $y$ then $\widehat S^y$ is the level $y$ local time at which some excursion of $\olX$ about level $y$ first reaches down to level zero:\vspace{-0.1cm}
 \begin{equation*}
  \widehat S^y = \cf\big\{\life^+\big(\whN\big) > y\big\}\inf\left\{s\!>\!0\colon \olF^y\big([0,s]\times\{N\!\in\!\Hxc{\pm}\colon \life^-(N) \geq y\}\big) > 0 \right\}.\vspace{-0.1cm}
 \end{equation*}%\label{eq:LT_y_clade_death}
 Conditionally given the event $\{\life^+(\widehat\bN) > y\}$ of survival beyond level $y$, it follows from the Poisson property of $\olF^y$ that  
 $\widehat S^y>0$ a.s. In light of this, up to null events,\vspace{-0.1cm}
 \begin{equation}\label{eq:type-1:no_revival}
  \left\{\life^+\left(\widehat\bN\right) \leq y\right\} = \left\{\wh S^y = 0\right\} = \left\{\widehat\bF^{\geq y} = 0\right\} = \left\{\widehat\beta^y = \emptyset\right\}.\vspace{-0.1cm}
 \end{equation}
 Recall from Lemma \ref{lem:BESQ:length} the law of $\life(\bff)$. By the two-sided exit problem for spectrally one-sided \StableA\ processes (e.g. 
 \cite[Theorem VII.8]{BertoinLevy}), $x+\bX$ exits $[0,y]$ at 0 with probability $(1-x/y)^\alpha$, for all $x\in(0,y)$. Hence,\vspace{-0.1cm}
 \begin{align*}\Pr\big\{\life^+\big(\whN\big)\le y\big\}&=\int_0^y\frac{a^{1+\alpha}}{\Gamma(1+\alpha)2^{1+\alpha}}x^{-2-\alpha}e^{-a/2x}(1-x/y)^\alpha dx\\
   &=\frac{a^{1+\alpha}y^{-1-2\alpha}}{\Gamma(1+\alpha)2^{1+\alpha}}\int_0^1u^{-2-\alpha}(1-u)^\alpha e^{-a/2uy}du.\vspace{-0.1cm}
 \end{align*}
 Setting $y=a/2z$, we need to show that this equals $e^{-z}$. This follow by calculating the Mellin transform\vspace{-0.1cm}
 \begin{align*}&\int_0^\infty z^r\frac{z^{1+2\alpha}}{\Gamma(1+\alpha)}\int_0^1u^{-2-\alpha}(1-u)^\alpha e^{-z/u}dudz\\
   &=\frac{1}{\Gamma(1+\alpha)}\int_0^1u^{-2-\alpha}(1-u)^\alpha\Gamma(2+2\alpha+r)u^{2+2\alpha+r}du=\Gamma(r+1),\vspace{-0.1cm}
 \end{align*}
 which is the Mellin transform of $e^{-z}$. Since the exponential distribution is determined by its moments, this completes the proof.
\end{proof}

%As an aside, note that Lemma \ref{lem:LMB} is equivalent, via Proposition \ref{prop:clade:stats} \ref{item:CS:max} and \ref{item:CS:mass:max}, to the following rate formula:
%\begin{equation*}
% \mClade\{m^0\in da,\ m^y\in db\} = \frac{1}{\pi\sqrt{8}}\frac{\sqrt{y}}{(ab)^{3/2}}e^{-\frac{a+b}{2y}}\left(1 - \cosh\left(\frac{\sqrt{ab}}{y}\right) + \frac{\sqrt{ab}}{y}\sinh\left(\frac{\sqrt{ab}}{y}\right)\right)dadb
%\end{equation*}
%This formula has the aesthetic advantage of being symmetric in $a$ and $b$; this corresponds to $\reverseH$ distributional symmetry of segments of $\bN$ corresponding to last-exit, first passages of $\bX$ from level $0$ to level $y$. We won't go into detail, as we note this formula only out of interest.

We now extend the Markov-like property of Proposition \ref{prop:PRM:Fy-_Fy+} to the present setting.

\begin{proposition}\label{prop:clade:Fy-_Fy+}
 $\widehat\bF^{\geq y}_0$ is conditionally independent of $\ol\cF^y$ given $\widehat\beta^y$ with regular conditional distribution $\Pr^{(\alpha)}_{\widehat\beta^y}\big\{ F^{\geq 0}_0\in \cdot\,\big\}$, where this law is as in Definition \ref{constr:type-1_2}.
\end{proposition}

\begin{proof}
 By \eqref{eq:type-1:no_revival}, the claimed regular conditional distribution holds trivially on $\big\{\life^+\big(\widehat\bN\big) \leq y\big\}$. Likewise, the result is trivial for $y=0$.
 
 The mid-spindle Markov property for $T^{\geq y}$, Lemma \ref{lem:mid_spindle_Markov}, may be extended from $\bN$ to apply to $\olN$. Indeed, if $\ol T^{\geq y} > 0$ then the same proof goes through; otherwise, if $\ol T^{\geq y} = 0$, i.e.\ if $\life(\bff) > y$, then the lemma reduces to the Markov property of $\bff$ at $y$. We use this extension to split $\olN$ into three segments.
 
 Let $T := \ol T^{\geq y}$. Let $(T,f_T)$ denote the spindle of $\ol\bN$ at this time, which equals $(0,\bff)$ if $\life(\bff) > y$. Let $\hat f^y_T$ and $\check f^y_T$ denote the broken spindles of \eqref{eq:spindle_split}. Extending the notation of Definition \ref{def:bi-clade_PP} (and Appendix \ref{app:exc_intervals}), set
 \begin{equation*}
  \olN^{\leq y}_{\textnormal{first}} := \Restrict{\olN}{[0,T)} + \Dirac{T,\check f_T^y},
  \ \olN^{\geq y}_{\textnormal{first}} := \ShiftRestrict{\olN}{\left(T,\ol T^y\right)} + \Dirac{0,\hat f_T^y},
  \ \olN^y_* := \ShiftRestrict{\olN}{[\ol T^y,\infty)}.
 \end{equation*}%\label{eq:clade:triple_split}
 Let $H_1$ and $H_2$ be non-negative measurable functions on $\Hfin$, and likewise for $H_3$ on $\H$. Recall \eqref{eq:LMB_def} defining $m^y$. In the present setting $m^y(\olN) = f_T\big(y-\olX(T-)\big)$. By the preceding extension of the mid-spindle Markov property and the disintegration of $\mClade^+$ in Proposition \ref{prop:clade_splitting},
 \begin{equation*}
  \EV \left[ H_1\big(\olN^{\leq y}_{\textnormal{first}}\big)H_2\big(\olN^{\geq y}_{\textnormal{first}}\big) \right] = \EV \left[ H_1\big(\olN^{\leq y}_{\textnormal{first}}\big) \mClade^+\left[H_2\;\middle|\;m^0 = m^y\big(\olN^{\leq y}_{\textnormal{first}}\big)\right] \right],
 \end{equation*}%\label{eq:clade:Fy+-_indep_2}
 Moreover, by the strong Markov property of $\olN$ applied at $\ol T^y$, $(\olN^{\leq y}_{\textnormal{first}},\olN^{\geq y}_{\textnormal{first}})$ is independent of $\olN^y_*$ and the latter is distributed like $\bN$. Thus,
 \begin{align*}
  &\EV \left[ H_1\big(\olN^{\leq y}_{\textnormal{first}}\big)H_2\big(\olN^{\geq y}_{\textnormal{first}}\big)H_3\big(\olN^y_*\big) \right] \\ 
  &= \EV \left[ H_1\big(\olN^{\leq y}_{\textnormal{first}}\big) \mClade^+\left[H_2\;\middle|\;m^0 = m^y\big(\olN^{\leq y}_{\textnormal{first}}\big)\right] \right] \EV\big[ H_3(\bN) \big].
 \end{align*}%\label{eq:clade:Fy+-_indep_split}
 The event $\{\life^+(\widehat\bN) > y\}$ equals the event that the process $\xi\big(\olN^{\leq y}_{\textnormal{first}}\big)$ is non-negative. % $\big\{\xi\big(\olN^{\leq y}_{\textnormal{first}}\big) \geq 0\big\}$. 
 In particular, this belongs to the $\sigma$-algebra $\sigma\big(\olN^{\leq y}_{\textnormal{first}}\big)$. Thus, the above formula also holds for the conditional expectation given this event. On this event, $\ol T^{\geq y} = \widehat T^{\geq y} < \infty$.
 
 Let $\tdN := \restrict{\olN^y_*}{[0,T^{-y}(\olN^y_*))}$. The stopping time $T^{-y}(\bN)$ satisfies the hypotheses of Proposition \ref{prop:PRM:Fy-_Fy+}. Thus, that proposition applies to the stopped \PRM\ $\tdN$. On the event $\{\life^+(\widehat\bN) > y\}$, which is independent of $\tdN$, we have $\tdN = \shiftrestrict{\whN}{[\wh T^y,\wh T^0)}$ and so $\whF^{y} = F^{0}_0(\tdN)$. Conditionally given this event, by Proposition \ref{prop:PRM:Fy-_Fy+}, the clade point process $\whF^{\geq y}$ is conditionally independent of $\olF^{\leq y}$ given $\skewer\big(0,\tdN\big) =: \beta$, with regular conditional distribution $\Pr^{(\alpha)}_{\beta}\{F^{\geq 0}_0\in\cdot\,\}$. It follows from Proposition \ref{prop:nice_level} that level $y$ is a.s.\ nice for $\ol\bN$. Thus, by Lemma \ref{lem:cutoff_vs_PP}, $\big(\olN^{\leq y}_{\textnormal{first}},\olF^{\leq y}\big)$ generates $\ol\cF^y$ up to $\Pr$-null sets. Putting all of this together, $\whF^{\geq y}_0 = \whF^{\geq y} + \cf\big\{\life^+\big(\widehat\bN\big) > y\big\}\DiracBig{0,\olN^{\geq y}_{\textnormal{first}}}$ a.s., and this has the desired conditional independence and regular conditional distribution.
\end{proof}

\section{$\nu_{q,c}^{(-2\alpha)}$-IP-evolutions with arbitrary initial states}\label{sec:arbitraryinitialdata}

\subsection{Simple Markov property of $\nu_{\tt BESQ}^{(-2\alpha)}$-IP-evolutions}
\label{sec:type-1_gen:cts}

In this section we fix $\alpha\in(0,1)$ and follow the notation of Definition \ref{constr:type-1} for $c=q=1$. Let $\beta\in\IPspace:=\IPspace_\alpha$, $(\bN_U,\,U\in\beta)$, $\bN_{\beta}$, and $(\beta^y,\,y\geq 0)$. We treat these objects as maps on a probability space $(\Omega,\cA,\Pr)$. We additionally define
\begin{equation}\label{eq:concat_skewers:type-1}
 \beta^y_U := \skewer(y,\bN_U) \quad \text{for }y\ge0,\ U\in\beta, \qquad \text{so} \qquad \beta^y = \ConcatIL_{U\in\beta}\beta^y_U.
\end{equation}
For each of the filtrations $(\cF_t)$, $(\cF_{t-})$, $(\cF^y)$, and $(\cF^{y-})$ on $\H$ introduced in Definition \ref{def:filtrations}, we accent with a bar, as in $(\ol\cF_t,\ t\geq 0)$, to denote the completion of the filtration under the family of measures $(\Pr^{(\alpha)}_{\beta},\ \beta\in\IPspace)$.

We begin this section by showing that $(\beta^y,\,y\geq 0)$ is a.s.\ an $\IPspace$-valued process, and we derive its transition kernel. Then we prove a simple Markov property of $(\beta^y,\,y\geq 0)$ as a random element of the product space $\IPspace^{[0,\infty)}$. Finally, we prove the existence of a continuous version of $(\beta^y,\,y\geq 0)$ 
%, in the sense of Definition \ref{def:version}, 
as well as a simple Markov property for this continuous process.

\begin{lemma}\label{lem:finite_survivors}
 For $(\beta^y_U,\,y\ge0,\,U\in\beta)$ as above, 
 $\EV [\#\{U\in J\colon \beta_U^y \not= \emptyset\} ] \leq \frac{1}{2y}\sum_{U\in J}\Leb(U)$ and $\Pr \{\forall U\in J,\,\beta_U^y = \emptyset \} \geq 1 - \frac{1}{2y}\sum_{U\in J}\Leb(U)$ for all $J\subseteq\beta$ and $y>0$. 
 In particular, a.s.\ only finitely many of the $\big(\beta_U^z,\,z\geq 0\big)$ survive to level $y$.
\end{lemma}

\begin{proof}
 The variables $\cf\{\beta^y_U = \emptyset\}$ are independent Bernoulli trials with respective parameters $e^{-\Leb(U)/2y}$, by \eqref{eq:transn:lifetime}. Thus, both inequalities follow from $e^{-x} \geq 1-x$.
\end{proof}

We can extend Theorem \ref{thm:LT_property_all_levels} to the present setting.

\begin{proposition}\label{prop:type-1:LT_diversity} 
 It is a.s.\ the case that $\IPLT_{\beta^y}\big(M^y_{\bN_{\beta}}(t)\big) = \ell^y_{\bN_{\beta}}(t)$ for all $t\geq 0$, $y>0$.
\end{proposition}

\begin{proof}
 Appealing to Corollary \ref{cor:clade:cts_skewer} and Lemma \ref{lem:finite_survivors}, we may restrict to an a.s.\ event on which:
 \begin{equation}\label{eq:good_clades_event}
  \forall U\!\in\!\beta,\ \bN_U\!\in\!\Hfins, \text{ and } \forall n\!\in\!\BN,\ \#\left\{V\!\in\!\beta\colon \life^+\left(\bN_V\right)\!>\!1/n\right\} < \infty.
 \end{equation}
 Let $y>0$ and consider the left-to-right ordered sequence $U_1,\ldots,U_K$ of intervals $U\in\beta$ for which $\life^+(\bN_{U})>y$. For $U\in\beta$, define $S(U-):=\sum_{V\in\beta\colon V<U}\len(\bN_V)$ and $S(U):=S(U-) + \len(\bN_U)$.
 
 Since no clade prior to time $S(U_1-)$ survives to level $y$, $\IPLT_{\beta^y}\big(M^y_{\bN_{\beta}}(t)\big) = \ell^y_{\bN_{\beta}}(t) = 0$ for $t\leq S(U_1-)$. We assume for induction that the same holds up to time $S(U_j-)$. Then for all $t\in [S(U_j-),S(U_j)]$
 \begin{equation*}
 \begin{split}
  \IPLT_{\beta^y}\big(M^y_{\bN_{\beta}}(t)\big) &= \IPLT_{\beta^y}\big(M^y_{\bN_{\beta}}(S(U_j-))\big) + \IPLT_{\beta^y_{U_j}}\big(M^y_{\bN_{U_j}}(t-S(U_j-))\big)\\
  		&= \ell^y_{\bN_{\beta}}(S(U_j-)) + \ell^y_{\bN_{U_j}}(t-S(U_j-)) = \ell^y_{\bN_{\beta}}(t),
 \end{split}
 \end{equation*}
 where the middle equality follows from our assumption $\bN_{U_j}\in\Hfins$ and the inductive hypothesis. For $t\in [S(U_j),S(U_{j+1}-)]$ or, if $j=K$, for all $t \geq S(U_j)$, no additional local time accrues and at most one skewer block arrives at level $y$ during this interval. Thus, on this interval,
 $$\IPLT_{\beta^y}\big(M^y_{\bN_{\beta}}(t)\big) = \IPLT_{\beta^y}\big(M^y_{\bN_{\beta}}(S(U_j))\big) = \ell^y_{\bN_{\beta}}(S(U_j)) = \ell^y_{\bN_{\beta}}(t).$$
 By induction, this proves that the identity holds at all $t\geq 0$ at level $y$, for all $y>0$.
\end{proof}

\begin{lemma}\label{lem:type-1:wd}
 It is a.s.\ the case that for every $y>0$, the collection of interval partitions $(\beta_U^y,U\!\in\! \beta)$ is strongly summable in the sense defined above Lemma
 \ref{lem:IP:concat}, %MW of Definition \ref{def:IP:concat}, 
 i.e.\ $\beta^y\!=\!\Concat_{U\in\beta}\beta_U^y$ is well-defined and lies in $\IPspace$.
\end{lemma}

\begin{proof}
 This holds on the event in \eqref{eq:good_clades_event}, as finite sequences in $\IPspace$ are strongly summable.
\end{proof}

\begin{proposition}[Transition kernel for $\nu_{\tt BESQ}^{(-2\alpha)}$-IP-evolutions]\label{prop:type-1:gen_transn}
 Fix $y\!>\!0$. Let $(\gamma^y_U,U\!\in\!\beta)$ denote an independent family of partitions, with each $\gamma^y_U$ a $\nu_{\tt BESQ}^{(-2\alpha)}$-IP-evolution
 starting from the single interval $\{(0,\Leb(U))\}$ at time $y$. Then $\skewer(y,\bN_{\beta})\stackrel{d}{=}\ConcatIL_{U\in\beta}\gamma^y_U$, and this law is supported on $\IPspace$.
\end{proposition}

\begin{proof}
 This follows from Definition \ref{constr:type-1_2} of $\Pr^{(\alpha)}_{\beta}$ via the observation that the skewer map commutes with concatenation of clades. By Lemma \ref{lem:type-1:wd}, the resulting law is supported on $\IPspace$.
\end{proof}

\begin{lemma}\label{lem:type-1:nice_lvl}
 For $y\!>\!0$, it is a.s.\ the case that level $y$ is nice for $\xi(\bN_{\beta})$ in the sense of Proposition \ref{prop:nice_level} and $\beta^y$ is nice in the sense of Lemma \ref{lem:cutoff_skewer} \ref{item:CPS:clades_skewer_0}.
\end{lemma}

\begin{proof}
 Proposition \ref{prop:nice_level} implies that for each $U$, level $y$ is a.s.\ nice for $\xi(\bN_U)$. It follows from this and Lemma \ref{lem:finite_survivors} that $y$ is a.s.\ nice for $\xi(\bN_{\beta})$. In particular, no two level $y$ excursion intervals arise at the same local time. Proposition \ref{prop:agg_mass_subord} characterizes a correspondence between level $y$ excursion intervals of $\xi(\bN_{\beta})$, including the incomplete first excursion interval, and blocks in $\beta^y$ whereby, via Proposition \ref{prop:type-1:LT_diversity}, the diversity up to each block $U\in\beta^y$ equals the level $y$ local time up to the corresponding excursion interval. Thus, $\beta^y$ is a.s.\ nice as well.
\end{proof}

We now extend the Markov-like property of Propositions \ref{prop:PRM:Fy-_Fy+} and \ref{prop:clade:Fy-_Fy+} to the present setting. Also recall notation from the beginning of this 
section.

\begin{proposition}\label{prop:type-1:Fy-_Fy+}
 For $y>0$, the point process $F^{\geq y}_0(\bN_\beta)$ is conditionally independent of $\ol\cF^y$ given $\beta^y$, with regular conditional distribution (r.c.d.)  $\Pr^{(\alpha)}_{\beta^y}(F^{\geq 0}_0\in\cdot\,)$.
\end{proposition}

\begin{proof}
 By Lemma \ref{lem:type-1:nice_lvl}, we may restrict to the a.s.\ event that level $y$ is nice. For $U\in\beta$ let $\bF^{\geq y}_{0,U} := F^{\geq y}_0(\bN_U)$. By Proposition \ref{prop:clade:Fy-_Fy+} and the independence of the family $(\bN_U,\,U\in\beta)$, the process $\bF^{\geq y}_{0,U}$ is conditionally independent of $\ol\cF^y$ given $\beta^y_U$, with r.c.d.\ $\Pr^{(\alpha)}_{\beta^y_U}\{F^{\geq 0}_0\in\cdot\,\}$, for each $U\in\beta$. By Lemma \ref{lem:finite_survivors}, only finitely many of the $\bF^{\geq y}_{0,U}$ are non-zero, so $F^{\geq 0}_0(\bN_{\beta}) = \ConcatIL_{U\in\beta}\bF^{\geq y}_{0,U}$. In light of this, the claimed conditional independence and r.c.d.\ follow from Definition \ref{constr:type-1_2} of the kernel $\gamma\mapsto\Pr^{(\alpha)}_{\gamma}\{F^{\geq 0}_0\in\cdot\,\}$.
\end{proof}

\begin{corollary}[Simple Markov property for the skewer process under $\Pr^{(\alpha)}_{\mu}$]\label{cor:type-1:simple_Markov_1}
 Let $\mu$ be a probability distribution on $\IPspace$. Take $z>0$ and $0\leq y_1<\cdots<y_n$. Let $\eta\colon \Hfin \to [0,\infty)$ be $\ol\cF^{z}$-measurable. Let $f\colon\IPspace^n\to [0,\infty)$ be measurable. Then
  \begin{align*}
   &\Pr^{(\alpha)}_\mu\left[\eta f\left(\skewer(z\!+\!y_j,\cdot),j\!\in\![n]\right) \right]\\
     &=\! \int\!\eta(N)\Pr^{(\alpha)}_{\skewer(z,N)}\left[f\left(\skewer(y_j,\cdot),j\!\in\![n]\right)\right]d\Pr^{(\alpha)}_{\mu}(N).
  \end{align*}
\end{corollary}

\begin{proof}
 By Proposition \ref{prop:type-1:Fy-_Fy+}, for $\eta$ as above and $g\colon \cNRHf \to [0,\infty)$ measurable,
 \begin{equation}\label{eq:Fy-_Fy+:Markov}
  \Pr^{(\alpha)}_\mu \left[ \eta\, g\big(F^{\geq z}_0\big)\right] = \Pr^{(\alpha)}_\mu \left[ \eta\, \Pr^{(\alpha)}_{\skewer(z,\cdot\,)}\big[g\big(F^{\geq 0}_0\big)\big]\right].
 \end{equation}
 By Lemma \ref{lem:cutoff_skewer} \ref{item:CPS:clades_skewer}, there is a measurable function $h$ for which we have $\left(\skewer(z+y_j,N),\,j\in[n]\right) = h(F^{\geq z}_0(N))$ identically on the event that level $z$ is nice for $N\in\Hfin$. Moreover, if $\beta\in\IPspace$ is nice in the sense 
 of Lemma \ref{lem:cutoff_skewer} \ref{item:CPS:clades_skewer_0}, 
 then that result gives $\left(\skewer(y_j,\bN_{\beta}),\,j\in[n]\right) = h(F^{\geq 0}_0(\bN_{\beta}))$. By Lemma \ref{lem:type-1:nice_lvl}, for $\bN_{\mu}\sim\Pr^{(\alpha)}_{\mu}$, level $z$ is a.s.\ nice for $\bN_{\mu}$ and $\skewer(y_1,\bN_{\mu})$ is a.s.\ a nice interval partition. Thus, setting $g := f\circ h$ in \eqref{eq:Fy-_Fy+:Markov} gives the claimed result.
\end{proof}

\subsection{Path-continuity and Markov property of $\nu_{\tt BESQ}^{(-2\alpha)}$-IP-evolutions}

We proceed towards proving continuity of $(\beta^y,\,y\ge0)$. We require the following.

\begin{lemma}\label{lem:IP:domination}
 Fix $\beta\in\IPspace$ and $\delta>0$, and let $\gamma$ denote a \Stable[\alpha] interval partition with total diversity $\IPLT_{\gamma}(\infty) = \IPLT_{\beta}(\infty)+\delta$, as in Proposition \ref{prop:IP:Stable}. Then with positive probability, there exists a matching between their blocks such that every block of $\beta$ is matched with a larger block in $\gamma$. (This is not a correspondence as used in Definition \ref{def:IP:metric}, as it need not respect left-right order.) In this event, we say $\gamma$ \emph{dominates} $\beta$. If, on the other hand, $\gamma$ is a \Stable[\alpha] interval partition with total diversity $\IPLT_{\gamma}(\infty) = \IPLT_{\beta}(\infty)-\delta$ then with positive probability it is dominated by $\beta$.
\end{lemma}

\begin{proof}
 We begin with the case $\IPLT_{\gamma}(\infty) = \IPLT_{\beta}(\infty)+\delta$. We will abbreviate $\IPLT := \IPLT_{\beta}(\infty)$. By the diversity properties of these two partitions,\vspace{-0.2cm}
 \begin{equation*}
  \lim_{h\downto 0}h^{\alpha}\#\{U\in \beta\colon \Leb(U)>h\} = \frac{1}{\Gamma(1-\alpha)}\IPLT\vspace{-0.1cm}
 \end{equation*}
 and $\displaystyle\qquad\qquad 
  \lim_{h\downto 0}h^{\alpha}\#\{V\in \gamma\colon \Leb(V)>h\} = \frac{1}{\Gamma(1-\alpha)}(\IPLT+\delta)$. 
  
 Thus, there is a.s.\ some $H > 0$ sufficiently small so that
 \begin{equation}\label{eq:IP:domination}
  \#\{U\in \beta\colon \Leb(U)>h\} < \#\{V\in \gamma\colon \Leb(V)>h\} \quad \text{for all }h<H.
 \end{equation}
 Take $a>0$ sufficiently small that this holds for $H=a$ with positive probability. It follows from the definition of the \Stable[\alpha] 
 interval partition that, conditionally given that \eqref{eq:IP:domination} holds for $H = a$, there is positive probability that all of the blocks 
 in $\gamma$ with mass greater than $a$ also have mass greater than that of the largest block of $\beta$. In particular, there is positive 
 probability that $\gamma$ dominates $\beta$ by matching, for each $n\geq 1$, the $n^{\text{th}}$ largest block of $\beta$ with that of $\gamma$.
 
 If we instead take $\IPLT_{\gamma}(\infty) = \IPLT_{\beta}(\infty)-\delta$ then there is a.s.\ some $H>0$ such that \eqref{eq:IP:domination} holds 
 in reverse. Let $a$ be as before. Conditionally given that the reverse of \eqref{eq:IP:domination} holds for $H = a$, there is positive 
 probability that no blocks in $\gamma$ have mass greater than $a$. In this event, $\beta$ dominates $\gamma$ by matching blocks in ranked order, 
 as in the previous case.
\end{proof}

\begin{proposition}\label{prop:cts_lt_at_0}
 The diversity process $(\IPLT_{\beta^y}(\infty),\,y\geq 0)$ of a $\nu_{\tt BESQ}^{(-2\alpha)}$-IP-evolution $(\beta^y,y\!\ge\! 0)$ starting from $\beta^0\!=\!\beta\!\in\!\IPspace$ is a.s.\ continuous at $y \!=\! 0$.
\end{proposition}

\begin{proof}
 %Fix $\theta\in (0,\alpha/2)$, 
 Take $\delta > 0$, and abbreviate $\IPLT := \IPLT_{\beta}(\infty)$. Following the notation and situation of Corollary \ref{cor:type-1:cts_from_Stable}, let $\tdN$ denote an $\Hfins$-version of a \PRM[\Leb\times\mBxc_{\tt BESQ}^{(-2\alpha)}] stopped at an inverse local time $\tau^0(\IPLT+\delta)$ and let $(\td\beta^y,\,y\ge0) := \skewerP(\tdN)$. %We set $S := \IPLT+\delta$. 
 Then, as in Corollary \ref{cor:type-1:cts_from_Stable}, $\td\beta^0$ is a \Stable[\alpha] interval partition with total diversity $\IPLT+\delta$. By Lemma \ref{lem:IP:domination}, $\td\beta^0$ dominates $\beta$ with positive probability. Since $\beta$ is deterministic, this domination event is independent of $(\beta^y,\,y\ge0)$. We condition on this event.
 
 We now define an alternative construction of $(\beta^y)$, coupled with $(\td\beta^y)$. Let $(U_i)_{i\geq 1}$ and $(V_i)_{i\geq 1}$ denote the blocks of $\beta$ and $\td\beta^0$ respectively, each ordered by non-increasing Lebesgue measure, with ties broken by left-to-right order. For each $i$ let $\tdN_{V_i}$ denote the clade of $\tdN$ corresponding to that block. By Proposition \ref{prop:PRM:Fy-_Fy+} %Corollary \ref{cor:type-1:cts_from_Stable} \ref{item:c_f_S:determinstic} 
 the $(\tdN_{V_i})_{i\ge1}$ are conditionally independent given $\td\beta^0$, with conditional laws $\mClade^+(\cdot\ |\ m^0 = \Leb(V_i))$. Then\vspace{-0.1cm}
 \begin{equation*}
  \td\beta^y = \Concat_{V\in \td\beta^0} \td\beta_V^y \qquad \text{where} \qquad \left(\td\beta_V^y,\,y\ge0\right) = \skewerP\big(\tdN_V\big).\vspace{-0.1cm}
 \end{equation*}
 
 Let $(0,g_i)$ denote the left-most point in $\tdN_{V_i}$. This is the spindle associated with the block $V_i$. Conditionally given $V_i$, the process $g_i$ is a \BESQ[-2\alpha] starting from $\Leb(V_i)$. We define\vspace{-0.1cm}
 \begin{equation*}
  f_i := \frac{\Leb(U_i)}{\Leb(V_i)}\scaleB g_i,\quad\mbox{and}\quad
  \bN_{U_i} := \Dirac{0,f_i} + \ShiftRestrict{\tdN_{V_i}}{(T_i,\infty)},\vspace{-0.1cm}
 \end{equation*}
 where $ T_i := \inf\left\{t\geq 0\colon \xi_{\tdN_{V_i}}(t) \leq \life(f_i)\right\}$.
 To clarify, $\bN_{U_i}$ is obtained from $\tdN_{V_i}$ by scaling down its leftmost spindle $g_i$ to get $f_i$ and cutting out the segment of $\tdN_{V_i}$ corresponding to the first passage of $\xi(\tdN_{V_i})$ down to level $\life(f_i)$. 
 From \BESQ\ scaling and the Poisson description of the laws $\mClade^+(\,\cdot\;|\;m^0)$ in Proposition \ref{prop:clade_splitting}, it follows that the $(\bN_{U_i})_{i\ge1}$ are jointly independent and have respective laws $\bN_{U_i} \sim \mClade^+(\cdot\ |\ m^0 = \Leb(U_i))$. As in \eqref{eq:concat_skewers:type-1} we define\vspace{-0.1cm}
 \begin{equation*}
  \beta^y := \Concat_{U\in\beta} \beta_U^y \qquad \text{where} \qquad \left(\beta_{U}^y,\,y\ge0\right) = \skewerP(\bN_U) \quad \text{for }U\in \beta.\vspace{-0.1cm}
 \end{equation*}
 The resulting $(\beta^y,\,y\geq 0)\sim\Pr^{(\alpha)}_{\beta}\{\skewerP\in\cdot\,\}$. %is distributed like the object of Construction \ref{constr:type-1}.
 By virtue of this coupling, having conditioned on $\td\beta^0$ dominating $\beta$, it is a.s.\ the case that $\IPLT_{\beta_{U_i}^y}(\infty) \leq \IPLT_{\td\beta_{V_i}^y}(\infty)$ for $i\geq 1$, $y\geq 0$. Thus, by the continuity in Proposition \ref{prop:PRM:cts_skewer},\vspace{-0.1cm}
 $$\limsup_{y\downto 0}\IPLT_{\beta^y}(\infty) \leq \limsup_{y\downto 0}\IPLT_{\td\beta^y}(\infty) = \IPLT + \delta\quad \text{a.s.}.\vspace{-0.1cm}$$
 Since this holds for all $\delta>0$, the left hand side expression is a.s.\ bounded above by $\IPLT$.
 
 If we repeat this argument but $\tdN$ stopped at $\tau^0(\IPLT-\delta)$ then we can condition on $\beta$ dominating $\td\beta^0$ and reverse roles in the above coupling to show that\vspace{-0.1cm}
 \begin{equation*}
  \liminf_{y\downto 0}\IPLT_{\beta^y}(\infty) \geq \liminf_{y\downto 0}\IPLT_{\td\beta^y}(\infty) = \IPLT - \delta\vspace{-0.1cm}
 \end{equation*}
 almost surely for any positive $\delta$. The desired result follows.
\end{proof}

\begin{proposition}\label{prop:ex-BESQ_mass_temp} In the setting of Proposition \ref{prop:cts_lt_at_0}, the total mass process $(\IPmag{\beta^y},y\ge 0)$ is a.s.\ continuous.
\end{proposition}
\begin{proof} Consider $\bN_\beta$ as in Definition \ref{constr:type-1}. We show separately the continuity of $\sum_{U\in\beta}\bff_U$ and of the total mass
  process of the remaining spindles. 

  For the former, we recall from \cite[p.442]{RevuzYor} that \BESQ[-2\alpha] has scale function
  $s(x)=x^{1+\alpha}$. Therefore, the amplitude $A$ has distribution $\Pr(A(\bff_U)>m)=(a/m)^{1+\alpha}$, where $a=\Leb(U)$, so that
  $\EV[A(\bff_U)]=a/\alpha$. Since we have $\sum_{U\in\beta}A(\bff_U)<\infty$ a.s., continuity of $\sum_{U\in\beta}\bff_U$ follows. 

  For the remaining spindles, we use the coupling of point measures $\bN_{U_i}$ and $\tdN$ of the proof of Proposition \ref{prop:cts_lt_at_0}. 
  with $\IPLT_\gamma(\infty)=\IPLT_\beta(\infty)+\delta$ and note that all unbroken spindles of $\bN_{U_i}$, $i\ge 1$, are positioned by the     
  associated scaffoldings $\bX_{U_i}=\xi(\bN_{U_i})$, $i\ge 1$, at the same levels as the corresponding spindles of $\tdN$. By Lemma 
  \ref{lem:spindle_piles} and the proof of Corollary \ref{cor:PRM:type-1_wd}, the associated total mass process is continuous.  
\end{proof}

\begin{proposition}[Path-continuity of $\nu_{\tt BESQ}^{(-2\alpha)}$-IP-evolutions]\label{prop:type-1:cts}
 For $\beta\in\IPspace$, $\bN_{\beta}$ belongs to $\Hfins$ almost surely. %, and the outer measure $\Pr^{1,*}_{\beta}$ associated with $\Pr^1_{\beta}$ restricts to a probability measure on $(\Hfins,\SHfins)$. 
 In particular, $\skewerP(\bN_{\beta})$ is a.s.\ path-continuous in $(\IPspace,\dI)$. Moreover, this process is a.s.\ H\"older-$\theta$ for every $\theta\in\big(0,\alpha/2\big)$, except possibly at time zero.
\end{proposition}

\begin{proof}
 We have already checked properties (i) and (ii) of Definition \ref{def:domain_for_skewer} of $\Hfins$, in Lemma \ref{lem:type-1:wd} and Proposition \ref{prop:type-1:LT_diversity} respectively. It remains only to confirm the claimed path-continuity.
 
 By Lemma \ref{lem:finite_survivors}, for $z > 0$ the process $(\beta^y,\,y\geq z)$ equals the concatenation of an a.s.\ finite subset of the processes $(\beta_U^y,\,y\geq z)$ of \eqref{eq:concat_skewers:type-1}. By Corollary \ref{cor:clade:cts_skewer}, each of the $(\beta_U^y,\,y\ge0)$ is a.s.\ H\"older-$\theta$ for $\theta\in(0,\alpha/2)$. This proves the a.s.\ H\"older continuity of $(\beta^y,\,y\ge z)$, by way of \eqref{eq:IP:concat_dist}. Since this holds for every $z$, it remains only to establish a.s.\ continuity at $y=0$.
 
 Fix $\epsilon > 0$. Take a subset $\{U_1,\ldots,U_k\}\subseteq\beta$ of sufficiently many large blocks so that $\IPmag{\beta} - \sum_{i=1}^k\Leb(U_i) < \epsilon/4$. We define a correspondence by pairing each $U_i$ with the leftmost block in $\beta^y_{U_i}$. Then there is a.s.\ some sufficiently small $\delta>0$ so that for $y<\delta$:
 \begin{enumerate}[label=(\roman*), ref=(\roman*)]
  %each $i$, the leftmost block of the clade associated with $U_i$ stays within mass $\epsilon/4k$ of $\Leb(U_i)$;
  \item for $i\in[k]$, $\big|\Leb(U_i) - m^y(\bN_{U_i})\big| < \epsilon/4k$, where $m^y$ is as in \eqref{eq:LMB_def};
  \item $\big|\IPmag{\beta^0} - \IPmag{\beta^y}\big| < \epsilon/4$;
  \item for $i\in[k]$, $\left|\IPLT_{\beta^0}(U_i) - \sum_{V\in\beta\colon V<U_i}\IPLT_{\beta^y_V}(\infty)\right| < \epsilon$; and
  \item $\big|\IPLT_{\beta^0}(\infty) - \IPLT_{\beta^y}(\infty)\big| < \epsilon$.
 \end{enumerate}
 The third and fourth of these can be controlled via Proposition \ref{prop:cts_lt_at_0}. The first can be controlled since each block $U_i$ is associated with the initial leftmost spindle of $\bN_{U_i}$, and said spindle evolves continuously as a \BESQ[-2\alpha]. Finally, the second comes from Proposition \ref{prop:ex-BESQ_mass_temp}. Hence, $(\beta^y,\,y\ge0)$ is a.s.\ continuous at $y=0$.
\end{proof}

\begin{definition}[$\BPr^{(\alpha)}_{\beta}$, $\BPr^{(\alpha)}_{\mu}$, $(\cFI^y)$]\label{def:IP_process_space_1}
 For $\beta\in\IPspace$, let $\BPr^{(\alpha)}_{\beta}$ denote the distribution on $\cCRI$ of a continuous version of $\skewerP(\bN_{\beta})$. As in Definition \ref{constr:type-1}, for probability measures $\mu$ on $\IPspace$, let $\BPr^{(\alpha)}_{\mu}$ denote the $\mu$-mixture of the laws $(\BPr^1_\beta)$. We write $(\cFI^y,\,y\geq 0)$ to denote the right-continuous filtration generated by the canonical process on $\cC([0,\infty),\IPspace)$. In integrals under the aforementioned laws, we will denote the canonical process by $(\beta^y,\,y\geq 0)$.
\end{definition}

In this setting, Corollary \ref{cor:type-1:simple_Markov_1} extends via a monotone class theorem to the following.

\begin{corollary}[Simple Markov property for $\nu_{\tt BESQ}^{(-2\alpha)}$-IP-evolutions]\label{cor:type-1:simple_Markov}
 $\!\!\!$Let $\mu$ be a probability distribution on $\IPspace$. Fix $y>0$. Take $\eta,f\colon\cCRI\to [0,\infty)$ measurable, with $\eta$ measurable with respect to $\cFI^y$. Let $\theta_y$ denote the shift operator. Then 
  $\BPr^{(\alpha)}_\mu\big[\eta\, f\circ\theta_y \big] = \BPr^{(\alpha)}_{\mu}\left[\eta\, \BPr^{(\alpha)}_{\beta^y}[f]\right].$
\end{corollary}

%%%%%%%%%%%%%%%%%%%%%%%%%%%%%%%%%%%%%%%%%%%%%%%%%%
%%%%%%%%%%%%%%%%%%%%%%%%%%%%%%%%%%%%%%%%%%%%%%%%%%
%%%%                                          %%%%
%%%%   START OF NEW MATERIAL 11 FEB 2019 MW   %%%%
%%%%                                          %%%%
%%%%%%%%%%%%%%%%%%%%%%%%%%%%%%%%%%%%%%%%%%%%%%%%%%
%%%%%%%%%%%%%%%%%%%%%%%%%%%%%%%%%%%%%%%%%%%%%%%%%%

\subsection{$\nu_{q,c}^{(-2\alpha)}$-IP-evolutions}\label{IPevolgeneral} 

In this section, we generalise our results for $\nu_{\tt BESQ}^{(-2\alpha)}$-IP-evolutions to IP-evolutions whose block diffusion is any 
self-similar diffusion on $[0,\infty)$ that is absorbed in 0. As mentioned at the end of Section \ref{sec:BESQ}, with reference to 
\cite{Lamperti72}, such self-similar diffusions form a three-parameter class and can all be obtained from \BESQ\ diffusions by space transformations
of the form $x\mapsto cx^q$. We need $\alpha\in(0,1)$ for the \StableA\ scaffolding. We now also see that the restriction $q>\alpha$ is needed to get 
$\Stable[\alpha/q]$ interval partitions and evolutions in $\IPspace_{\alpha/q}$. We need $c>0$ to preserve positive spindles with absorption in 0. For $\bff\sim\nu_{\tt BESQ}^{(-2\alpha)}(\,\cdot\,|\,\zeta=1)$ we have
\begin{equation}\int_0^1\EV[(\bff(y))^\alpha]dy=\frac{2^\alpha\Gamma(1\!+\!\alpha)}{1+\alpha}\ \Rightarrow\ 
  \int_0^1\EV[(c(\bff(y))^q)^{\alpha/q}]=\frac{c^{\alpha/q}2^\alpha\Gamma(1\!+\!\alpha)}{1+\alpha}
  \label{eqn:cnu}
\end{equation}
and hence $\nu_{q,c}^{(-2\alpha)}=c^{-\alpha/q}\nu_{\tt BESQ}^{(-2\alpha)}(cf^q\in\,\cdot\,)$.

%HERE!!!

We define $\IPspace_{\alpha/q}^{(1/q)}:=\{\beta\in\IPspace_{\alpha/q}\colon\sum_{U\in\beta}({\rm Leb}(U))^{1/q}<\infty\}$ and note that      $\IPspace_{\alpha/q}^{(1/q)}=\IPspace_{\alpha/q}$ for $q\le 1$, but is a strict subset when $q>1$.
Fix $\alpha\in(0,1)$, $q>\alpha$ and $c>0$. Consider the initial interval partition $\gamma\in\IPspace_{\alpha/q}^{(1/q)}$, and let $\beta\in\IPspace_\alpha$ be the 
interval partition obtained from $\gamma$ by transforming all block sizes by $x\mapsto(x/c)^{1/q}$. Let $(\beta^y,y\ge 0)\sim\BPr^{(\alpha)}_\beta$.
For each $y\ge 0$, let $\gamma^y\in\IPspace_{\alpha/q}^{(1/q)}$ be the interval partition obtained from $\beta^{y}$ by transforming all block sizes by 
$x\mapsto cx^q$. Then $(\gamma^y,y\ge0)$ is an $(\alpha,q,c)$-IP-evolution starting from $\gamma\in\IPspace_{\alpha/q}^{(1/q)}$. The operation on block
sizes is naturally carried out spindle by spindle, from a scaffolding-and-spindles construction of $(\beta^y,y\ge 0)$.

Let us show that this restriction of $\gamma$ to $\IPspace_{\alpha/q}^{(1/q)}$ is necessary and not just a feature of the above construction of a 
$\nu_{q,c}^{(-2\alpha)}$-IP-evolution from a $\nu_{\tt BESQ}^{(-2\alpha)}$-IP-evolution. In Lemma \ref{prop:type-1:transn}, we showed that 
$\Pr^{(\alpha)}_{\{(0,a)\}}(\life^+\!>\!z)\!=\!1\!-\!\exp(-a/2z)$, which implies here that $\Pr^{\alpha,q,c}_{(0,b)}(\zeta^+\!>\!z)\!=\!1\!-\!\exp(-(b/c)^{1/q}/2z)$. Hence, 
Lemma \ref{lem:finite_survivors} generalises to yield that a $\nu_{q,c}^{(-2\alpha)}$-IP-evolution starting from $\gamma\!\in\!\IPspace_{\alpha/q}$ has finitely 
many surviving clades if and only if $\sum_{V\in\gamma}({\rm Leb}(V))^{1/q}\!<\!\infty$ and hence summable interval lengths at all levels $z\!>\!0$ if and only if 
$\gamma\!\in\!\IPspace_{\alpha/q}^{(1/q)}$. 

\begin{corollary}\label{cor:genpathcont} $\!\!\nu_{q,c}^{(-2\alpha)}$-IP-evolutions are path-continuous $\IPspace_{\alpha/q}^{(1/q)}$-valued Markov processes, for all $\alpha\in(0,1)$, 
 $q>\alpha$ and $c>0$.
\end{corollary}
\begin{proof} W.l.o.g. $c=1$. The Markov property follows from the construction of the $\nu_{q,c}^{(-2\alpha)}$-IP-evolution by transforming a 
  $\nu_{\tt BESQ}^{(-2\alpha)}$-IP-evolution, which is Markovian by Corollary \ref{cor:type-1:simple_Markov}. This construction also establishes
  that $\nu_{q,c}^{(-2\alpha)}$-IP-evolutions are $\IPspace_{\alpha/q}^{(1/q)}$-valued. It remains to establish path-continuity. For $q>1$, this is a consequence
  of the continuity of the map $r_q\colon\IPspace_\alpha \rightarrow\IPspace_{\alpha/q}^{(1/q)}$ that maps $\beta\in\IPspace_\alpha $ to the interval partition
  $r_q(\beta)$ formed from $\beta$ by transforming all block sizes by $x\mapsto cx^q$.

  For $\alpha<q<1$, we retrace the argument for $\nu_{\tt BESQ}^{(-2\alpha)}$-IP-evolutions concluded in Proposition \ref{prop:type-1:cts}. 
  First, Propositions \ref{prop:agg_mass_subord} and \ref{prop:PRM:cts_skewer} yield that a stopped \PRM[\Leb\otimes\nu_{q,c}^{(-2\alpha)}] gives 
  rise to a $\nu_{q,c}^{(-2\alpha)}$-IP-evolution that is $\theta$-H\"older for all $\theta\in(0,\min\{\alpha/2,q-\alpha\})$ and starting from a 
  \Stable[\alpha/q] initial state. 
  Second, Corollary \ref{cor:clade:cts_skewer} and its proof are easily adapted to show that $\nu_{q,c}^{(-2\alpha)}$-IP-evolutions starting from 
  $\{(0,a)\}$ are also $\theta$-H\"older for all $\theta\in(0,\min\{\alpha/2,q-\alpha\})$. 
  Third, since starting from any $\beta\in\IPspace_{\alpha/q}^{(1/q)}$, only finitely many clades survive beyond level $z>0$, the Markov property and Lemma 
  \ref{lem:IP:concat} show that $(\beta^y,y\ge z)$ is $\theta$-H\"older, too. 
  Fourth, we need to establish path-continuity at $y=0$.    
  
  To establish path-continuity at $y=0$ as in the proof of Proposition \ref{prop:type-1:cts}, we note that the continuity of total diversity 
  (and hence block diversity) at $y=0$ follows as in Proposition \ref{prop:cts_lt_at_0}, $(\alpha,q,c)$-block diffusions are path-continuous, and
  path-continuity of total mass follows as in Proposition \ref{prop:ex-BESQ_mass_temp}.
\end{proof}

We will denote their distributions on $\cC([0,\infty),\IPspace_{\alpha/q}^{(1/q)})$ by $\BPr^{\alpha,q,c}_\beta$, $\beta\in\IPspace_{\alpha/q}^{(1/q)}$.

%%%%%%%%%%%%%%%%%%%%%%%%%%%%%%%%%%%%%%%%%%%%%%%%
%%%%%%%%%%%%%%%%%%%%%%%%%%%%%%%%%%%%%%%%%%%%%%%%
%%%%                                        %%%%
%%%%   END OF NEW MATERIAL 11 FEB 2019 MW   %%%%
%%%%                                        %%%%
%%%%%%%%%%%%%%%%%%%%%%%%%%%%%%%%%%%%%%%%%%%%%%%%
%%%%%%%%%%%%%%%%%%%%%%%%%%%%%%%%%%%%%%%%%%%%%%%%

\subsection{Continuity in the initial state, strong Markov property, proof of Theorem \ref{thm:diffusion}}
\label{sec:Markov}

In this section, we fix $\alpha\in(0,1)$ and $c,q\in(0,\infty)$ and work with $\nu_{q,c}^{(-2\alpha)}$-IP-evolutions. 
As we have seen, they take values in $\IPspace:=\IPspace_{\alpha/q}^{(1/q)}$, equipped with the metric $\dI:=d_{\alpha/q}$.

\begin{proposition}[Continuity in the initial state]\label{prop:type-1:cts_in_init_state}
 For $f\colon\IPspace\to [0,\infty)$ bounded and continuous and $z>0$, the map $\beta\mapsto\BPr^{\alpha,q,c}_{\beta}[f(\beta^z)]$ is continuous on $(\IPspace,\dI)$.
\end{proposition}

\begin{proof} It suffices to prove this for $\nu_{q,1}^{(-2\alpha)}$-IP-evolutions.  
 Fix $z>0$. We will show that for every $\epsilon>0$ and $\beta\in\IPspace$ there is some $\delta>0$ such that for $\gamma\in\IPspace$, $\dI(\beta,\gamma)<\delta$ implies the existence of a pair of $\nu_{q,1}^{(-2\alpha)}$-IP-evolutions $(\beta^y,\,y\geq 0)$ and $(\gamma^y,\,y\geq 0)$ starting from these two initial states, with
 \begin{equation}
  \Pr\{\dI(\beta^{z},\gamma^{z}) \geq 3\epsilon\} < 6\epsilon.\label{eq:level_z_IP_cnvgc_claim}
 \end{equation}
 
 Fix $0 < \epsilon < z$ and $\beta\in\IPspace$. Let $U_1,U_2,\ldots$ denote the blocks of $\beta$, listed in non-increasing order by mass. Let $(\bN_{U_j})_{j\geq 1}$ be as in Definition \ref{constr:type-1}, let $(\beta_{U_j}^y,\,y\geq 0) := \skewerP(\bN_{U_j})$, and set $a_j := \Leb(U_j)$. We take suitable versions so that the process $(\beta^y,\,y\geq 0)$ formed by concatenating the $(\beta_{U_j}^y)$ according to the interval partition order of the $U_j$ in $\beta$, as in \eqref{eq:concat_skewers:type-1}, is a path-continuous $\nu_{q,1}^{(-2\alpha)}$-IP-evolution starting from $\beta$.
 
 We take $L$, $M$, and $K$ sufficiently large and $\delta > 0$ sufficiently small so that setting
% \begin{gather*}
  $E_1 := \big\{\IPLT_{\beta^{z}}(\infty)\!<\!L;\,\IPmag{\beta^{z}}\!<\!M\big\}$, $E_2 := \big\{ \forall j\!>\!K,\zeta^+(\bN_{U_j})\!<\!z\big\}$ \linebreak and $E^j_3 := \left\{ \sup\nolimits_{y \in \left[\left(1-(\delta/a_K)\right)z,\left(1+(\delta/a_K)\right)z\right]}\dI\left(\beta^y_{U_j},\beta^z_{U_j}\right) < \epsilon/K \right\}$, for $j\in [K]$,
% \end{gather*}
 we have $\Pr(E_1)\geq 1-\epsilon$, $\Pr(E_2)\geq 1-\epsilon$, and $\Pr(E^j_3)\geq 1-(\epsilon/K)$ for each $j$. By Lemma \ref{lem:finite_survivors}, it suffices that we take the smallest $K$ large enough that $\sum_{j>K}a_j < 2z\epsilon$. The existence of such a $\delta$ is then guaranteed by the continuity of the $\nu_{q,1}^{(-2\alpha)}$-IP-evolution. We further require
 \begin{equation}\label{eq:partn_SM_1_delta_constraints}
  \delta < \min\left\{a_K,z\epsilon,\frac{\epsilon a_K}{K L}, \frac{\epsilon a_K}{K M}\right\}.
 \end{equation}
 
 Now take $\gamma\in\IPspace$ with $\dI(\beta,\gamma) < \delta$. By definition of $\dI$, there exists a correspondence $(\td U_j,\td V_j)_{j\in [\td K]}$ from $\beta$ to $\gamma$ with distortion less than $\delta$. Since $\delta < a_K$, we get $\td K\geq K$ and
 $\{U_j\}_{j\in [K]} \subseteq \big\{\td U_j\big\}_{j\in [\td K]}$.
 Let $(V_j)_{j\in [K]}$ denote the terms paired with the respective $U_j$ in the correspondence; i.e.\ for each $j\in [K]$, the pair $(U_j,V_j)$ equals $(\td U_i,\td V_i)$ for some $i\in [\td K]$. For $j\in [K]$, let $b_j := \Leb(V_j)$.
 
 We assume w.l.o.g.\ that our probability space is sufficiently large for the following construction %, in the manner of Construction \ref{constr:type-1}, 
 of a $\nu_{q,1}^{(-2\alpha)}$-IP-evolution $(\gamma^y,\,y\geq 0)$ starting from $\gamma$, coupled with $(\beta^y,\,y\ge0)$. For $j\in [K]$, set $\bN_{V_j} := (b_j/a_j)\scaleH\bN_{U_j}$.  We take $(\bN_V,\,V\in\gamma\setminus\{V_j\colon j\in [K]\})$ to be an independent family, independent of $(\bN_U,\,U\in\beta)$, with distributions as in Definition \ref{constr:type-1}. We write $(\gamma_V^y,\,y\geq 0) := \skewerP(\bN_V)$ for each $V\in\gamma$. From Lemma \ref{lem:clade:mass_ker} and the definition of $\scaleH$ in \eqref{eq:clade:xform_def}, we deduce that for $j\in [K]$ and $y\geq 0$,
 \begin{equation*}
  \bN_{V_j} \sim \mClade^+\left(\,\cdot\;\middle|\;m^0 = \Leb(V_j)\right) \qquad \text{and} \qquad
  \gamma_{V_j}^y = \scaleI[\frac{b_j}{a_j}][\beta_{U_j}^{y(a_j/b_j)}].
 \end{equation*}
 Then $(\gamma^y,\,y\geq 0) := \left(\ConcatIL_{V\in \beta}\gamma_V^y,\,y\ge0\right)$ is a $\nu_{q,1}^{(-2\alpha)}$-IP-evolution from $\gamma$.% and is a.s.\ continuous. 

 By Definition \ref{def:IP:metric} of $\dI$ and our choices of $K$ and $\delta$,
 \begin{equation*}
  \IPmag{\gamma} - \sum_{j=1}^K b_j \leq \dI(\beta,\gamma) + \IPmag{\beta} - \sum_{j=1}^K a_j < \delta + 2z\epsilon < 3z\epsilon.
 \end{equation*}
 Thus, by Lemma \ref{lem:finite_survivors}, the event $E_4 := \big\{ \zeta^+(\bN_V) < z\text{ for every }V\in\gamma\setminus \{V_j\colon j\in [K]\}\big\}$ 
% \begin{equation*}
%  E_4 := \big\{ \forall V\in\beta\setminus \{V_j\colon j\in [K]\},\,\zeta^+(\bN_V) < z\big\}
% \end{equation*}%\bigcap_{V\in \beta\setminus \{V_j\colon j\in [K]\}}\big\{\zeta^+(\bN_V) < z\big\}.
 has probability at least $1-3\epsilon$. 
% Note that, unlike $E_1,\,E_2$, and $(E_3^j)_{j\leq K}$, this event is not in the $\sigma$-algebra generated by $(\alpha^y)$.
 On $E_2\cap E_4$, the partition $\beta^z$ is formed by concatenating, in interval partition order, the $\beta^z_{U_j}$, and correspondingly for $\gamma^z$. 
 %$\alpha^z = \ConcatIL_{U\in\alpha\colon U=U_j,\,j\in[K]} \alpha_{U_j}^z$, and correspondingly for $\beta^z$. 
 
 Inequality \eqref{eq:IP:scaling_dist_1} and the last two constraints on $\delta$ in \eqref{eq:partn_SM_1_delta_constraints} imply that on $E_1$,
 \begin{equation*}
  \dI\left(\beta_{U_j}^z,\gamma_{V_j}^{z(b_j/a_j)}\right) \leq \max\left\{\left|\left(\frac{b_j}{a_j}\right)^{\alpha/q} - 1\right|L, \left|\frac{b_j}{a_j}-1\right|M\right\} < \frac{\epsilon}{K}.
 \end{equation*}
 Moreover, \eqref{eq:IP:scaling_dist_2} implies that for each $j$, on $E^3_j\cap E_1$,
 \begin{equation*}
  \dI\!\left(\!\gamma^{z(b_j/a_j)}_{V_j},\gamma^{z}_{V_j}\right) < \max\left\{\frac{b_j}{a_j},\left(\frac{b_j}{a_j}\right)^{\alpha/q}\right\}\frac{\epsilon}{K} < \frac{2\epsilon}{K}, \quad \text{so} \ 
  \dI\left(\beta^{z}_{U_j},\gamma^{z}_{V_j}\right) < \frac{3\epsilon}{K}.
 \end{equation*}
 Finally, by Lemma \ref{lem:IP:concat}, % and \eqref{eq:level_z_IP_cnvgc-finite_sum}, this implies that
 $\dI(\beta^{z},\gamma^{z}) < 3\epsilon$ on $E_1\cap E_2\cap E_4\cap\bigcap_{j=1}^K E_3^j$, 
 and this intersection has probability at least $1-6\epsilon$, as claimed in \eqref{eq:level_z_IP_cnvgc_claim}.
\end{proof}

\begin{corollary}\label{cor:type-1:cts_init_2}
 Take $m\in\BN$, let $f_1,\ldots,f_m\colon\IPspace\rightarrow[0,\infty)$ be bounded and continuous, and take $0\le y_1<\cdots<y_m$. Then $\beta \mapsto \BPr^{\alpha,q,c}_{\beta}\left[\prod_{i=1}^m f_i(\beta^{y_i})\right]$ is continuous.
\end{corollary}

\begin{proof}
 The case $m=1$ is covered by Proposition \ref{prop:type-1:cts_in_init_state}. Assume for induction that for some $m\geq 1$, the assertion holds for all $m$-tuples $(f_1,\ldots,f_m)$ and $y_1 < \cdots < y_m$ as above. Now, fix $0\le y_1<\cdots<y_m<y_{m+1}$ and suppose $f_1,\ldots,f_m,f_{m+1}\colon\IPspace\rightarrow[0,\infty)$ are bounded and continuous. Then by the inductive hypothesis and the continuity of $f_1$, the function
 $$h(\beta) = f_1(\beta)\BPr^{\alpha,q,c}_{\beta}\left[\prod_{i=1}^{m}f_{i+1}(\beta^{y_{i+1}-y_1})\right]$$
 is bounded and continuous. The simple Markov property, noted in Corollaries \ref{cor:type-1:simple_Markov} and \ref{cor:genpathcont}, and Proposition \ref{prop:type-1:cts_in_init_state} applied to $h$ yield that for all sequences $\beta_j\rightarrow\beta$,
 $$
  \BPr^{\alpha,q,c}_{\beta_j}\left[\prod_{i=1}^{m+1}f_i(\beta^{y_i})\right] = \BPr^{\alpha,q,c}_{\beta_j}[h(\beta^{y_1})] \rightarrow
  \BPr^{\alpha,q,c}_{\beta}[h(\beta^{y_1})] = \BPr^{\alpha,q,c}_{\beta}\left[\prod_{i=1}^{m+1}f_i(\beta^{y_i})\right].
 $$
 This proves the continuity of 
 $\beta \mapsto \BPr^{\alpha,q,c}_{\beta}\left[\prod_{i=1}^{m+1}f_i(\beta^{y_i})\right],$ 
 thereby completing the induction.
\end{proof}

\newcommand{\cCRIq}{\mathcal{C}\left([0,\infty),\IPspace\right)}

\begin{proposition}[Strong Markov property]\label{prop:strong_Markov}
 Let $\mu$ be a probability distribution on $\IPspace$. Let $Y$ be an a.s.\ finite stopping time in $(\cFI^y,\,y\ge 0)$. Take $\eta,f\colon\cCRIq\!\to\! [0,\infty)$ measurable, with $\eta$ measurable with respect to $\cFI^Y$. Let $\theta_y$ denote the shift operator. Then 
  $\BPr^{\alpha,q,c}_\mu\big[\eta\, f\circ\theta_Y \big] = \BPr^{\alpha,c,q}_{\mu}\left[\eta\, \BPr^{\alpha,c,q}_{\beta^Y}[f]\right]$.
\end{proposition}

\begin{proof}
 If $Y$ only takes finitely many values, this is implied by the simple Markov property. In general, this follows via a standard discrete approximation of $Y$, as in the proof of \cite[Theorem 19.17]{Kallenberg}, in which we replace the Feller property by Corollary \ref{cor:type-1:cts_init_2}.
\end{proof}

We now prove our second main theorem, establishing that $\nu_{q,c}^{(-2\alpha)}$-IP-evolutions as Hunt processes.

\begin{proof}[Proof of Theorem \ref{thm:diffusion}]\label{page:pf_of_diffusions}
  For Theorem \ref{thm:diffusion}, 
 referring to Sharpe's definition of Borel-right Markov processes and Hunt processes, e.g.\ \cite[Definition A.18]{Li11}, we must check four properties.
 
 (i) The state space $(\IPspace,\dI):=(\IPspace_{\alpha/q}^{(1/q)},d_{\alpha/q})$ must be a Radon space. In fact it follows from Theorem \ref{thm:Lusin} that it is Lusin.
 
 (ii) The semi-groups must be Borel measurable in the initial state. From Proposition \ref{prop:type-1:cts_in_init_state}, they are continuous.
 
 (iii) Sample paths must be right-continuous and quasi-left-continuous. In fact they are continuous, by Proposition \ref{prop:type-1:cts} and Corollary \ref{cor:genpathcont}.
 
 (iv) The processes must be strong Markov under a right-continuous filtration. We have this from Proposition \ref{prop:strong_Markov}.
 
 For self-similarity, recall the construction $\bN_\beta=\Concat_{U\in\beta}\bN_U$ of Definition \ref{constr:type-1}. By Lemma \ref{lem:clade:invariance} and the scaling 
 $m^0(a\scaleH N)=a^qm^0(N)$ that follows from \eqref{eq:clade:xform_def}, \eqref{eq:clade:mass_def} and \eqref{eq:BESQ:scaling_def}, we have $a\scaleH \bN_\beta\sim\Pr_{a^q\scaleI\beta}^{\alpha,q,c}$. Therefore, if $(\beta^y,\,y\ge0) \sim \BPr_{\beta}^{\alpha,q,c}$ then $\big(a^q\scaleI \beta^{y/a},\,y\ge0\big) \sim \BPr_{a^q\scaleI\beta}^{\alpha,q,c}$, as required.
\end{proof}

\subsection{Interval partition evolutions started without diversity}
\label{sec:Hausdorff}

$\!$The construction in Definition \ref{constr:type-1} of $\bN_{\beta} = \ConcatIL_{U\in\beta}\bN_U$, for $\beta\in\IPspace$, can be carried out for 
$\beta\in\HIPspace$ as well. Extending the notation of that definition, let $\Pr^{\alpha,q,c}_{\beta}$ denote the law of the resulting point 
process and $\Pr^{(\alpha)}_{\beta}:=\Pr^{\alpha,1,1}_{\beta}$. The proof of Proposition \ref{prop:clade_lengths_summable} \ref{item:CLS:CLS} and 
\ref{item:CLS:kernel} that $\len(\bN_{\beta}) < \infty$ a.s.\ and $\beta\mapsto\Pr^{(\alpha)}_{\beta}$ is a kernel still holds, without 
modification, in this generality. The same is true of the proofs of results in Section \ref{sec:type-1_gen:cts}, from Lemma 
\ref{lem:finite_survivors} up through Corollary \ref{cor:type-1:simple_Markov_1}. Several of these involve $\IPLT_{\beta^y}(t)$ for $y>0$, but 
none take $y=0$. As in Section \ref{IPevolgeneral}, this extends to $\Pr^{\alpha,q,c}_\beta$ if we restrict to 
$\IPspace_{H}^{(1/q)}:=\{\gamma\in\HIPspace\colon\sum_{U\in\gamma}(\Leb(U))^{1/q}<\infty\}$.  
In particular, we note the extensions of Lemmas \ref{lem:finite_survivors} and \ref{lem:type-1:wd} to this setting.

\begin{lemma}\label{lem:Hausdorff_entrance_I}
 For $\beta\in\IPspace_{H}^{(1/q)}$ and $y>0$, a.s.\ only finitely many of the $(\bN_U,\,U\in\beta)$ survive to level $y$. 
 Moreover, it is a.s.\ the case that for every $y>0$ we have $\skewer(y,\bN_{\beta})\in\IPspace_{\alpha/q}^{(1/q)}$.
\end{lemma}

%Proposition \ref{prop:cts_lt_at_0} deals specifically with the diversity of the interval partition at level $0$. It is used in the proof of Proposition \ref{prop:type-1:cts} to check continuous diversities as the type-1 evolution enters from its initial state. However, without this observation, what remains of the proof of Proposition \ref{prop:type-1:cts} includes a proof of the following.

\begin{corollary}\label{prop:cts_from_Hausdorff}
 Let $\beta\in\IPspace_{H}^{(1/q)}$. Then $\skewerP(\bN_{\beta})$ is H\"older-$\theta$ in $(\IPspace_{\alpha/q}^{(1/q)},d_{\alpha/q})$ on the time interval $(0,\infty)$, for every $\theta\in (0,\min\{\alpha/2,q-\alpha)\}$ a.s.. In particular, $\skewerP(\bN_{\beta})$ is $d_H'$-path-continuous on the time 
 interval $(0,\infty)$.
\end{corollary}
\begin{proof} W.l.o.g.\ $c=1$. When $q=1$, the first part of the proof of Proposition \ref{prop:type-1:cts} applies to show that for $z>0$ the process $(\skewer(y+z,\bN_{\beta}),\,y\ge0)$ has the claimed H\"older continuity on $(\IPspace_\alpha,d_\alpha)$. 
By Definition \ref{def:IP:metric}, this implies continuity in $(\HIPspace,\dH')$. For $q\neq 1$, the spatial transformation of raising values to
their $q$th power shows that $\skewer(y,\bN_\beta)\in\IPspace_{\alpha/q}^{(1/q)}$, so again $(\skewer(y+z,\bN_{\beta}),\,y\ge0)$ has the claimed continuity.
% Then the latter part of the proof, and particularly the bounds on quantities (iii) and (iv), show that $(\skewer(y,\bN_{\beta}),\,y\ge0)$ enters $\dH'$-continuously from $y=0$, where $\dH'$ is as in Definition \ref{def:Hausdorff}. By Proposition \ref{prop:Hausdorff} \ref{item:Haus:equiv}, this yields $\dH$-continuity.
\end{proof}
We believe that $d_H$-path-continuity extends to time 0, but it seems our methods here are not powerful enough to prove this. In the special case
of $\nu_{\tt BESQ}^{(-2\alpha)}$-IP-evolutions, such continuity can be deduced from \cite[Theorem 1.4]{PartB} using the additivity of \BESQ[0].

Where appropriate, we can extend the notation of Definition \ref{def:IP_process_space_1} and Section \ref{IPevolgeneral} to define $\BPr^{(\alpha)}_{\beta}$ and $\BPr^{\alpha,q,c}_\beta$ for $\beta\in\IPspace_{H}^{(1/q)}\setminus\IPspace_{\alpha/q}^{(1/q)}$, to denote the law of a version of $\skewerP(\bN_{\beta})$ that enters $\dH$-continuously and is subsequently $d_{\alpha/q}$-continuous. We call this continuous version a \emph{Hausdorff $\nu_{q,c}^{(-2\alpha)}$-IP-evolution}. In Section \ref{sec:Markov}, %going into the proof of Proposition \ref{prop:type-1:cts_in_init_state}, deleting all references to diversity and $L$
the same coupling argument used to prove Proposition \ref{prop:type-1:cts_in_init_state} also proves the following variant.

\begin{proposition}\label{prop:type-1:cts_in_init_state_H}
 Let $\beta\in\IPspace_{H}^{(1/q)}$. For $f\colon\IPspace_{H}^{(1/q)}\to [0,\infty)$ bounded and continuous and $z>0$, the map $\beta\mapsto\BPr^{\alpha,q,c}_{\beta}[f(\beta^z)]$
 is continuous on $(\IPspace_{H}^{(1/q)},\dH')$.
\end{proposition}

\begin{proof}
 We prove continuity under $d_H'$. 
 We follow the same argument, but omit the definition of $L$ and resulting bound on $\delta$ in \eqref{eq:partn_SM_1_delta_constraints}. So $E_1$ becomes $\{\IPmag{\beta^z}\leq M\}$. Then we make the same coupling to define $(\gamma^y)$ based on $(\beta^y)$. In this setting, applying \eqref{eq:IP:Haus_scale}, the final two displays in the proof become:  for each $j$, on $E^3_j\cap E_1$,
 \begin{equation*}
  \dH'\left(\beta_{U_j}^z,\gamma_{V_j}^{z(b_j/a_j)}\right) \leq \left|\frac{b_j}{a_j}-1\right|M < \frac{\epsilon}{K}, \quad \dH'\left(\gamma^{z(b_j/a_j)}_{V_j},\gamma^{z}_{V_j}\right) < \frac{b_j}{a_j}\frac{\epsilon}{K} < \frac{2\epsilon}{K},
 \end{equation*}
 and so $\dH'\big(\beta^{z}_{U_j},\gamma^{z}_{V_j}\big) < 3\epsilon/K$. Otherwise, the proof is as before.
\end{proof}

This result extends to a Hausdorff variant of Corollary \ref{cor:type-1:cts_init_2}, in the same manner as before, via the simple Markov property. Then the statement of the strong Markov property, Proposition \ref{prop:strong_Markov}, holds for initial distributions $\mu$ on $\IPspace_{H}^{(1/q)}$, via the same standard argument.

\appendix

\section{Excursion intervals}\label{app:exc_intervals}

\begin{definition}\label{def:exc_intervals}
 We define the set of intervals of excursions of $g\in\DS$ about level $y\in\BR$ by
  $$%\label{eq:excursion_intervals}
    V^y(g) := \left\{[a,b]\subset [0,\len(g)]\ \middle|\begin{array}{c}
    		 a\! <\! b\! <\! \infty;\  g(t\minus)\!\neq\! y\!\neq\! g(t)\text{ for }t\in (a,b);\\[.2cm]
    		 g(a\minus)\! =\! y \text{ or }g(a)\! =\! y;\ g(b\minus)\! =\! y \text{ or }g(b)\! =\! y
         	\end{array}\!\!\right\}\!.$$
  We define $V^y_0(g)\supseteq V^y(g)$ to include incomplete first and/or last excursions. In particular, let
  \begin{equation*}
  \begin{split}
   T^y(g)   &:= \inf\left(\left\{t\in [0,\len(g)]\colon g(t)=y \text{ or }g(t-)=y\right\}\cup\{\len(g)\}\right),\\
   T^y_*(g) &:= \sup\left(\left\{t\in [0,\len(g)]\colon g(t)=y \text{ or }g(t-)=y\right\}\cup\{0\}\right)
  \end{split}
  \end{equation*}
  If $y\neq 0$ then we include $[0,T^y(g)]\cap [0,\infty)$ in $V^y_0(g)$. If $T^y_*(g) < \len(g)$ or $g(\len(g)) \neq y$, then we include $[T^y_*(g),\len(g)]\cap [0,\infty)$ in $V^y_0(g)$.
  
  For $[a,b]\in V^y_0(N)$ we define $I^y_N(a,b)$ to equal one of $[a,b]$, $(a,b]$, $[a,b)$, or $(a,b)$, as follows. We exclude the endpoint $a$ from $I^y_N(a,b)$ if and only if both $a<b$ and $g(a-) < y = g(a)$. We exclude $b$ if and only if both $a < b$ and $g(b-) = y < g(b)$. 
\end{definition}

\begin{proposition}\label{prop:exc_intervals}
 It is a.s.\ the case that for every $y\in\BR$, the following properties hold.
 \begin{enumerate}[label=(\roman*), ref=(\roman*)]
  \item $\displaystyle V^y\!\! =\!\! \left\{[a,b]\!\subset\! (0,\infty)\middle|a\!<\!b;\bX(a-)\!=\!y\!=\!\bX(b)\text{;\,and\;}\bX(t)\!\neq\!y\text{\;for\;}t\!\in\!(a,b)\right\}\!.$
  \item For $I,J\in V^y_0$, $I\neq J$, the set $I\cap J$ is either empty or a single shared endpoint.
  \item If two intervals $[a,b],[b,c]\in V^y_0$ share an endpoint $b$ then $\bX$ does not jump at time $b$.
  \item For every $t\notin \bigcup_{I\in V^y_0}I$, we find $\bX(t-) = \bX(t) = y$.
  \item $\displaystyle\Leb\left( [0,\infty)\setminus \bigcup\nolimits_{I\in V^y_0}I\right) = 0$.
 \end{enumerate}
\end{proposition}

\begin{proof}
 (i), (ii), and (iii). These properties follow from a common observation. In the terminology of Bertoin \cite{BertoinLevy}, $0$ is regular for $(-\infty,0)$ and for $(0,\infty)$ after and, by time reversal, before each of the countably many jump times.
 
 (iv). Take $t>T^y$ and set $a = \sup\{s \leq t\colon \bX(s-) = y\}$ and $b = \inf\{s \geq t\colon \bX(s) = y\}$. If $a=b=t$ then, by the \cadlag\ property of $\bX$, we have $\bX(t-)= \bX(t) = y$. Otherwise, by assertion (i), $[a,b]\in V^y$ and $t\in [a,b]$.
 
 (v). This follows from (iv) and the a.s.\ existence of occupation density local time at all levels, per Theorem \ref{thm:Boylan}. Since occupation measure therefore has a derivative in level, it cannot jump at any level.
\end{proof}

\begin{proposition}\label{prop:nice_level}
 For each $y\in\BR$ it is a.s.\ the case that level $y$ is \emph{nice} for $\bX$ in the following sense.
 \begin{enumerate}[label=(\roman*), ref=(\roman*)]
  \item There are no degenerate excursions of $\bX$ about level $y$.
  \item Local times $(\ell^y(t),\,t\ge0)$ exist. For $[a,b],[c,d] \in V^y_0$, $\ell^y(a)\neq \ell^y(c)$ unless $[a,b]=[c,d]$.
  %\item Local times $(\ell^y(t),\,t\ge0)$ exist. For $a<b<c$, if $\bX(b) = y$ or $\bX(b-)=y$ then $\ell^y(a)\neq \ell^y(c)$.
  \item If $y>0$, we also have $T^y>T^{\geq y}:=\inf\{t\ge 0\colon\bX\geq y\}$.
 \end{enumerate}
\end{proposition}

\begin{proof}
 (i) There are four cases of potential degeneracy: start with a jump or creep up from the starting level; end with a jump or creep up to the end level. Millar \cite{Millar73} showed that spectrally positive \StableA\ processes a.s.\ do not creep up to a fixed level. The distributions of pre-jump levels and jump levels are absolutely continuous, so a.s.\ no jump ends at a fixed level. Hence, there is a.s.\ no degeneracy at ends of excursions. By time reversal, the same holds at the start of excursions. 
 
 (ii) Existence of local times has been addressed in Theorem \ref{thm:Boylan}. The Poisson random measure $\bG^y$ of Proposition \ref{thm:excursion_PRM} places all excursions at different local times a.s..
 
 (iii) As noted in (i), there is a.s.\ no creeping up to a level. Hence, $T^y>T^{\geq y}$ a.s..
\end{proof}

\begin{proposition}\label{prop:partn_spindles}
 Take $y\in\BR$ and $N\in\H$. Then for $[a,b]\in V^y(N)$, the process $\shiftrestrict{N}{I^y_N(a,b)}$ is a bi-clade. Moreover, the set $\big\{\shiftrestrict{N}{I^y_N(a,b)}\colon [a,b]\in V_0^y(N)\big\}$ partitions the spindles of $N$, in the sense that for each point $(t,f_t)$ of $N$ there is a unique $[a,b]\in V^y_0(N)$ for which $t\in I_N^y(a,b)$. 
\end{proposition}

Recall concatenation on $\Hfin$ as defined in \eqref{eqn:concat}. To form $\cutoffL{y}{N}$, we concatenate the anti-clades of $N$ below level $y$, along with potentially incomplete anti-clades at the start $[0,T^y(\xi(N)])$ and/or end $[T^y_*(\xi(N)),\len(N)]$, as in Definition \ref{def:exc_intervals}. To formally describe these incomplete anti-clades and the corresponding incomplete clades, we specify their crossing times:
\begin{equation*}
\begin{split}
 					T^{\geq y} &:= \inf\big(\{t\in [0,\len(N)]\colon \xi_N(t)  \geq y\}\cup\{\len(N)\}\big)\\
 \text{and} \quad T^{\geq y}_* &:= \sup\big(\{t\in [0,\len(N)]\colon \xi_N(t-) \leq y\}\cup\{0\}      \big).
\end{split}
\end{equation*}
Note that $T^{\geq y} = T^{\geq y}_*$ if and only if $\xi(N)$ is a single incomplete excursion about level $y$ that neither begins nor ends at $y$. 
To avoid duplication in our formulas, we adopt the convention that in this case, this sole incomplete bi-clade is called the last, and there is no first.
\begin{align*}
 N^{\leq y}_{\textnormal{first}}  &= \Big(\Restrict{N}{[0,T^{\geq y})}
 										+ \cf\big\{\xi_N\left(T^{\geq y}-\right) < y\big\}
 	\Dirac{T^{\geq y},\check f^y_{T^{\geq y}}}\!\Big)\cf\!\left\{T^{\geq y}\neq T^{\geq y}_*\right\}\!,\\[.2cm]
 N^{\geq y}_{\textnormal{first}}  &= \Big(\ShiftRestrict{N}{(T^{\geq y},T^y]}	
 										+ \cf\big\{y\!\neq\! 0;\, \xi_N\!\left(T^{\geq y}\right) > y\!\vee\!0\big\}
 	\delta\big( 0,\hat f^y_{T^{\geq y}}\big)\!\Big)\cf\!\left\{T^{\geq y}\!\neq\! T^{\geq y}_*\right\}\!,\\[.2cm]
 N^{\leq y}_{\textnormal{last}}   &= \ShiftRestrict{N}{[T^y_*,T^{\geq y}_*)}\\
 										&\ \ + \cf\big\{y\!\neq\!\xi_N(\len(N));\xi_N\big(T^{\geq y}_*\!\!-\!\!\big)\!<\big(y\!\wedge\! \xi_N\big(T^{\geq y}_*\big)\big)\big\}
 	\delta\!\left(T^{\geq y}_*\!-\!T^y_*,\check f^y_{T^{\geq y}_*}\right)\!,\\[.2cm]
 N^{\geq y}_{\textnormal{last}}	  &= \ShiftRestrict{N}{(T^{\geq y}_*,\len(N)]}
 										+ \cf\big\{\xi_N \big(T^{\geq y}_*\big) > y\big\}
 	\Dirac{0,\hat f^y_{T^{\geq y}_*}}\!.
\end{align*}
%\!\!\begin{array}{l@{\ :=\ }l@{\;+\;}l}
% N^{\leq y}_{\textnormal{first}}	& \Big(\Restrict{N}{[0,T^{\geq y})}				&\cf\big\{\xi_N\left(T^{\geq y}-\right) < y\big\}
% 	\Dirac{T^{\geq y},\check f^y_{T^{\geq y}}}\!\Big)\cf\!\left\{T^{\geq y}\neq T^{\geq y}_*\right\}\!,\\[.2cm]
% N^{\geq y}_{\textnormal{first}}	& \Big(\ShiftRestrict{N}{(T^{\geq y},T^y]}		&\cf\big\{y\neq 0;\;\xi_N\!\left(T^{\geq y}\right) > (y\vee 0)\big\}
% 	\Dirac{0,\hat f^y_{T^{\geq y}}}\!\Big)\cf\!\left\{T^{\geq y}\neq T^{\geq y}_*\right\}\!,\\[.2cm]
% N^{\leq y}_{\textnormal{last}} 	& \ShiftRestrict{N}{[T^y_*,T^{\geq y}_*)}	&\cf\big\{y\!\neq\!\xi_N(\len(N));\;\xi_N\big(T^{\geq y}_*\!-\!\big)\!<\big(y\wedge \xi_N%\big(T^{\geq y}_*\big)\big)\big\}
% 	\delta\!\left(T^{\geq y}_*\!-T^y_*,\check f^y_{T^{\geq y}_*}\right)\!,\\[.2cm]
% N^{\geq y}_{\textnormal{last}}		& \ShiftRestrict{N}{(T^{\geq y}_*,\len(N)]}	&\cf\big\{\xi_N \big(T^{\geq y}_*\big) > y\big\}
% 	\Dirac{0,\hat f^y_{T^{\geq y}_*}}\!.
%\end{array}
The first bi-clade is complete if and only if $y=0$, in which case $N^{\leq y}_{\textnormal{first}}=N^{\geq y}_{\textnormal{first}} \!=\! 0$. Similarly, the last bi-clade is complete if and only if $y\!=\!\xi_N(\len(N))$, in which case $N^{\leq y}_{\textnormal{last}} = N^{\geq y}_{\textnormal{last}} = 0$.

\begin{supplement}[id=SuppMeas]
 \sname{Supplement A}
 \stitle{Measure theoretic details}
 \slink[doi]{COMPLETED BY THE TYPESETTER}
 \sdatatype{.pdf}
 \sdescription{Technical results and proofs, mainly dealing with measurability, that have been omitted from the main document to aid readability.}
\end{supplement}

\begin{supplement}[id=SuppSim]
 \sname{Supplement B}
 \stitle{Simulation of IP-evolution}
 \slink[doi]{COMPLETED BY THE TYPESETTER}
 \sdatatype{.gif}
 \sdescription{Simulation of a construction and process of the type described in Theorem \ref{thm:diffusion} and Definition \ref{constr:type-1}. Simulation by N.\ Forman, G.\ Brito, D.\ Clancy, M.\ Chacon, R.\ Chou, A.\ Forney, C.\ Li, Z.\ Siddiqui, and N.\ Wynar.}
\end{supplement}

%%%%%%%%%%%%%%%%%%%%%%%%%%%%%%%%%%%%%%%%%%%%%%%%%%%%%%%%%%%%
\bibliographystyle{abbrv}
\bibliography{AldousDiffusion}
%%%%%%%%%%%%%%%%%%%%%%%%%%%%%%%%%%%%%%%%%%%%%%%%%%%%%%%%%%%%
\end{document}